\documentclass[letter, oneside, labels,11pt]{amsart}   	
\usepackage{geometry}                		
\usepackage{amsmath} 
\usepackage{amsthm} 
\usepackage{amssymb}	
\usepackage{graphicx} 

\usepackage{color}

\usepackage{hyperref}	

\usepackage{enumitem}

\newcommand{\NatOne}{{\mathbb N}^*}
\newcommand{\NatZer}{{\mathbb N}}

\newtheorem{theorem}{Theorem}[section]
\newtheorem{lemma}[theorem]{Lemma}
\newtheorem{corollary}[theorem]{Corollary}
\newtheorem{prop}[theorem]{Proposition}
\theoremstyle{definition}

\newtheorem{prb}{Extremal Problem}
\theoremstyle{remark}
\newtheorem*{rmk}{Remark}

\newcommand{\R}{{\mathbb R}}
\newcommand{\B}{{\mathbb B}}
\newcommand{\sph}{{\mathbb S}}
\renewcommand{\div}{\text{$\nabla\cdot$}}
\newcommand{\grad}{\text{$\nabla$}}
\newcommand{\bs}[1]{{\boldsymbol  #1}}
\newcommand{\bb}[1]{{\mathbf  #1}}
\newtheorem{exa}{Example}[section]

\DeclareMathOperator{\supp}{supp}

\title[Divergence-free measures and inverse potential problems]{Divergence-free measures in the plane
and inverse Potential Problems in divergence form}

    \author{L. Baratchart}
   \address{Projet APICS, INRIA, 2004 route des Lucioles, BP 93,
Sophia-Antipolis, 06902 Cedex, FRANCE}
\email{Laurent.Baratchart@inria.fr}
  \author{C. Villalobos Guill\'en}
\address{Department of Mathematics,
Vanderbilt University,
Nashville, TN 37240, USA}
\email{cristobal.villalobos.guillen@protonmail.com}

  \author{D.P.~Hardin}
\address{Department of Mathematics,
Vanderbilt University,
Nashville, TN 37240, USA}
\email{doug.hardin@vanderbilt.edu}

\thanks{This research  was supported, in part, by the U.\ S.\ National Science Foundation under
grant  DMS-1521749.}


\setenumerate[0]{label={\rm (\alph*)}}

\begin{document}

\begin{abstract}
We show that a divergence-free measure on the plane is a continuous sum of unit tangent vector fields on rectifiable Jordan curves.
This \emph{loop} decomposition is more precise than the general decomposition in terms of elementary solenoids given by S.K.  Smirnov
when applied to the planar case.  
The proof rests on a version of the co-area formula for homogeneous BV 
functions,  and on the approximate continuity of measure theoretic connected components of suplevel sets of such functions with respect to  the level.
We apply these results to inverse potential problems
whose source term is the divergence of some unknown (vector-valued) measure; e.g.,   inverse magnetization problems
when magnetizations are modeled by $\R^3$-valued Borel measures. 
We investigate methods  for  recovering a magnetization $\bs{\mu}$  
by penalizing  the measure theoretic total variation norm  $\|\bs{\mu}\|_{TV}$.
In particular,  if a magnetization is supported in a plane, then 
$TV$-regularization schemes always have a unique minimizer, even in the presence of noise.  It is further shown that $TV$-norm minimization (among magnetizations generating the same field) uniquely recovers  planar magnetizations in the following cases: 
when the magnetization is carried by a collection of sufficiently separated line segments and a set that is purely 1-unrectifiable, or when a superset of the support is tree-like.  We note that such magnetizations can be recovered via $TV$-regularization schemes in the zero noise limit by taking the regularization parameter to zero.  This suggests definitions of sparsity in the present infinite dimensional context, that generate results akin to compressed sensing.
\end{abstract}

\maketitle
 
  \section{Introduction}
\label{intro}
This paper deals  with the structure of finite divergence-free
measures  in the  plane, and  applications thereof  to inverse magnetization
problems on thin plates. These are prototypical of
 inverse potential problems with source term in divergence form, and  have 
been 
the main motivation of the authors to develop  a purely measure-geometric
result like Theorem \ref{rep}. The latter asserts that
a planar divergence-free measure can be 
decomposed as a superposition of elementary ``loops'';
{\it i.e.}, unit tangent vector fields on rectifiable Jordan curves. This result is more precise  than the general structure 
theorem for solenoids  given by Smirnov in
\cite{Smi94} (valid in any dimension), and is  hinted at on page 843 of  that reference. Because divergence-free distributions in the plane
are rotations by $\pi/2$ of distributional gradients, one is quickly left to decompose gradients of
``homogeneous'' $BV$-functions; {\it i.e.}, locally integrable functions whose 
partial derivatives are finite measures. 
To do this, we combine a version of the co-area formula for
homogeneous $BV$-functions (Theorem \ref{coarea})
with a decomposition into Jordan curves of the 
measure-theoretic boundary of planar
sets of finite perimeter given in  \cite{Ambetal}. The latter
is a special case of
the decomposition of 1-dimensional  integral currents into 
indecomposable elements \cite[4.2.25]{FedererBook}, in which the 
pattern of orientations has special structure. To handle measurability issues 
in the integral expressing the decomposition of a divergence-free measure
as a superposition of loops, and to relate the tangent fields of the loops to the polar decomposition of the measure, we also 
establish (in any dimension) an approximate continuity property of measure-theoretic 
connected components of suplevel sets
for homogeneous $BV$-functions (Theorem \ref{approxceq}), which is 
interesting in its own right.

The loop decomposition of planar divergence-free measures has interesting 
applications to inverse magnetization problems for thin plates, when
magnetizations are modelled by $\R^3$-valued measures  supported on a set $S$ (in the thin plate case, $S\subset \R^2$).
Then, the inverse magnetization  problem consists 
in recovering such a measure, say $\bs{\mu}$, from knowledge of
the magnetic field $\bb{b}(\bs{\mu})$ that it generates, see Section 
\ref{SPOR} for details.
Magnetizations
supported in a plane generate the zero magnetic field if and only if
they are tangent to that plane and divergence-free there (see Lemma~\ref{2Dsole}). Thus, the 
kernel of the forward operator mapping $\bs{\mu}$ to
$\bb{b}(\bs{\mu})$ consists precisely of planar divergence-free measures 
in this case. The loop decomposition gives insight
on the structure of this kernel, enabling us to give
sufficient conditions for a magnetization to be  {\em $TV$-minimal  on $S$}; {\it i.e.},   the magnetization has  minimum total variation
among those magnetizations supported on $S$ that generate the same field. When a $TV$-minimal magnetization on $S$ is unique among magnetizations generating the same field, we   call it {\em strictly $TV$-minimal 
on $S$}. By standard
regularization theory, strictly $TV$-minimal magnetizations can be recovered by 
solving a sequence of minimization problems for the
so-called regularizing  functional, which
is the sum of the quadratic residuals
and a penalty term consisting of the product of a regularization parameter $\lambda>0$ and the total variation of the 
unknown, see \eqref{defcrit0}. Then, any sequence of  minimizers of the regularizing functional
converges weak-$*$ to the strictly $TV$-minimal measure generating the data
(when it exists), as the regularizing parameter and the noise tend jointly to 
zero in a suitable manner, see {\it e.g.} \cite{BurOsh}. In short: regularizing schemes that penalize the total variation are consistent to recover strictly
$TV$-minimal magnetizations, and thus, 
any assumption ensuring strict $TV$-minimality gives rise to a consistency result. 
For the larger class of magnetizations 
supported
on a slender set $S$ (see Section \ref{SPOR} for a definition), such a consistency result is obtained in \cite[Theorem~2.6]{BVHNS} by showing, using 
reference \cite{Smi94}, that magnetizations supported on a purely 1-unrectifiable set are
strictly $TV$-minimal. Specializing  to the case of planar $S$ and appealing to the loop decomposition 
will allow us  to obtain
more general conditions, proving for instance that magnetizations carried by 
the union of a purely 1-unrectifiable set and a collection of sufficiently separated line segments are strictly $TV$-minimal (Corollary~\ref{Cor5.4}  and Theorem~\ref{Thm6.2}). 

The  results just mentioned are reminiscent
of compressed sensing, where  underdetermined 
systems of linear equations in $\R^n$ are approximately solved
by minimizing the 
residuals while penalizing the $l^1$-norm. This favors
the recovery of sparse 
solutions ({\it i.e.} solutions having
a large number of zero components) when they exist, see {\it e.g.}
\cite{FouRau} and \cite{BreCar}. 
In this connection, the gist of \cite[Theorem~2.6]{BVHNS} and 
its sharpening described above 
for the planar case is to define
notions of ``sparsity'' in the 
present, infinite-dimensional context.
Our results warrant the use of  regularizing schemes that penalize the total variation
(a natural analog of the $l^1$-norm), in order
to recover magnetizations which are sparse according such definitions.

Our second application of the loop decomposition to inverse magnetization 
problems on thin plates is to prove that, for each value of the 
regularization parameter,  the minimizer of the regularizing
functional is unique (Theorem \ref{uniqueness}).
This result is important for algorithmic 
approaches to the inverse magnetization problem, because it tells us that
for every choice of the regularization parameter 
there is a unique estimate of the unknown magnetization based on
the regularization scheme \eqref{crit0}. It is also surprising,
for in the case that a  magnetization is $TV$-minimal, but not strictly $TV$ minimal, one would rather
expect the regularizing  functional to have several minimizers, at least 
for small values 
of the regularizing parameter.  

To conclude this introduction,
let us stress that magnetizations supported in a plane are commonly considered 
in   paleomagnetic studies,  where thin slabs of rock are modeled by
planar regions \cite{IMP,KW,WLFB,LWBHS}.  
It would be interesting to carry over 
the contents of the present paper to more general slender surfaces 
in $\R^3$ than the plane, as the results could apply to other situations 
in geosciences or medical imaging. 
In practice, the development of numerically 
effective algorithms for these inverse problem raises delicate issues of discretization.  Such considerations are not addressed in this paper, but will be taken up in future work.

{\bf Note:} After they completed the present research, the authors became aware of the interesting   manuscript \cite{BG}
where the loop decomposition of planar divergence-free measures is proven by different methods. Still, the result there cannot be substituted right away 
for Theorem 
\ref{rep} of this paper, because we establish specific properties of the 
representing measure $\rho$ in \eqref{eq:intrep} that we use in Section \ref{AIP} 
for the inverse magnetization problem, and cannot be found in \cite{BG}.  We also mention that an earlier version of Theorem 
\ref{rep} already appears in the 2019 thesis \cite{VillPhDThesis}. 

\subsection{Background and Overview of Results}
\label{SPOR}
Let us  first describe the inverse magnetization problem,  which 
serves as a motivation for the results to come.
For a closed subset   $S\subset \R^3$, let $\mathcal{M}(S)$ denote the space of
finite signed Borel measures supported on $S$.   We 
shall use the space $\mathcal{M}(S)^3$ of $\R^3$-valued measures supported on $S$ to model physical magnetizations distributed on $S$ and shall often use ``magnetization on $S$'' interchangeably with ``element of $\mathcal{M}(S)^3$''.   For $\bs{\mu}\in \mathcal{M}(S)^3$, we let $|\bs{\mu}|$ denote the 
{\em total variation measure} of $\bs{\mu}$. The latter is a positive measure, and we put   $\|\bs{\mu}\|_{TV}:= |\bs{\mu}|(\R^3)$ for the {\em total variation} of $\bs{\mu}$, see Section~\ref{notprl}.

The {\em magnetic field} $\bb{b}(\bs{\mu})$ generated by
a magnetization $\bs{\mu}\in \mathcal{M}(S)^3$  is defined, at a point 
$x$ not in the support of 
$\bs{\mu}$,  in terms of the {\em scalar magnetic potential} $\Phi(\bs{\mu})$ by  (see \cite{Jac1975}):
\begin{equation}
\label{bDef1}
\bb{b}(\bs{\mu})(x) =-\mu_0\grad \Phi(\bs{\mu})(x), \qquad x\not\in \text{supp } \bs{\mu},
\end{equation}
where  $\mu_0$ is the {\em magnetic constant} and $\nabla$ indicates the gradient.  Here, $\Phi(\bs{\mu})(x)$ is given by
\begin{equation}\label{PhiDef}
\Phi(\bs{\mu})(x):=\frac{1}{4\pi}\int \grad_y\frac{1}{|x-y|}\cdot d\bs{\mu}(y)=\frac{1}{4\pi}\int \frac{x-y}{|x-y|^3}\cdot d\bs{\mu}(y),
\end{equation} 
where,  for $x, y\in \R^3$, $x\cdot y$ and $|x |$ denote the Euclidean scalar product and norm  and $\grad_y$  the gradient with respect to $y$. 
  Clearly,
$\Phi(\bs{\mu})$ and the components of
$\bb{b}(\bs{\mu})$ are harmonic functions on $\R^3\setminus S$.  
Moreover, formula \eqref{PhiDef} defines $\Phi(\bs{\mu})$ on the whole 
of $\R^3$ as a member of
$L^2(\R^3)+L^1(\R^3)$ (see \cite[Proposition 2.1]{BVHNS})  so that $\bb{b}(\bs{\mu})$, initially defined on $\R^3\setminus S$,  extends
to a $\R^3$-valued divergence-free  distribution on $\R^3$. Indeed,
we may write
\begin{equation} \label{Phibextension}
\Delta \Phi=\div \bs{\mu} \quad  \text{ and } \quad \bb{b}(\bs{\mu})=\mu_0\left(\bs{\mu}-\grad\Phi(\bs{\mu})\right),
\end{equation} 
where $\div \bs{\mu}$ indicates the divergence of $\bs{\mu}$.
%
Note  that \eqref{Phibextension} yields a Helmholtz-Hodge decomposition of 
$\bs{\mu}$, as the sum of a gradient and a divergence-free distribution.
However, neither term is  a measure in general but rather a distribution 
of order $-1$.

The inverse magnetization  problem is to
recover $\bs{\mu}$ from measurements of  $\bb{b}(\bs{\mu})$ taken on a set $Q\subset\R^3 \setminus S$ which, due to the oriented nature  of sensors (coils), 
are usually observed in one direction only, say along 
some 
unit vector $v\in\R^3$. We assume for simplicity that $v$ is the same at
each measurement point. For instance, it is so in usual 
  Scanning Magnetic Microscopy experiments (SMM) where data consist of 
point-wise values of the normal component of the magnetic field on a 
planar region not intersecting $S$, see
\cite{KW,WLFB,LWBHS}. 
Geometric conditions on 
$Q$, $S$ and $v$, ensuring that such measurements  
  suffice to  determine
$\bb{b}(\bs{\mu})$ in the entire region $\R^3\setminus S$, 
are given   in \cite[Lemma 2.3]{BVHNS}, and recalled
for convenience when $S$ is planar in Section \ref{MtoFop}
further below.  In the remainder of this introduction, we assume that these assumptions are satisfied. 

Still, the mapping $\bs{\mu}\to \bb{b}(\bs{\mu})$ is generally not injective,
which is a major  difficulty with this inverse problem. 
In this connection, we say that  $\bs{\mu},\bs{\nu}\in\mathcal{M}(S)^3$ are {\em $S$-equivalent} if  $\bb{b}(\bs{\mu})$ and $\bb{b}(\bs{\nu})$ agree on $\R^3\setminus S$.
A magnetization $\bs{\mu}$ is said to be {\em $S$-silent}   if $\bs{\mu}$ is $S$-equivalent to the zero magnetization;
i.e., if $\bb{b}(\bs{\mu})$ vanishes on $\R^3\setminus S$.   

Since no nonzero harmonic
function lies in $L^2(\R^3)+L^1(\R^3)$,
it follows  from \eqref{Phibextension} that
a divergence-free magnetization 
is $S$-silent.   A partial converse is given in \cite[Theorem~2.2]{BVHNS}, namely
a $S$-silent magnetization is divergence-free
provided that $S$ is  \emph{slender}, meaning  it has Lebesgue measure  
zero and each 
connected component of $\R^3\setminus S$ has infinite Lebesgue measure. 
The slenderness assumption 
is a strong one: for instance it rules out
the case where $S$ is a volumic sample or a closed surface. However, it is satisfied
in important special cases, for example in paleomagnetic studies, 
as mentioned already,  or in Geomagnetism where some regions
of the Earth's crust are assumed to be non-magnetic (or much less magnetic) 
than the
others \cite{Gerhards2015}, or even in Electro-Encephalography where 
sources of primary current are often considered to lie  on the surface of the encephalon
(which is closed and therefore not slender)
but their support should arguably leave out
the brain stem connecting to the spinal cord (therefore the support is
contained in a slender set).

 In \cite{Smi94}, Smirnov   describes divergence-free measures 
 in $\R^n$, also known as {\em solenoids}, in terms of integrals of 
elementary components that are absolutely continuous with respect to 
1-dimensional 
Hausdorff measure $\mathcal{H}^1$.   
Consequently, if $S$ is slender and $\bs{\mu}\in\mathcal{M}(S)^3$  
is such that there is a \emph{purely 1-unrectifiable} set ({\it i.e.}, whose intersection 
with any 1-rectifiable set has $\mathcal{H}^1$- measure zero, 
see \cite{Mattila}) of full $|\bs{\mu}|$ measure, then $\bs{\mu}$
is mutually singular to every $S$-silent magnetization and so has minimum total variation amongst all magnetizations that are $S$-equivalent to $\bs{\mu}$.    
This observation led the authors in \cite{BVHNS}
to consider the following extremal 
problem involving 
 the quantity $M_S(\bs{\mu})$, defined for $\bs{\mu} \in \mathcal{M}(S)^3$ by
\begin{equation*}\label{defmineq}
M_S(\bs{\mu}):=\inf \{ \|\bs{\nu}\|_{TV}\colon \bs{\nu} \text{ is $S$-equivalent to  }{\bs{\mu}}\}. 
\end{equation*}

 \begin{prb}
Given $\bs{\mu}_0\in \mathcal{M}(S)^3$, find $\bs{\mu}$ that is $S$-equivalent to  $\bs{\mu}_0$   satisfying  $$\|\bs{\mu}\|_{TV}=M_S(\bs{\mu}_0).$$
\end{prb}
A solution to Extremal Problem 1 is, by definition, $TV$-minimal on $S$ and is strictly $TV$-minimal on $S$ if this solution is unique.
When  $S\subset \R^3$ is slender and  $\bs{\mu}_0\in\mathcal{M}(S)^3$, we find that $\bs{\mu}_0$ is strictly   $TV$-minimal on $S$ for the  three cases listed below.  Here case (a) is essentially \cite[Theorem~2.6]{BVHNS} and a special case of Theorem~\ref{Thm6.2} to come,  while
(b) is contained in \cite[Theorem~2.11]{BVHNS} and (c) follows from 
Corollary~\ref{Cor5.4} further below. 
\begin{itemize}
\item[(a)]   there is a purely 1-unrectifiable set   of full $|\bs{\mu}_0|$ measure;
\item[(b)] the set $S$ is a  finite disjoint union of compact sets $S_1, \ldots S_k$ and $$\bs{\mu}_0\mathcal{b}_{S_i}={\bf u}_i|\bs{\mu}_0|\mathcal{b}_{S_i},$$
for some set of unit vectors  ${\bf u}_1, \ldots, {\bf u}_k\in \R^3$, in which case we say  $\bs{\mu}_0$ is {\em piecewise unidirectional};
 \item[(c)] $ \bs{\mu}_0$ has a carrier contained in a countable union   of coplanar disjoint line segments $L_k$  such that the distance from any $L_k$ to any $L_j$, $j\ne k$, is greater than or equal to $\mathcal{H}^1(L_k)$.
\end{itemize}
Corollary~\ref{Cor5.4} also implies that (a) can be combined (c),
namely if a measure satisfies (c) and we add to it a measure on $S$ carried by a purely 1-unrectifiable set, 
then we get a measure which is strictly $TV$-minimal again.
Now, for $\rho$ a positive measure  on $Q$, let $A:\mathcal{M}(S)^3\to L^2(Q,\rho)$ be the \emph{forward operator} 
 mapping $\bs{\mu}$ to the restriction of $\bb{b}(\bs{\mu})\cdot v$ on $Q$ (see \eqref{Adef}).  The measure $\rho$ does not  play a significant role in what follows (e.g., it could be chosen to be Lebesgue measure on $Q$), but  it is  important for practical applications. 
To recover solutions of Extremal Problem 1 knowing the restriction
$f$ of $\bb{b}(\bs{\mu}_0)\cdot v$ to $Q$,
the theory of regularization for 
convex 
problems \cite{BurOsh} 
 suggests  to  minimize with respect to $\bs{\mu}\in\mathcal{M}(S)^3$
the    functional 
\begin{equation}
\label{defcrit0}
\mathcal{F}_{f,\lambda}(\bs{\mu}):=
\|f-A\bs{\mu}\|_{L^2(Q,\rho)}^2+\lambda \|\bs{\mu}\|_{TV}
\end{equation}
for some  suitable value of the \emph{regularization parameter} $\lambda>0$.
That is, we consider:
\begin{prb} Given   $f\in L^2(Q)$ and $\lambda>0$, find $\bs{\mu}_\lambda\in\mathcal{M}(S)^3$ such that
\begin{equation}
\label{crit0}
\mathcal{F}_{f,\lambda}(\bs{\mu_\lambda})=\inf_{\bs{\mu}\in\mathcal{M}(S)^3} \mathcal{F}_{f,\lambda}(\bs{\mu}).
\end{equation}
\end{prb}
When $Q$ and $S$ are positively separated, the existence of at least one minimizer  is a consequence of the weak-$*$ compactness of the unit ball in $\mathcal{M}(S)^3$ see {\it e.g.} \cite[Proposition~3.6]{BrePikk}.
Solving Extremal Problem 2 is a particular {\it regularization scheme} 
for the 
Inverse Magnetization Problem,  namely one that penalizes the total 
variation of the unknown.

It is standard that
if $f=A\bs{\mu}_0$ and $\lambda_n\to0$,
then any subsequence of $\bs{\mu}_{\lambda_n}$ has a subsequence 
converging weak-$*$ to a solution of Extremal Problem 1. To account for 
measurement noise, one usually replaces $f$ by $f_n=A\bs{\mu}_0+e_n$, and then
the same result  holds for a sequence $\bs{\mu}_n$ minimizing \eqref{defcrit0} 
with $f=f_n$ and $\lambda=\lambda_n$,
provided that both $\lambda_n$ and $\|e_n\lambda_n^{-1/2}\|_{L^2(Q,\rho)}$ 
tend to $0$, see \cite[Theorems~2\&5]{BurOsh} or
 \cite[Theorems~3.5\&4.4]{HoKaPoSch}.
In particular, if there is a unique solution $\bs{\mu}_0$ of
Extremal Problem 1, then we get weak-$*$ convergence of $\bs{\mu}_{n}$ to $\bs{\mu}_0$. A stronger result, involving weak-$*$ convergence of the total variation 
measure  $|\bs{\mu}_n|$, can be found in \cite[Theorem 4.3]{BVHNS}.
To recap, we have a consistency property  asserting that a magnetization 
meeting a certain assumptions 
({\it e.g.} either (a), (b) or (c) above) can be approximately
recovered {\it via} the regularization scheme \eqref{crit0},
when the noise is small and the regularization parameter $\lambda$ is 
chosen  small but still larger than the square of the noise
(the so-called Morozov discrepancy principle). 
Note that \eqref{crit0} may {\it a priori} have several minimizers,
for the total variation norm is 
not strictly convex  and 
the kernel of $A$ is nontrivial, whence
 the objective function \eqref{defcrit0}
is not strictly convex either as is easy to see.   


In Section \ref{AIP}, we analyze Extremal Problems 1 and 2 further 
in the case where $S$ is contained in a plane. 
We prove that $\bs{\mu}=\bs{\mu}_0$ is the unique solution to Extremal Problem 1 in case (c) listed above (Theorem \ref{Thm6.3}), and also 
that Extremal Problem 2  has a unique solution  for any data
(Theorem \ref{uniqueness}).

Both results depend on Theorem \ref{rep},
asserting that a two-dimensional divergence-free measure $\bs{\nu}$
can be decomposed 
into loops, {\it i.e.} contour integrations along rectifiable Jordan curves,
in such a way that the Radon-Nykodim derivative $d\bs{\nu}/d|\bs{\nu}|(x)$ is
essentially the unit tangent to any of these curves through $x$.
The proof of the latter occupies Section \ref{loppdecSM}, after  some 
preparation in Section \ref{BVdsec} where we recall the co-area formula 
and show approximate continuity of suplevel 
sets of 
homogeneous $BV$-functions. Section \ref{sparse3DSec} describes relevant results from 
\cite{Smi94}, while  Appendix \ref{sec:appendix}  gathers technical facts
connected to the latter.

\subsection{Notation}
\label{notprl}

We conclude this section with some notation and definitions
regarding measures and distributions.  For a vector $x$ in the Euclidean space $\R^n$ (we mainly deal with $n=2$ or 3),  we denote the $j$-th component of $x$ by $x_j$ and   the partial derivative with respect to $x_j$ by $\partial_{x_j}$.  By default, we consider vectors  as column vectors; e.g., for $x\in \R^3$ we write 
 $x=(x_1,x_2,x_3)^T$ where ``$T$'' denotes ``transpose''.
We write $\NatZer$ for the nonnegative integers, $\NatOne$ for the positive integers, and $\R^+$ for the nonnegative real numbers.
We  use bold symbols to represent vector-valued functions and measures, and   the corresponding nonbold symbols with subscripts to denote the respective components; e.g., $\bs{\mu}=(\mu_1, \mu_2, \mu_3)^T$ or $\bb{b}(\bs{\mu})=(b_1(\bs{\mu}),b_2(\bs{\mu}), b_3(\bs{\mu}))^T$. 
For $x\in\R^n$ and $R>0$, we let $\B(x,R)$ indicate the open ball centered at $x$ with radius $R$, and $\sph(x,R)$ the boundary sphere. 
This notation does not show dependence on $n$, but no confusion should arise.
We denote by  $\mathcal{M}(E)$ the space of finite signed 
measures on $E\subset\R^n$. 

 We write $\chi_E$ for the characteristic function of a set $E$ and
$\delta_x$ for the Dirac delta measure at $x$. 
Given a $\R^m$-valued measure in $\bs{\mu}\in\mathcal{M}(\R^n)^m$ and a 
Borel set $E\subset\R^n$, we denote by $\bs{\mu}\mathcal{b}E$ the measure obtained by restricting  $\bs{\mu}$ to $E$ ({\it i.e.} for every Borel set $B\subset\R^n$, $\bs{\mu}\mathcal{b}E(B):=\bs{\mu}(E\cap B))$.

 For $\bs{\mu}\in  \mathcal{M}(\R^n)^m$, the {\em total variation measure} $|\bs{\mu}|$  is defined on Borel sets $B\subset \R^n$ by
\begin{equation}\label{totVarMeasDef}
|\bs{\mu}|(B):=\sup_{\mathcal{P}} \sum_{P\in \mathcal{P}} |\bs{\mu}(P)|, 
\end{equation}
where the supremum is taken over all finite Borel partitions $\mathcal{P}$ of $B$. The {\em total variation norm of $\bs{\mu}$} is then defined as
\begin{equation}\label{TVnormdef}\|\bs{\mu}\|_{TV}:=|\bs{\mu}|(\R^n).\end{equation}  
The support  of  $\bs{\mu}$ ({\it i.e.} the complement of the
largest open set $U$ such that 
$|\bs{\mu}|(U)=0$) is denoted as $\text{\rm supp}\,\bs{\mu}$.
Since $|\bs{\mu}|$ is a Radon measure, the Radon-Nikodym derivative $\bb{u}_{\bs{\mu}}:=d\bs{\mu}/d{|\bs{\mu}|}$ exists as a $\R^m$-valued $|\bs{\mu}|$-integrable function  and it satisfies $|\bb{u}_{\bs{\mu}}|=1$ a.e.  with respect to $|\bs{\mu}|$.

For $\Omega\subset\R^n$ an open set, we denote by $C_c(\Omega,\R^m)$ 
the space of $\R^m$-valued continuous  functions with compact support on 
$\Omega$, equiped with the sup-norm.
When $m=1$, we drop the dependence on $m$  and
simply write $C_c(\Omega)$. A similar notational simplification 
is used for other functional spaces introduced below.

We shall identify $\bs{\mu} \in \mathcal{M}(\R^n)^m$ with the linear form on 
$C_c(\R^n,\R^m)$ given by 
\begin{equation}
\label{defmesf}\langle \bs{\mu},\bb{f}  \rangle := \int \bb{f} \cdot d\bs{\mu}, \qquad \bb{f}\in C_c(\R^n,\R^m).\end{equation}

The norm of the functional
\eqref{defmesf},  is $\|\bs{\mu}\|_{TV}$. More generally, for 
$\Omega\subset\R^n$
an open set, it follows from 
Lusin's theorem \cite[Cor. to Theorem 2.23]{Rudinrca},
applied to the restriction of $\bb{u}_{\bs{\mu}}$ to
``large'' compact sets in $\Omega$,
and from the dominated convergence theorem that 
\begin{equation}
\label{rouv}
|\bs{\mu}|(\Omega)=\sup\{\langle\bs{\mu},\bs{\varphi}\rangle,\,
\bs{\varphi}\in C_c(\Omega,\R^m),\,|\bs{\varphi}|\leq1\}.
\end{equation}
The functional  \eqref{defmesf}
extends naturally with the same norm to the Banach space 
$C_0(\R^n,\R^m)$ of $\R^m$-valued continuous functions on $\R^n$ vanishing at infinity.

At places, we also identify $\bs{\mu}$ with the restriction of \eqref{defmesf}
to $C^\infty_c(\R^n,\R^m)$, 
the space of $C^\infty$-smooth functions with compact support,
equiped with the usual topology of test functions
\cite{Schwartz}.   
We refer to a continuous 
 linear functional on $C^\infty_c(\R^n,\R^m)$ as being
a distribution, and put $\partial_{x_i}$ to mean distributional 
derivative with respect to the variable $x_i$.

We denote     Lebesgue measure on $\R^n$ by $\mathcal{L}_n$ and $d$-dimensional Hausdorff measure by $\mathcal{H}^d$, see \cite{evans_gariepy_2015} for the 
definitions.  We normalize $\mathcal{H}^d$    for $d=1$ and $2$ so that it coincides with  arclength and   surface area for smooth curves and surfaces, 
 and more generally that it agrees with $d$-dimensional volume for nice $d$-dimensional subsets  of $\R^n$.  We denote   the Hausdorff dimension of a set $E$ by $\dim_{\mathcal{H}}(E)$.  
We say that $E\subset R^n$ is $m$-\emph{rectifiable} if it is the countable union of images of Lipschitz functions from $\R^m$ to $\R^n$, up to a set of $\mathcal{H}^m$-measure zero, see \cite[Def. 15.3]{Mattila}.

For $E\subset\R^n$ a measurable set and $1\leq p\leq\infty$, 
we write $L^p(E)$ for the 
familiar Lebesgue space 
of (equivalence classes of $\mathcal{L}_n$-a.e. coinciding) real-valued  
measurable 
functions on $E$ whose $p$-th power is integrable, with norm
$\|g\|_{L^p(E)}=(\int_E|g|^pd\mathcal{L}_n)^{1/p}$ ($\textrm{ess. sup}_E\,|g|$
if $p=\infty$). If $E$ is open, we set $L^1_{loc}(E)$ to consist
of functions $f$ whose restriction $f_{|K}$ to $K$ 
lies in $L^1(K)$, for every compact
$K\subset E$. Since $E=\cup_n K_n$ with $K_n$ compact,
$L^1_{loc}(E)$ is a Fr\'echet space for the distance
$d_1(f,g)=\sum_n2^{-n}\|f-g\|_{L^1(K_n)}/(1+\|f-g\|_{L^1(K_n)})$. 
For $\nu\in\mathcal{M}(\R^n)$ a positive measure different from
$\mathcal{L}_n$, we put 
$L^1[d\nu]$ for the space of real-valued
integrable functions against $\nu$.

We are particularly concerned with magnetizations supported on $\R^2\times \{0\}\subset\R^3$ and hence, with a slight abuse of notation, given  $S\subset \R^2$ and  $\bs{\mu}\in \mathcal{M}(S\times\{0\})^3$, we shall identify $S$ with $S\times\{0\}\subset\R^3$ and $\bs{\mu}$ with $\bs{\mu}\mathcal{b}(\R^2\times\{0\})$.
In addition, we let $\mathfrak{R}$ denote the rotation by $\pi/2$ in $\R^2$; 
i.e., $\mathfrak{R}((x_1,x_2)^T)=(-x_2,x_1)^T$.

For an open set $\Omega\subset \R^n$, recall the space $BV(\Omega)$ of functions of {\em bounded variation} comprised of  functions in $L^1(\Omega)$
whose distributional derivatives are signed measures on $\Omega$ (see, \cite{Ziemer}).  We let 
$BV_{loc}(\Omega)$ denote the space of functions whose restriction to
any relatively compact open subset $\Omega_1$ of $\Omega$ lies in $BV(\Omega_1)$.
We define the space $\dot{BV}(\Omega)$ of ``homogeneous'' BV-functions
to consist of locally integrable functions whose distributional derivatives are finite signed measures on $\Omega$.
 Note that $\phi\in\dot{BV}(\Omega)$ if and only if it is a distribution on $\Omega$ such that 
$\grad\phi\in \mathcal{M}(\Omega)^n$, by \cite[Theorem 6.7.7]{Dem2012}.
If $\phi\in \dot{BV}(\Omega)$, we see  from  \eqref{rouv}
by mollification that
\begin{equation}
\label{TVviadiv}
\|\grad\phi\|_{TV}=\sup_{\varphi\in C_c^1(\Omega,\R^n),|\varphi|\leq1}\int
\varphi\cdot d(\grad\phi)=\sup_{\varphi\in C_c^1(\Omega,\R^n),|\varphi|\leq1}\int
\phi \,\div\varphi \,d\mathcal{L}_2,
\end{equation}
where $C_c^1(\Omega,\R^n)$ denotes the space of $\R^n$-valued continuously differentiable functions with compact support in $\Omega$, see \cite[Ch. 5]{evans_gariepy_2015}.


\section{Divergence-free measures on $\R^n$}
\label{sparse3DSec}
We recall in this section the 
decomposition of divergence-free measures into
elementary components obtained in \cite{Smi94}.
We also point at  additional properties of the elementary components,
the proofs of which are appended
in Appendix \ref{sec:appendix} to streamline 
the exposition.
\subsection{Curves as measures}
\label{sec:cam}
For $a<b$ two real numbers, we call a Lipschitz mapping  $\bs{\gamma}:[a,b]\to \R^n$  a {\em  parametrized rectifiable curve}, while the image $\Gamma:=\bs{\gamma}([a,b])$ 
is simply termed a  (non-parametrized) \emph{rectifiable curve}.
 By   Rademacher's Theorem (see \cite{evans_gariepy_2015}),   $\bs{\gamma}$ is differentiable a.e. on $[a,b]$. 
Note that $\bs{\gamma}$ needs not be injective, {\it i.e.} the curve needs not be simple. If we let $N(\bs{\gamma},x)$  
be the 
cardinality (finite or infinite) of the preimage $\bs{\gamma}^{-1}(x)$, 
then the length $\ell(\bs{\gamma})$ of $\bs{\gamma}$ is  
\begin{equation}
\label{areas}
\ell(\bs{\gamma}):=\int_a^b |\bs{\gamma}^\prime(t)|\,dt=\int N(\bs{\gamma},x) \,d\mathcal{H}^1(x),
\end{equation} 
where the second equality follows from the area formula 
\cite[3.2.3]{FedererBook}.
In particular, $\mathcal{H}^1(\Gamma)<\infty$ and
$\mathcal{H}^1$-almost every $x\in\Gamma$ is attained only
finitely many times by $\bs{\gamma}$. 
Observe that $\ell(\bs{\gamma})\neq\mathcal{H}^1(\Gamma)$ in general.
When
$|\bs{\gamma}'(t)|=1$ a.e. on $[a,b]$,
we call $\bs{\gamma}$  
a unit speed parametrization.
This means that $\bs\gamma$ parametrizes $\Gamma$ (non injectively perhaps)
by  percursed arclength. 

If $\bs\gamma$ is injective on $[a,b)$ 
and  $\bs\gamma(a)=\bs\gamma(b)$, 
we say that $\bs\gamma$ is a parametrized rectifiable Jordan curve 
and $\Gamma$ a rectifiable Jordan curve;
 in this case $\ell(\bs{\gamma})=\mathcal{H}^1(\Gamma)$. 
Given a Jordan curve $\Upsilon$ ({\it i.e.} the image of a circle 
by an injective continuous map)
 such that $\mathcal{H}^1(\Upsilon)<\infty$, 
one can easily construct
a unit speed parametrization $\bs\gamma:[0,\mathcal{H}^1(\Upsilon)]\to \Upsilon$ which is
injective on $[0,\mathcal{H}^1(\Upsilon))$ with 
$\bs{\gamma}(0)=\bs{\gamma}(\mathcal{H}^1(\Upsilon))$.
Thus, a Jordan curve $\Upsilon$ is rectifiable if and only if 
$\mathcal{H}^1(\Upsilon)<\infty$. 

For $\bs\gamma:[a,b]\to\R^n$  a parametrized rectifiable curve,
we define $\mathbf{R}_{\bs{\gamma}}\in \mathcal{M}(\R^n)^n$ by
\begin{equation}\label{Rgamma}
\langle \mathbf{R}_{\bs{\gamma}},\mathbf{g}\rangle 
:=\int_a^b \mathbf{g}(\bs{\gamma}(t))\cdot\bs{\gamma}'(t) dt=\int_{\Gamma}\left( \sum_{t\in\bs{\gamma}^{-1}(x)}\bb{g}(x)\cdot\bs{\gamma}^\prime(t) \right)d\mathcal{H}^1(x),\qquad
\mathbf{g}\in C_0(\R^n)^n,
\end{equation}
where the second equality follows from the area formula.
Clearly, $ \mathbf{R}_{\bs{\gamma}}$ is supported on $\Gamma$ and
$\|\mathbf{R}_{\bs{\gamma}}\|_{TV}\leq \ell(\bs{\gamma})$.
If we define $\psi:[a,b]\to[0,\ell(\bs{\gamma})]$  by
$\psi(t)=\int_a^t|\bs{\gamma}^\prime(\tau)|d\tau$, then
$\psi$ is Lipschitz with $\psi^\prime(t)=|\bs{\gamma}^\prime(t)|$ a.e. 
and there is  a unit speed parametrization
$\widetilde{\bs{\gamma}}:[0,\ell(\bs{\gamma})]\to \Gamma$ such that
$\bs{\gamma}=\widetilde{\bs{\gamma}}\circ\psi$, by the chain rule 
and Sard's theorem for Lipschitz functions (see \cite[Theorem~7.4]{Mattila}).
Moreover, we see from the area formula that 
$\mathbf{R}_{\bs{\gamma}}=\mathbf{R}_{\widetilde{\bs{\gamma}}}$,  so
we assume unless otherwise stated  that parametrized rectifiable curves are unit speed 
parametrizations.  

By Lemma \ref{RNcurve}, 
$\mathbf{R}_{\bs{\gamma}}$ is absolutely continuous with 
respect to $\mathcal{H}^1\mathcal{b}\Gamma$ and has 
Radon-Nykodim derivative 
$d\mathbf{R}_{\bs{\gamma}}/d(\mathcal{H}^1\mathcal{b}\Gamma)(x)=\sum_{t\in\bs{\gamma}^{-1}(x)} \bs{\gamma}^\prime(t)$ at
$\mathcal{H}^1$-a.e. $x\in\bs{\Gamma}$. Hence,
for every Borel set $B\subset\R^n$, we have that
\begin{equation}\label{Rgammappv}
\mathbf{R}_{\bs{\gamma}}(B) 
=
\int_{\Gamma\cap B}\left( \sum_{t\in\bs{\gamma}^{-1}(x)}
\bs{\gamma}^\prime(t) \right)d \mathcal{H}^1(x),\qquad
|\mathbf{R}_{\bs{\gamma}}|(B) 
=
\int_{\Gamma\cap B}\left| \sum_{t\in\bs{\gamma}^{-1}(x)}
\bs{\gamma}^\prime(t) \right|d \mathcal{H}^1(x).
\end{equation}

It may happen that $\|\bb{R}_{\bs\gamma}\|_{TV}<\ell(\bs{\gamma})$,
because cancellation can occur
in \eqref{Rgamma}.
To  discard such cases,
we consider for each $\ell>0$ the collection $\mathcal{C}_\ell$
of those $\mathbf{R}_{\bs{\gamma}}$ associated to
a parametrized rectifiable curve $\bs\gamma$ of length $\ell$ that satisfy
$\|\bb{R}_{\bs\gamma}\|_{TV}=\ell$.
By Lemma \ref{sCell}, we have that  
$\mathbf{R}_{\bs{\gamma}}\in\mathcal{C}_\ell$ if and only if 
$\Gamma$ has a well  
defined (oriented) unit tangent $\bs{\tau}(x)$ at $\mathcal{H}^1$-a.e. $x$,
given  by $\bs{\gamma}^\prime(t)$ for any $t$ such that $\bs{\gamma}(t)=x$.
In this case, we note that \eqref{Rgammappv} 
can be rewritten as
\begin{equation}\label{Rgammas}
\mathbf{R}_{\bs{\gamma}}(B) 
=
\int_{\Gamma\cap B}N(\bs{\gamma},x)\bs{\tau}(x)\,d \mathcal{H}^1(x),\qquad
|\mathbf{R}_{\bs{\gamma}}|(B) 
=
\int_{\Gamma\cap B}N(\bs{\gamma},x)\,d \mathcal{H}^1(x).
\end{equation}

\subsection{Decomposition of solenoids into curves}
\label{subsec:curves}
Since $\mathcal{M}(\R^n)^n$ is dual to $C_c(\R^n,\R^n)$ which is separable,
the closed ball $\mathcal{B}_\ell\subset \mathcal{M}(\R^n)^n$
centered at $0$ of radius $\ell$  is
a compact metrizable space  for 
the weak-$*$ topology. In particular,  
$\mathcal{C}_\ell$ equipped with the weak-$*$ 
topology is a (non complete) metric space.
Now, suppose that $\bs{\mu}\in\mathcal{M}(\R^n)^n$ is a solenoid,
{\it i.e.} that $\div\bs\mu=0$ 
(as a distribution).  Then, it follows from  \cite[Theorem A]{Smi94} that 
$\bs{\mu}$ can be decomposed into elements from $\mathcal{C}_\ell$, meaning there is a positive  finite Borel
measure $\rho$ on $\mathcal{C}_\ell$  such that, for $\rho$-a.e. $\bs{\gamma}$,
the measure $\bb{R}_{\bs\gamma}$ is
supported in $\textrm{supp}\,\bs{\mu}$  and
\begin{equation}\label{Smiw}
\langle\bs{\mu},\bb{g}\rangle=\int_{\mathcal{C}_\ell} \langle
\mathbf{R}_{\bs\gamma},\bb{g}\rangle d\rho(\mathbf{R}_{\bs\gamma}),
\qquad
\langle|\bs{\mu}|,\varphi\rangle=\int_{\mathcal{C}_\ell} \langle
|\mathbf{R}_{\bs\gamma}|,\varphi\rangle d\rho(\mathbf{R}_{\bs\gamma}),
\end{equation}
for all $\bb{g}\in  C_c^\infty(\R^n,\R^n)$ and $\varphi\in  C_c^\infty(\R^n)$. Of course, by mollification, it is clear that
\eqref{Smiw}  more generally holds for 
 $\bb{g}\in  C_c(\R^n,\R^n)$ and $\varphi\in  C_c(\R^n)$. 
By Lemma \ref{foncset}, the two equalities in 
Equation \eqref{Smiw} amount to say that,
for each Borel set $B\subset\R^n$,
\begin{equation}\label{Smi1}
\bs{\mu}(B)=\int_{\mathcal{C}_\ell} \mathbf{R}_{\bs\gamma}(B)\ d\rho(\mathbf{R}_{\bs\gamma}),
\qquad
|\bs{\mu}|(B)=\int_{\mathcal{C}_\ell} |\mathbf{R}_{\bs\gamma}|(B)\ d\rho(\mathbf{R}_{\bs\gamma}).
\end{equation}
We note that \eqref{Smi1} was used 
in the proof of \cite[Theorem 2.6]{BVHNS} without further
justification.

The representation \eqref{Smiw}
is far from unique: for instance $\ell>0$ was arbitrary.
Moreover, the
$\mathbf{R}_{\bs\gamma}$ need not be divergence-free 
even though $\bs{\mu}$ is; {\it i.e.}, the solenoid $\bs{\mu}$
gets decomposed {\it via} \eqref{Smiw} into elementary components 
$\mathbf{R}_{\bs\gamma}$ that may not be solenoids.
In this connection, observe that $\div \mathbf{R}_{\bs\gamma}=
\delta_{\bs{\gamma}(b)}-\delta_{\bs{\gamma}(a)}$ which 
vanishes if only if $\bs{\gamma}$ is a closed parametrized curve.
In the next section, 
we  discuss a more subtle  decomposition of $\bs{\mu}$, this time
into divergence-free components, which is established in  
\cite[Theorem B]{Smi94}. In a sense, it is obtained
by letting $\ell\to\infty$ in \eqref{Smiw}.

\subsection{Decomposition of solenoids into elementary solenoids}
\label{sec:decss}
In the terminology of \cite{Smi94}, an \emph{elementary solenoid} $\bb{T}_{\bb{f}}$ is 
a $\R^n$-valued measure associated to a Lipschitz function
$\bb{f}:\R\to\R^n$ with $|\bb{f}'(t)|\leq1$,
acting on $\bs{\varphi}\in C_c(\R^n)^n$  by the formula:
\begin{equation}
\label{SmirnovA}
\bb{T}_{\bb{f}}(\bs{\varphi})=\lim_{s\to+\infty}\frac{1}{2s}\int_{-s}^s \bs{\varphi}(\bb{f}(t))\cdot \bb{f}'(t)\,dt,
\end{equation}
where the existence of the limit is \emph{assumed} for every $\bs{\varphi}$
(for instance, it will exist
if $\bb{f}$ is periodic or quasi-periodic). In addition,
it is required that $\bb{f}(\R)\subset \mbox{\rm supp}\,\bb{T}_{\bb{f}}$ 
 and that
$\|\bb{T}_{\bb{f}}\|_{TV}=1$. 
Letting  $\bb{f}_s:=\bb{f}|_{[-s,s]}$, we get with the notation of 
Section 
\ref{sec:cam} that  $\bb{T}=*\lim\,\mathbf{R}_{\bb{f}_s}/(2s)$ as 
$s\to+\infty$,
where $*\lim$ indicates the weak-$*$ limit. 
It is clear from \eqref{SmirnovA} that
$\mbox{\rm supp}\,\bb{T}_{\bb{f}}\subset \overline{\bb{f}(\R)}$, therefore 
the condition  that $\bb{f}(\R)\subset \mbox{\rm supp}\,\bb{T}_{\bb{f}}$ 
really means  that $\mbox{\rm supp}\,\bb{T}_{\bb{f}}=\overline{\bb{f}(\R)}$.
By Lemma \ref{unitels}, we may assume without loss of generality
that $|\bb{f}'(t)|=1$ a.e. on $\R$ in the definition of $\bb{T}_{\bb{f}}$.
It is straightforward to check
that 
$\div\bb{T}_{\bb{f}}=0$, since  
$\bb{T}_{\bb{f}}(\grad\Psi)=\lim_s(\Psi(\bb{f}(s)-\Psi(\bb{f}(-s))/s=0$
for any $\Psi\in C_c^1(\R^n)$.
Hence, $\bb{T}_{\bb{f}}$ is indeed a solenoid.
We denote by  $\mathfrak{S}(\R^n)$ the set of elementary 
solenoids on $\R^n$. Since it is contained in
$\mathcal{B}_1$, the set 
$\mathfrak{S}(\R^n)$ is a metric space when 
endowed with the weak-$*$ topology. 

It is more difficult to describe members of
$\mathfrak{S}(\R^n)$ 
than members of $\mathcal{C}_\ell$, but still their
structure is reminiscent
of \eqref{Rgammappv} as we now indicate.
Indeed, putting $\Gamma_s=\bb{f}([-s,s])$ and  $N(\bb{f},x,s)$  for
the cardinality (finite or infinite)
of those $t\in[-s,s]$ such that $\bb{f}(t)=x$,  let us define the
\emph{normalized arclength}  of the parametrization $\bb{f}_s:[-s,s]\to\R^n$ 
to be the measure on $\R^n$ given by 
\begin{equation}
\label{prob}
 d\nu_s(x):=\frac{N(\bb{f},x,s)}{2s}d(\mathcal{H}^1\mathcal{b}\Gamma_s)(x). 
\end{equation}
From \eqref{areas}, we see that $\nu_s$ is a probability measure for each 
$s>0$, and by Lemma \ref{desces},
the family  $(\bs{\nu}_s)_{s>0}$ converges 
weak-$*$, as $s\to+\infty$, to the probability measure $|\bb{T}_{\bb{f}}|$.
Moreover, the Radon Nykodim derivative $\bb{u}_{\bb{T}_{\bb{f}}}$ extrapolates,
in a sense made precise in that lemma, a limit of 
averaged tangents to $\bb{f}(\R)$. 
For instance, if $\bb{g}_k$ is a sequence in $C_c(\R^n)$  such that $|\bb{g}_k|\leq1$ and
$\lim_k\bb{g}_k(x)= \bb{u}_{\bb{T}_{\bb{f}}}(x)$ for $|\bb{T}_{\bb{f}}|$-a.e.
$x\in\R^n$ (such a sequence exists by Lusin's theorem), then to any 
real sequence $s_k\to+\infty$ there is
a subsequence $s_{j(k)}$ such that (compare \eqref{limtan}):
\[\lim_{k\to\infty}\int\left|\bb{g}_k(x)-\frac{\sum_{t\in\bb{f}^{-1}(x),\,|t|\leq s_{j(k)}}\bb{f}^\prime(t) }{N(\bb{f},x,s_{j(k)})}\right|^2d\bs{\nu}_s(x)
=0.
\]
 A typical example is obtained when $\bb{f}$ is a line
winding on a torus with irrational slope. Then $|\bb{T}_{\bb{f}}|$
is the normalized area measure and  $\bb{u}_{\bb{T}_{\bb{f}}}$ is a continuous tangential vector field on the torus.

It is shown in \cite[Theorem B]{Smi94} that each 
$\bs{\mu}\in\mathcal{M}(\R^k)^k$ with  $\div \bs{\mu}=0$ can be 
expressed as
\begin{equation}
\label{Smirdecp}
\langle\bs{\mu},\bs{\varphi}\rangle=\int_{\mathfrak{S}(\R^k)} 
\langle \bb{T},\bs{\varphi}\rangle \,d\rho(\bb{T}),\qquad\bs{\varphi}\in C_c(\R^n,\R^n),
\end{equation}
for some positive Borel measure $\rho=\rho(\bs{\mu})$ on $\mathfrak{S}(\R^k)$,
in such a way that
\begin{equation}
\label{Smirdeca}
\langle|\bs{\mu}|,\bs{\varphi}\rangle=\int_{\mathfrak{S}(\R^k)} 
\langle|\bb{T}|,\bs{\varphi}\rangle\,d\rho(\bb{T}).
\end{equation}
 Arguing as in 
Lemma \ref{foncset}, one sees  that  \eqref{Smirdecp} and \eqref{Smirdeca}
together are equivalent to
\begin{equation}
\label{measonb}
\bs{\mu}(B)=\int_{\mathfrak{S}(\R^k)} \bb{T}(B)\,d\rho(\bb{T}),
\qquad |\bs{\mu}|(B)=\int_{\mathfrak{S}(\R^k)} |\bb{T}|(B)\,d\rho(\bb{T})
\end{equation}
for every Borel set $B$, in particular  
$\text{supp}\,\bb{T} \subset \text{supp}\,\bs{\mu}$ for $\rho$-a.e. 
$\bb{T}\in\mathfrak{S}(\R^k)$.
In \cite{Smi94}, the relations \eqref{Smirdecp} and \eqref{Smirdeca} are 
summarized 
by saying that a divergence-free measure can be completely decomposed into
elementary solenoids.

In dimension 3 already, the functions $\bb{f}$ 
giving rise to a well-defined measure $\bb{T}_{\bb{f}}$ {\it via}
\eqref{SmirnovA} can have rather complex 
behaviour, see examples in \cite[Sec. 1.3]{Smi94}. However, in dimension 2, 
the decomposition \eqref{Smirdecp} can be achieved using periodic
$\bb{f}$ parametrizing rectifiable Jordan curves:
this follows from Theorem \ref{rep} in Section \ref{loppdecSM}.
In this connection,
we note that if $\bb{f}:\R\to\R^n$
satisfies $|\bb{f}^\prime|=1$ a.e. and
is periodic of period $L>0$, then the limit in \eqref{SmirnovA}
does exist and in fact $\bb{T}_{\bb{f}}=\bb{R}_{\bs{\gamma}}/L$,
where $\bs{\gamma}:[0,L]\to\R^n$
is the restriction $\bb{f}_{|[0,L]}$. Clearly then, we have that
$\text{supp}\,\bb{T}_{\bb{f}}=\bs{\gamma}([0,L])=\bb{f}(\R)$,
and in order that
$\bb{T}_{\bb{f}}$  be an elementary solenoid it is necessary and sufficient 
that $\|\bb{T}_{\bb{f}}\|_{TV}=1$. This amounts 
to require that  $\|\bb{R}_{\bs{\gamma}}\|_{TV}=L$ or, equivalently,
that $\bb{R}_{\bs{\gamma}}\in \mathcal{C}_L$.
By the discussion after \eqref{Rgammappv},
this is the case when
$\bs{\gamma}([0,L])$ is a rectifiable Jordan curve.

$\R^3$-valued solenoids with planar  support are of particular significance for our applications.  The following elementary lemma, essentially contained in \cite{IMP}, gives  simple characterizations of such solenoids.  We include a proof for the convenience of the reader. Recall the definition of $\dot{BV}$
and the notation $\mathfrak{R}$ for the rotation by $\pi/2$ in $\R^2$.
\begin{lemma} 
\label{2Dsole}
Let $S\subset\R^2\times \{0\}$ be closed,  $\bs{\mu}=(\mu_1,\mu_2,\mu_3)^T\in \mathcal{M}(S)^3$, and $\bs{\mu}_T=(\mu_1,\mu_2)^T$.
The following are equivalent: 
\begin{enumerate}
 \item $\div \bs{\mu}=0$ in the distributional sense on $\R^3$. 
\item $\mu_3=0$ and $\div \bs{\mu}_T=0$ in the distributional sense on $\R^2$.
\item $\mu_3=0$ and $\bs{\mu}_T=\mathfrak{R}\grad \phi =(-\partial_{x_2}\phi ,\partial_{x_1}\phi)^T$ for some $\phi\in \dot{BV}(\R^2)$. 
\end{enumerate}
%
\end{lemma}
\begin{proof}
%
%
%
Since $\bs{\mu}$ has support contained in $\R^2\times\{0\}$,
 it  can be written in tensor product form as $\bs{\mu}= (\bs{\mu}\mathcal{b}\R^2)\otimes\delta_{x_3=0}$ and thus $\div\bs\mu=(\div\bs \mu_T)\otimes\delta_{x_3=0} +  \mu_3\otimes\delta_{x_3=0}'$, where $\delta_{x_3=0}$ is the Dirac 
mass at zero on $\R$ in the variable $x_3$ and $\delta_{x_3=0}'$ its distributional derivative. 
Hence, (b) implies that $\div\bs\mu=0$ and therefore
(b)$\Rightarrow$(a).
Next, for any $\phi\in C_c^\infty(\R^3)$, let  $\phi_0,\phi_1\in
C_c^\infty(\R^2)$ be given by $\phi_0(x_1,x_2)=\phi(x_1,x_2,0)$ and 
$\phi_1(x_1,x_2)=\partial_{x_3}\phi(x_1,x_2,0)$. Then,
it holds that
\begin{equation}\label{lem2pf1}
\langle \div \bs{\mu}, \phi\rangle = -\langle \mu_1, \partial_{x_1}\phi_0\rangle -\langle \mu_2, \partial_{x_2}\phi_0\rangle- \langle    {\mu_3},\phi_1 \rangle.
\end{equation}
Pick $\phi$ of the form $\phi(x_1,x_2,x_3)=\psi(x_1,x_2)\eta(x_3)$
where $\psi\in C_c^\infty(\R^2)$ and   $\eta\in C_c^\infty(\R)$.
First, letting $\eta$ be such that $\eta(0)=1$ and $\eta'(0)=0$, we deduce 
from \eqref{lem2pf1}  that if $\div\bs\mu=0$
then $\div \bs{\mu}_T=0$. Second, letting $\eta$ be such that $\eta(0)=0$ and $\eta'(0)=1$, we deduce 
from \eqref{lem2pf1}  again that if $\div\bs\mu=0$ then $\mu_3=0$, whence 
(a)$\Rightarrow$(b).

 Suppose now that (b) holds.  
Then $(-\mu_2,\mu_1)^T$ satisfies the Schwartz rule when viewed as a $\R^2$ valued distribution on $\R^2$; i.e, $\partial_{x_2}(-\mu_2)=\partial_{x_1}\mu_1$.
  Therefore, $\mathfrak{R}\,\bs{\mu}_T=(-\mu_2,\mu_1)^T$ is the gradient of a scalar valued distribution $\Psi$ (see, \cite{Schwartz}). Since  the components of $\grad\Psi$ are finite signed measures, $\Psi\in BV_{\text{loc}}$ \cite[Theorem 6.7.7]{Dem2012} so that in fact $\Psi\in \dot{BV}(\R^2)$. Thus,
(c) holds with $\phi=-\Psi$ and we get that (b)$\Rightarrow$(c).
 In the other direction if   
$\bs{\mu}_T=(-\partial_{x_2}\phi ,\partial_{x_1}\phi)^T$ for some 
distribution $\phi$, then $\div \bs{\mu}_T=- \partial_{x_1}\partial_{x_2}\phi+\partial_{x_2}\partial_{x_1}\phi =0$ so that (c)$\Rightarrow$(b). 
\end{proof}

Lemma \ref{2Dsole} entails
that decomposing
solenoids in the plane is equivalent, up to a rotation,
to decomposing gradients. 
As surmised in \cite{Smi94}, the latter can be achieved {\it via} 
the co-area formula and the 
decomposition of the measure-theoretical boundary of sets of finite 
perimeter in $\R^2$ into rectifiable Jordan curves.  
In Section \ref{BVdsec} to come, we record a version of the
co-area formula for  $\dot{BV}$-functions,
and we  establish approximate continuity of $M$-connected components 
of sup-level sets of  
such functions (see Proposition~\ref{convmcc} and Theorem ~\ref{approxceq}).
The latter is needed to handle 
measurability issues in 
the loop decomposition of planar divergence-free measures 
(see Proposition \ref{mesi}), but is also of independent interest.
Though we later lean on the planar case, 
it would be artificial to restrict to $\R^2$ in Section \ref{BVdsec} and
we shall present the material  in $\R^n$.

\section{Sup-level sets of functions in $\dot{BV}(\R^n)$ and the co-area formula}
\label{BVdsec}

We first record a summability property of homogeneous $BV$-functions.

\begin{lemma}\cite[Theorem 3.47]{AFP}\label{poin}
If $\phi\in\dot{BV}(\R^n)$, there is $p\in\R$ such that $\phi-p\in L^{n/(n-1)}(\R^n)$.
\end{lemma}

Next, we collect several definitions and properties
that are central to what follows. 
For  $E\subset \R^n$ a Borel set, the {\em measure-theoretical boundary} of $E$ is the set $\partial_M E$  defined by
\begin{equation}
\label{defthb}
\partial_M E := \left\{ x\in\R^n : \limsup_{\rho\rightarrow 0}\frac{\mathcal L_n(\B(x,\rho)\cap E)}{\mathcal L_n(\B(x,\rho))} > 0 \text{ and } \limsup_{\rho\rightarrow 0}\frac{\mathcal L_n(\B(x,\rho)\setminus E)}{\mathcal L_n(\B(x,\rho))} > 0 \right\}.
\end{equation}

Note that for any set $E$, $\partial_ME$ is a subset of the topological boundary of $E$.

A measurable set $E\subset \R^n$ such that $\grad \chi_E\in \mathcal{M}(\R^n)^n$  is said to be {\em of finite perimeter}\footnote{In \cite{Attetal2006,evans_gariepy_2015,Ziemer}, the definition is that $\chi_E\in BV(\R^n)$. The present definition means that
$\chi_E\in \dot{BV}(\R^n)$ and, in view of Lemma \ref{poin}, 
amounts to requiring that either $\chi_E$ or $\chi_{\R^n\setminus E}$ lies in $ BV(\R^n)$.}.
For such a set it holds that
\begin{equation}
\label{GGma}
|\grad \chi_E | = \mathcal{H}^{n-1}\mathcal{b}\partial_ME,
\end{equation} 
and $\|\grad \chi_E \|_{TV}=\mathcal{H}^{n-1}(\partial_ME)$ is called the perimeter of $E$, denoted as $\mathcal{P}(E)$.
The identity (\ref{GGma}) can be obtained by 
combining \cite[Theorem 5.15 (iii)]{evans_gariepy_2015}, saying
that (\ref{GGma}) holds when $\partial_ME$ is replaced by the so-called 
reduced boundary of $E$, with \cite[Lemma 5.5]{evans_gariepy_2015},
asserting that $\partial_M E$ differs from the reduced boundary by a set of $\mathcal{H}^{n-1}$-measure zero (see also \cite[Theorem 10.3.2]{Attetal2006}).

It follows from \eqref{GGma} that a set of finite perimeter has a measure-theoretical
boundary  of finite $\mathcal{H}^{n-1}$-measure. In contrast,
its Euclidean boundary can be much larger and
even have  positive $\mathcal{L}_n$-measure, as the following example shows.

\begin{exa}\label{cutout}
	
	Let $E_1=\overline{\B}(0,1)\subset\R^2$ and $\{q_j\}_{j\in\NatZer}$ 
enumerate  all points in $E_1$ with rational coordinates.
	Having defined inductively a closed set $E_n$ for $n\geq1$, 
let $j_n$ 
be the smallest integer such that $q_{j_n}$ lies interior to
$E_n$ and $B_n$  the
largest open ball centered at $q_{j_n}$ contained in $E_n$, with radius $r_n\leq2^{-n}$ ( at some steps $B_n$ could be empty).
	Then, define $E_{n+1}=E_n\setminus B_n$ which must be 
a closed set with nonempty interior, otherwise a finite 
union of balls of total $\mathcal{L}_2$-measure less than $\pi/3$ would cover
$\B(0,1)$. 
Hence, the process can continue indefinitely,
and we let $E=\bigcap E_n$ which is a closed set.

	Clearly $E$ has no interior, for all the $q_j$ have been excised out
in the process; therefore its Euclidean boundary is $E$ itself. Moreover,
$\mathcal{L}_2(E)\geq\pi-\pi\sum_{n=1}^\infty r_n^2
\geq\pi(1-\sum_{n=1}^\infty4^{-n})>0$.
	
	Now, by the standard Green formula, 
each $E_n$ is of finite perimeter, because it is a finitely
connected set with piecewise smooth boundary. Thus, 
	$\{\chi_{E_n}\}$ is a nonincreasing sequence of $BV$-functions and
 their point-wise limit $\chi_E$ is integrable.
	Also, by \eqref{GGma}, it holds that $\|\grad\chi_{E_n}\|_{TV}\leq2\pi\sum_{n=0}^\infty r_n\leq 4\pi$, therefore we can use \cite[Remark 5.2.2]{Ziemer} to the effect that $\chi_E\in BV(\R^2)$, i.e. $E$ is a set of finite perimeter with Euclidean boundary of positive $\mathcal{L}_2$-measure, as announced.
\end{exa}

For any $E\subset\R^n$ of finite perimeter,
one defines the {\em generalized unit inner normal} 
$\bs{\nu}_E$ to $\partial_ME$ as the Radon-Nikodym derivative  
$\bb{u}_{\grad\chi_E}$ which is but
$d\grad \chi_E/d(\mathcal{H}^{n-1}\mathcal{b}\partial_ME)$, by \eqref{GGma}.   Then, the Radon Nikodym Theorem entails the following version of the
Gauss-Green formula: \\
	{\em if $E\subset \R^n$ is a set of finite perimeter, then 
for each Borel set $B\subset\R^n$ it holds that
	\begin{equation}
	\label{dermeas}
	\grad \chi_E(B) = \int_{B} \bs{\nu}_E \, d\left(\mathcal{H}^{n-1}\mathcal{b}\partial_ME\right). 
	\end{equation}
	 }

The connection with the classical Gauss-Green formula is  more transparent
on the distributional version of (\ref{dermeas}), namely:
\begin{equation}
\label{dermeasd}
\int \chi_E\,\mbox\div{\bs \varphi} \,d\mathcal{L}_n = -\int {\bs \varphi}\cdot\bs{\nu}_E \, d\left(\mathcal{H}^{n-1}\mathcal{b}\partial_ME\right), \quad{\bs \varphi} \in C_c^1(\R^n,\R^n).
\end{equation}
The identity \eqref{dermeasd} was proven in 
\cite{DeGiorgi1,DeGiorgi2} and \cite{Federer1,Federer2}; see also \cite[Theorem~5.16]{evans_gariepy_2015} and \cite[Theorem~10.3.2]{Attetal2006}.
Note that if $E$ has finite perimeter, then so does $\R^n\setminus E$ and
$\bs{\nu}_{\R^n\setminus E}=-\bs{\nu}_E$. 

\begin{rmk}
When $n=2$, we see from \eqref{dermeasd} that $\bs\nu_E$ coincides with 
the usual, differential-geometric  inner unit normal to the boundary of 
$E$
when the latter is a rectifiable Jordan curve, for in this case the Gauss-Greenformula is valid for both definitions of the normal
(see \cite[Theorem 10--43]{Apostol1957} for a suitable version of the Gauss-Green formula here).
Actually, Lemma \ref{Jordan_decomposition} entails that the 
measure-theoretical boundary of any planar set of finite perimeter is comprised of a countable union of rectifiable Jordan curves, up to a set of $\mathcal{H}_1$-measure zero.
Thus, both notions of inner unit normal coincide
$\mathcal{H}^1$-a.e. on the 
measure-theoretical 
 boundary of such a set. 
\end{rmk}

Given a function  $\phi\in L^1_{loc}(\R^n)$,
we consistently denote with $E_t$ the suplevel sets:
\begin{equation}
\label{defsl}
E_t:=\{ x\in \R^n\mid \phi(x)>t\}.
\end{equation} 
Of course,
the set $E_t$, as well as a number of subsequent sets in $\R^n$ that we will consider, is defined up to a set of $\mathcal{L}_n$-measure zero only, but 
which representative is chosen will be irrelevant for our purposes. 
Hereafter, we abbreviate the sentence ``up to a set of $\mathcal{L}_n$-measure zero'' by ``mod-$\mathcal{L}_n$'', and similarly for $\mathcal{H}^{n-1}$.
The sup-level sets are the main components   of  the co-area 
(or Fleming-Rishel) formula for 
$\dot{BV}$-functions, 
of which we record a version in Lemma~\ref{coarea} below; see, {\it e.g.}
\cite[Theorem 3.40]{AFP}.

\begin{lemma}\label{coarea}
	Suppose $\phi\in \dot{BV}(\R^n)$
and let $E_t$ be as
	in (\ref{defsl}).
	Then, $E_t$ has finite perimeter
for a.e. $t\in\R$ and for any Borel set $B\subset\R^n$, $\bb{g}\in L^1[d|\grad\phi|]^n$ and $h\in L^1[d|\grad\phi|] $, it holds that
	\begin{enumerate}
		\item 
		$\displaystyle |\grad \phi |(B) = \int_{-\infty}^{\infty} |\grad \chi_{E_t} |(B) \, dt=
		\int_{-\infty}^{\infty} \mathcal{H}^{n-1}(\partial_ME_t\cap B) \, dt
		, $
			\item   $\displaystyle \int h\, d(|\grad \phi|) = \int_{-\infty}^{\infty}\int h \, d(|\grad \chi_{E_t}|) \, dt = \int_{-\infty}^{\infty}\int h \, d\left(\mathcal{H}^{n-1}\mathcal{b}\partial_ME_t\right) \, dt ,$
	
		\item $\displaystyle \grad \phi (B) = \int_{-\infty}^{\infty} \grad \chi_{E_t} (B) \, dt= \int_{-\infty}^{\infty}\int_{B} \bs{\nu}_{E_t} \, d\left(\mathcal{H}^{n-1}\mathcal{b}\partial_ME_t\right) \, dt, $
		
		\item $\displaystyle \int \bb{g} \cdot d(\grad \phi) = \int_{-\infty}^{\infty}\int \bb{g}\cdot d(\grad \chi_{E_t}) \, dt = \int_{-\infty}^{\infty}\int \bb{g}\cdot \bs{\nu}_{E_t} \, d\left(\mathcal{H}^{n-1}\mathcal{b}\partial_ME_t\right) \, dt $,
	\end{enumerate}
where in (b) the function $h$ lies in
both $L^1[d|\nabla\chi_{E_t}|]$ and  $L^1[d\mathcal{H}^{n-1}\mathcal{b}\partial_ME_t]$ for a.e. $t$ and in (d) the functions ${\bf g}$ and ${\bf g}\cdot\bs{\nu}_{E_t}$
lie in $L^1[d|\nabla\chi_{E_t}|]^n$ and  $L^1[d\mathcal{H}^{n-1}\mathcal{b}\partial_ME_t]$, respectively, for a.e. $t$.  
\end{lemma}
In fact, that $E_t$ has finite perimeter for a.e. $t$  and that
(a) and (c) hold follows from \cite[Theorem 3.40]{AFP} and \eqref{dermeas}.
Then, (a) and (c) respectively yield (b) and (d) for simple functions, 
and the general case follows by dominated convergence,
using (a) to ascertain that a Borel set $B$ such that $|\nabla\phi|(B)=0$
has $|\nabla\chi_{E_t}|(B)=0$ and $\mathcal{H}^{n-1}\mathcal{b}\partial_ME_t(B)=0$ for a.e. $t$.
One can  also describe the 
``measure theoretical discontinuities'' of $\dot{BV}$-functions as follows.
For $\Omega\subset\R^n$ an open set and a $\mathcal{L}_n$-measurable
$f:\Omega\to\R$,
define for $x\in\R^n$ (see \cite[Def. 5.8, 5.9]{evans_gariepy_2015}):
\begin{gather}
f^{\sup}(x):=\text{ap}\limsup_{y\to x}f(y)=
\inf \left\{
t\left|\lim_{r\to0}\frac{\mathcal{L}_n (\B(x,r) \cap \{\phi>t\})}{\mathcal{L}_n (\B(x,r))}=0\right.\right\},\notag\\
f^{\inf}(x):=\text{ap}\liminf_{y\to x}f(y)=
\sup \left\{
t\left|\lim_{r\to0}\frac{\mathcal{L}_n (\B(x,r) \cap \{\phi<t\})}{\mathcal{L}_n (\B(x,r))}=0\right.\right\}\label{jump}\\
\text{and } J(f):=\left\{x\left| f^{\inf}(x) < f^{\sup}(x) \right.\right\}.\notag
\end{gather}
\begin{lemma}\label{jump_lemma}
	Given $\phi\in\dot{BV}(\R^n)$, the set $J(\phi)$ is $(n-1)$-rectifiable.
	Furthermore, $\grad\phi\mathcal{b}J(\phi)$ is absolutely continuous with respect to $\mathcal{H}^{n-1}$  and, with $E_t$ as in \eqref{defsl},
its Radon-Nykodim derivative satisfies for a.e. $t\in\R$ and $\mathcal{H}^{n-1}$-a.e. $x\in\partial_ME_t\cap J$:
$d{\grad\phi}/d\mathcal{H}^{n-1}=(\phi^{\sup}-\phi^{\inf})\bs{\nu}_{E_t}$.
\end{lemma}
\begin{proof}
	Clearly, it is enough that the result holds for the restriction 
$\phi_{|B}$ of $\phi$ to an arbitrary open ball $B\subset \R^n$, with $E_t$ replaced 
by $E_t\cap B$ and $J(\phi)$ by $J(\phi)\cap B$. 
Since $\phi_{|B}\in BV(B)$, the conclusion now follows from 
\cite[Remark 10.3.4, Theorem 10.4.1]{Attetal2006}, see also
\cite[Theorem 3.78]{AFP}. 
\end{proof}


A set $E\subset\R^n$ with finite perimeter is called
\emph{indecomposable} if it cannot be partitioned as $E=F_1\cup F_2$ with 
$\mathcal{L}_n(F_i)>0$ for $i=1,2$ and
$\mathcal{P}(F_1)+\mathcal{P}(F_2)=\mathcal{P}(E)$. Every set $E$ of finite 
perimeter can be partitioned as a countable union $\cup_iC_i$, where the $C_i$ 
are indecomposable with $\mathcal{L}_n(C_i)>0$ for each $i$ and
$\sum_i\mathcal{P}(C_i)=\mathcal{P}(E)$. Such a partition is 
unique mod-$\mathcal{L}_n$, and the $C_i$ are called \emph{the $M$-connected components of $E$}; moreover, if $F\subset E$ and $F$ is indecomposable, then $F\subset C_i$  mod-$\mathcal{L}_n$ for some $i$,  see
\cite[Theorem~1]{Ambetal}. 

We shall enumerate the $M$-connected components of a set $E$
with $\mathcal{P}(E)<\infty$ so that their
$\mathcal{L}_n$-measures are nonincreasing;
of course, several orderings with this property  exist
if distinct components have the same measure, which will eventually lead us 
to regard  $M$-connected components as equivalence classes. In this connection,
if $E$ has finitely many $M$-connected components, 
we find it convenient to append to them a countable infinity of spurious components 
having $\mathcal{L}_n$-measure zero (therefore also zero perimeter).
This will allow us to consistently 
index the $M$-connected components over $\NatZer$, regardless whether 
the set under consideration has finitely many nontrivial components or not.

Formally, let $\mathcal{S}$ be the set of sequences $(F_i)_{i\in\NatZer}$
of subsets of $\R^n$ mod-$\mathcal{L}_n$ such that 
$\mathcal{L}_n(F_i)\geq \mathcal{L}_n(F_{i+1})$
and $\lim_i\mathcal{L}_n(F_i)=0$. 
We say that two 
elements  $(F_i)_{i\in\NatZer}$, 
$(F_i^\prime)_{i\in\NatZer}$ of $\mathcal{S}$ are equivalent if
there is bijection 
$\sigma:\NatZer \to\NatZer$ such that
$F_{\sigma(i)}=F^\prime_i$ for all $i$ mod $\mathcal{L}_n$. 
We denote by $\dot{\mathcal{S}}$ the set of equivalence classes.
For $E$ a set of finite perimeter and $C_0, C_1,C_2,\cdots$ a list  of
its $M$-connected components, arranged so that $\mathcal{L}_n(C_i)\geq\mathcal{L}_n(C_{i+1}$, we consider $(C_i)_{i\in\NatZer}$ as (a representative of) 
an element of $\dot{\mathcal{S}}$.
If  $\mathcal{L}_n(E)<\infty$, then clearly $\mathcal{L}_n(C_i)<\infty$ 
for all $i$,  and if 
$\mathcal{L}_n(E)=\infty$, then $C_0$  is the only component with infinite 
$\mathcal{L}_n$-measure \cite[Rem. 1]{Ambetal}. 
In particular, since $\sum_i\mathcal{P}(C_i)=\mathcal{P}(E)$,
we have indeed that $\lim_i\mathcal{L}_n(C_i)=0$, by
the isoperimetric inequality 
(see {\it e.g.} \cite[Theorem~5.11]{evans_gariepy_2015}). Of course, 
$(C_i)_{i\in\NatZer}$ is a rather special element of $\mathcal{S}$,
because the $C_i$ are pairwise disjoint mod-$\mathcal{L}_n$ and the
$\partial_M C_i$ are pairwise disjoint mod-$\mathcal{H}^{n-1}$
(see \cite[Proposition~3]{Ambetal}).

We now recall  an extremal property of $M$-connected components. 
Fix $\alpha\in(1,n/(n-1))$ and, for
any measurable set $F\subset\R^n$,
set $G(F):=(\int_{F}e^{-|x|^2}dx)^{1/\alpha}$. 
If $E$ has finite perimeter,
then its $M$-connected components  are  the unique 
solution of
\begin{equation}
\label{maxcompc1}
\max\left\{\sum_{i\in\NatZer}G(F_i):\ 
(F_i)_{i\in\NatZer}\in\dot{\mathcal{S}},\ 
\text{\rm the $F_i$ partition }\ E,
\  \sum_{i\in\NatZer}\mathcal{P}(F_i)\leq\mathcal{P}(E)
\right\},
\end{equation}
see the proof of \cite[Theorem~1]{Ambetal}.

We will also need  the notion of local convergence in measure
for sets of finite perimeter, which is just the 
$L^1_{loc}$-convergence of their characteristic function. 
Any sequence of sets with uniformly bounded perimeters has a subsequence
converging locally in measure, and the perimeter is lower semi-continuous
for this type of convergence,
see {\it e.g.} \cite[Proposition~3.6 \& Theorem~3.7]{MirandaJr}.

\begin{prop}
\label{convmcc}
Let $\phi\in\dot{BV}(\R^n)$ and $E_t$ be as in \eqref{defsl}. For $t$ such that
$E_t$ has finite perimeter, let $(C_0^t, C_1^t,C_2^t,\cdots)\in\mathcal{S}$
be (a representative of) the $M$-connected components of $E_t$.
To each $\eta>0$, there is a $\sigma$-compact set $\Sigma_\eta\subset\R$, with
$\mathcal{L}_1(\R\setminus\Sigma_\eta)<\eta$, having the following properties.
\begin{enumerate}
\item[(i)] For each $t\in\Sigma_\eta$, it holds that $E_t$ has finite perimeter.
\item[(ii)] If $(t_m)_{m\geq1}$ is a sequence in $\Sigma_\eta$ converging to 
$t_0\in\Sigma_\eta$, there is a subsequence $t_{m_j}$ such that
$C_i^{t_{m_j}}$ converges locally in measure, for fixed  $i$
as $j\to\infty$, 
to a set $F_i\subset\R^n$ of finite perimeter, and
the sequence $(F_0,F_1,F_2,\cdots)$ is equivalent to 
$(C_0^{t_0},C_1^{t_0},C_2^{t_0},\cdots)$
in $\dot{\mathcal{S}}$.
\item[(iii)] 
it holds that 
$\lim_j\mathcal{L}_n((C^{t_{m_j}}_i\setminus  F_i)\cup(F_i\setminus C^{t_{m_j}}_i))=0$
and  $\lim_j\mathcal{P}(C^{t_{m_j}}_i)=\mathcal{P}(F_i)$ for each $i$.
\item[(iv)] One has the limiting relations:
\begin{equation}
\label{estqcc}
\lim_{p\to\infty}\,\,\limsup_j\,\,\sum_{i\geq p}\mathcal{P}_n(C_i^{t_{m_j}})=0,\quad\mathrm{and}\quad\lim_{p\to\infty}\,\,\limsup_j\,\,\sup_{i\geq p}\mathcal{L}_n(C_i^{t_{m_j}})=0.
\end{equation}
\end{enumerate}
\end{prop}
\begin{proof}
By Lemma \ref{poin}, we may assume
that $\phi\in L^{n/(n-1)}(\R^n)$.
For $t\in\R$, let us define
$M(t):=\lim_{\epsilon\to0}
\mathcal{L}_n(\{x:t-\epsilon<\phi(x)\leq t+\epsilon\})$.
If we fix $k\in\NatOne$, every finite sequence $t_1,\cdots, t_\ell$ with
$1/k<t_1<t_2<\cdots<t_\ell$ is such that
$\sum M(t_i)\leq k^{n/(n-1)}\|\phi\|_{L^{n/(n-1)}(\R^n)}^{n/(n-1)}$. Hence, the set of $t>0$ such that
$M(t)>0$ is at most countable, and the same holds for $t<0$. Let
$N\subset\R$ be a countable set with $0\in N$ such that $M(t)=0$ for 
$t\notin N$.
Let further $Z\subset\R$ be a Borel set of 
measure zero such that
$E_t$ has finite perimeter for $t\notin Z$, see Lemma \ref{coarea}.
It follows from Lemma
\ref{coarea} (a) that the map $t\mapsto\mathcal{P}(E_t)$ 
is integrable on $\R$ and  therefore, by Lusin's theorem
and the regularity of $\mathcal{L}_1$,
for each relative integer $k\in\mathbb{Z}$ we can find
a compact set $K_k\subset (k,k+1)$, with $K_k\cap (Z\cup N)=\emptyset$
and $\mathcal{L}_1((k,k+1)\setminus K_k)<3\eta/(2\pi^2(1+|k|)^2)$,
such that $t\mapsto\mathcal{P}(E_t)$ is continuous $K_k\to\R$.
Define $\Lambda_\eta:=\cup_{k\in\mathbb{Z}}K_k$, and observe that it is a $\sigma$-compact set 
such that $\mathcal{L}_1(\R\setminus\Lambda_\eta)<\eta/2$ 
and the restriction of $t\mapsto\mathcal{P}(E_t)$ to $\Lambda_\eta$ is 
continuous. 

Let $N_1$ denote the norm of
$t\mapsto \mathcal{P}(E_t)$ in $L^1(\R)$, and set
 $\Sigma_\eta:=\{t\in\Lambda_\eta, \mathcal{P}(E_t)\leq 2N_1/\eta\}$. 
By construction, $\Sigma_\eta$ is $\sigma$-compact
and $\mathcal{L}_1(\R\setminus\Sigma_\eta)<\eta$.
Note that $\Sigma_\eta\cap Z=\emptyset$, therefore $(i)$ holds.

Now, let $t_m\to t_0$ in $\Sigma_\eta$. 
As $t_0\neq0$ (for $0\notin\Sigma_\eta$) and 
$\phi\in L^{n/(n-1)}(\R^n)$, either $t_0>0$ in which 
case $\mathcal{L}_n(E_{t_0})<\infty$,  or else $t_0<0$
in which case $\mathcal{L}_n(E_{t_0})=\infty$.
In the former (resp. latter) case, we may assume that $t_m>0$ (resp. $t_m<0$), and then $\mathcal{L}_n(E_{t_m})<\infty$ (resp.
$\mathcal{L}_n(E_{t_m})=\infty$) for all $m$.
By the boundednes of $t_m\mapsto \mathcal{P}(E_{t_m})=\sum_i\mathcal{P}(C_i^{t_m})$ 
(since $t\mapsto \mathcal{P}(E_t)$ is bounded on
$\Sigma_\eta$ by construction),
we get that $\mathcal{P}(C_i^{t_m})$ is bounded independently 
of $i$ and $m$,  hence for each $i$ some subsequence $C_i^{t^{(i)}_{m_j}}$ converges locally 
in measure to a set $F_i$ of finite perimeter.
Using a diagonal argument, we may assume that $t_{m_j}^{(i)}=t_{m_j}$ is independent of $i$, and that $C_i^{t_{m_j}}$ converges locally in measure to $F_i$
for each $i\geq0$. Next, recall the definition of 
$G$ given before \eqref{maxcompc1}  and observe that the argument of \cite[proof of Eqn. (12)]{Ambetal}
applies with minor modifications to yield
\begin{equation}
\label{cvdr1}
\lim_{p\to\infty}\limsup_j\sum_{i=p}^\infty G(C_i^{t_{m_j}})=0.
\end{equation}
Note also that  $G(C_i^{t_{m_j}})\to G(F_i)$ for fixed $i$ as $m\to\infty$,
because $x\mapsto e^{-|x|^2}$ is summable and so a 3-$\varepsilon$
argument reduces the issue to $L^1_{loc}$-convergence of $e^{-|x|^2}\chi_{C_i^{t_{m_j}}}(x)$ to $e^{-|x|^2}\chi_{F_i}(x)$, which follows from 
local convergence in  measure of $C_i^{t_{m_j}}$ to $F_i$.  
Now, by \eqref{cvdr1}, for every $\epsilon>0$ there is a $p>0$ such that $\limsup_j\sum_{i=p}^\infty G(C_i^{t_{m_j}})<\epsilon$.
Thus
\begin{align*}
\sum_iG(F_i) \leq \liminf_{j\to\infty}\sum_iG(C_i^{t_{m_j}})
\leq \lim_{j\to\infty}\sum_{i=0}^pG(C_i^{t_{m_j}})+\epsilon 
= \sum_{i=0}^p G(F_i) +\epsilon
\leq \sum_iG(F_i) +\epsilon,
\end{align*}
where the first inequality follows from Fatou's lemma (for series). 
Since $\epsilon$ was arbitrary, we get
\begin{equation}
\label{ElimG1}
\lim_{j\to\infty}\sum_iG(C_i^{t_{m_j}})=\sum_iG(F_i).
\end{equation}

Because the $C_i^{t_{m_j}}$ are pairwise disjoint mod-$\mathcal{L}_n$, so are the $F_i$. Moreover, since $t_0\notin N$ by definition of $\Sigma_\eta$, we have that
\begin{equation}
\label{ulsint}
\lim_{t\to t_0}\mathcal{L}_n\left((E_t\setminus E_{t_0})\bigcup (E_{t_0}\setminus E_t)\right)=0,
\end{equation}
implying by local convergence in measure 
that $F_i\subset E_{t_0}$ mod-$\mathcal{L}_n$ for each $i$. In addition,
as $\alpha>1$, we see that
\eqref{cvdr1} {\it a fortiori} implies: 
\[
\sum_i  \int_{F_i}e^{-|x|^2}dx=\lim_{j\to\infty}\sum_i\int_{C_i^{t_{m_j}}}e^{-|x|^2}dx=\lim_{j\to\infty}\int_{E_{t_{m_j}}}e^{-|x|^2}dx
=\int_{E_{t_0}}e^{-|x|^2}dx,
\]
where the last equality follows from \eqref{ulsint}.
Thus, as $e^{-|x|^2}>0$ for all $x$, we
get $\mathcal{L}_n(E_{t_0}\setminus\cup_iF_i)=0$, 
whence the $F_i$ partition $E_{t_0}$ mod-$\mathcal{L}_n$. 
Also,  by the lower semi-continuity of perimeter
with respect to local convergence in measure, we get that 
\begin{equation}
\label{lscp}
\sum_i\mathcal{P}(F_i)\leq\lim_j\sum_i\mathcal{P}(C_i^{t_{m_j}})
=\lim_j\mathcal{P}(E_{t_{m_j}})=\mathcal{P}(E_{t_0}),
\end{equation}
where the last equality comes from the continuity of
$t\mapsto\mathcal{P}(E_t)$ on $\Sigma_\eta$. Therefore, 
by the maximizing property \eqref{maxcompc1}
 of $M$-connected components, it holds that
\begin{equation}
\label{limpres1}
\sum_i G(F_i)\leq \sum_i G(C_i^{t_0}).
\end{equation}
 \emph{We claim} that in fact 
$\sum_i G(F_i)= \sum_i G(C_i^{t_0})$. To show this, it is enough to 
consider separately the two cases where $t_{m_j}\to t_0$ from above and
from below. Assume first that $t_{m_j}>t_0$ for all $j$, whence 
$E_{t_{m_j}}\subset E_{t_0}$.
Set $F_i^{t_{m_j}}:=E_{t_{m_j}}\cap C_i^{t_0}$ and observe that
the $(F_i^{t_{m_j}})_{i\in\NatZer}$ are disjoint mod-$\mathcal{L}_n$
and form a partition of $E_{t_{m_j}}$ mod-$\mathcal{L}_n$. 
As 
$\partial_MF_i^{t_{m_j}}\subset \partial_ME_{t_{m_j}}\cup\partial_MC_i^{t_0}$
by  definition  \eqref{defthb},
and because each point of $\partial_MF_i^{t_{m_j}}\setminus\partial_MC_i^{t_0}$
is clearly a density point of $C_i^{t_0}$, we get since 
the sets of density points of the $C_i^{t_0}$ are pairwise disjoint while  
$\mathcal{H}^{n-1}(\partial_MC_{i_1}^{t_0}
\cap\partial_MC_{i_2}^{t_0})=0$ for $i_1\neq i_2$ (see \cite[Proposition 3]{Ambetal})
that  the $\partial_MF_i^{t_{m_j}}$ are pairwise disjoints mod-$\mathcal{H}^{n-1}$. Hence, by \cite[Proposition 3]{Ambetal} again, it holds that
$\mathcal{P}(E_{t_{m_j}})=\sum_i\mathcal{P}(F_i^{t_{m_j}})$ and so the
$F_i^{t_{m_j}}$ are candidate maximizers in \eqref{maxcompc1} if 
we put $E=E_{t_{m_j}}$ there.
However, 
as $\mathcal{L}_n(E_{t_0}\setminus E_{t_{m_j}})\to0$ by
\eqref{ulsint},
it holds that
$\sum_i\mathcal{L}_n(C_i^{t_0}\setminus F_i^{t_{m_j}})\to0$ when $j\to\infty$, and
since $e^{-|x|^2}$ is summable we get by dominated convergence that
\begin{equation}
\label{partind}
\sum_iG(C_i^{t_0})=\lim_j\sum_iG(F_i^{t_{m_j}})\leq\lim_j\sum_iG(C_i^{t_{m_j}}), 
\end{equation}
where the last inequality comes from the maximizing character of the
$(C_i^{t_{m_j}})$ in \eqref{maxcompc1} when $E=E_{t_{m_j}}$. 
The claim in this case now follows from \eqref{partind}, \eqref{limpres1}
and \eqref{ElimG1}. Assume next that $t_{m_j}<t_0$ for all $j$, whence 
$E_{t_{m_j}}\supset E_{t_0}$. Since $C^{t_0}_i$ is indecomposable and 
$C^{t_0}_i\subset E_{t_{m_j}}$, it holds that $C^{t_0}_i\subset C^{t_{m_j}}_{\ell_i}$
mod-$\mathcal{L}_n$
for some $\ell_i$,
by \cite[Theorem~1]{Ambetal}. Obviously then,
$\sum_iG(C_i^{t_0})\leq\sum_iG(C_i^{t_{m_j}})$, and in view of
\eqref{ElimG1}, \eqref{limpres1} \emph{this proves the claim in all cases}.

From the claim,  we deduce by uniqueness of a maximizer in \eqref{maxcompc1}
that $(F_i)_{i\in\NatZer}$ and $(C_i^{t_0})_{i\in\NatZer}$ are equivalent in $\dot{\mathcal{S}}$,
thereby proving $(ii)$. In particular
$\sum_i\mathcal{P}(F_i)=\mathcal{P}(E_{t_0})$, and since $\lim_j\mathcal{P}(C^{t_{m_j}}_i)\geq \mathcal{P}(F_i)$ for each $i$ 
by lower semi-continuity of the perimeter under local convergence in measure, 
we deduce from \eqref{lscp} that 
$\lim_j\mathcal{P}(C^{t_{m_j}}_i)=\mathcal{P}(F_i)$, thereby proving the second
half of $(iii)$. To prove the first half, observe that if $t_{m_j}>t_0$ then 
$E_{t_{m_j}}\subset E_{t_0}$. Therefore 
$C^{t_{m_j}}_i$, which is indecomposable, must be included in $C^{t_0}_\ell$
for some $\ell=\ell(i,j)$. But for $j$ large enough $C^{t_0}_\ell$
can be none but $F_i$,
and so $\lim_j\mathcal{L}_n(F_i\setminus C^{t_{m_j}}_i)\leq
\lim_j\mathcal{L}_n(E_{t_0}\setminus E_{t_{m_j}})=0$, by \eqref{ulsint}.
If on the contrary  $t_{m_j}<t_0$, then 
$E_{t_{m_j}}\supset E_{t_0}$ and each
$C^{t_0}_\ell$, which is indecomposable, must be included in $C^{t_{m_j}}_i$
for some  $i=i(\ell,j)$.
Necessarily then, it holds that $C^{t_0}_\ell=F_i$, and so
$\lim_j\mathcal{L}_n(C^{t_{m_j}}_i\setminus F_i)\leq
\lim_j\mathcal{L}_n( E_{t_{m_j}}\setminus E_{t_0})=0$, 
by \eqref{ulsint} again. Since every $F_i$ is a
$C^{t_0}_\ell$ for some $\ell=\ell(i)$, this proves $(iii)$.

To establish $(iv)$, note since
$\sum_{i=0}^\infty\mathcal{P}(C^{t_0}_i)<\infty$ that  to each $\varepsilon>0$
there is 
$i_0\geq1$ with $\sum_{i=i_0}^\infty\mathcal{P}(C^{t_0}_i)<\varepsilon$.
Then, by  lower-semi continuity of the
perimeter with respect to local convergence in measure, there is
$j_0=j_0(i_0)$ so large that
\[\sum_{i=0}^{i_0-1}\mathcal{P}(C^{t_{m_j}}_i)>\sum_{i=0}^{i_0-1}\mathcal{P}(C^{t_0}_i)-\varepsilon,\qquad j\geq j_0,
\]
and since $\lim_j\sum_i\mathcal{P}(C_i^{t_{m_j}})=\sum_i\mathcal{P}(C_i^{t_0})$
by \eqref{lscp}, we get for  $j$
large enough 
$\sum_{i=i_0}^\infty\mathcal{P}(C^{t_{m_j}}_i)\leq2\varepsilon$.
As $\varepsilon$ was arbitrary, this gives us the first limit in \eqref{estqcc}, which implies the second by the isoperimetric inequality
because $\mathcal{L}_n(C_i^{t_{m_j}})<\infty$ for $i\geq1$.
\end{proof}

We equip $\mathcal{S}$ with the distance $d_{\mathcal{S}}((E_i),(E_i^\prime))=\sup_id_1(\chi_{E_i},\chi_{E_i^\prime})$, where $d_1$ is a distance function
on $L^1_{loc}(\R^n)$, 
and we endow $\dot{\mathcal{S}}$ with the quotient topology ({\it i.e.} the
coarsest topology such that the canonical map
$\mathcal{S}\to\dot{\mathcal{S}}$ is continuous). Then, 
Proposition \ref{convmcc} may be construed as an approximate continuity result
of  the $M$-connected components of the suplevel sets of a homogeneous 
$BV$-function with respect to the level.
Recall that a map $\psi:\R\to \mathcal{E}$, with  $\mathcal{E}$ a topological space, is approximately contiuous at $t_0\in\R$ if, for
every neighborhood $V\subset\mathcal{E}$ of $\psi(t_0)$, it holds that  
\begin{equation}
\label{acE}
\lim_{r\to0}\frac{\mathcal{L}_1\left( \{t:\,|t-t_0|<r,\;\psi(t)\notin V\}\right)}{r}=0.
\end{equation}

\begin{theorem}
\label{approxceq}
Let $\phi\in\dot{BV}(\R^n)$ and $E_t$ its suplevel set
at level $t$, {\it cf.}  \eqref{defsl}. 
Then, 
the map  $\psi:\R\to\dot{\mathcal{S}}$ sending $t$ to the $M$-connected
components of $E_t$ is approximately continuous 
$\mathcal{L}_1$-a.e.
\end{theorem}
\begin{proof}
By Proposition \ref{convmcc} points $(ii)$, $(iv)$, 
and the definition of the quotient topology, $\psi$ is continuous on $\Sigma_\eta$ for each $\eta>0$.
So, when $t_0$ is a density point of $\Sigma_\eta$ for some $\eta>0$, then 
\eqref{acE} holds. 
But if $D_\eta$ denotes the set of such density 
points, then $\R\setminus(\cup_{k\geq1} D_{1/k^2})$ has measure zero, by the Borel-Cantelli lemma.
Hence \eqref{acE} holds a.e.
\end{proof}

\section{Loop decomposition of divergence-free planar measures}
 \label{loppdecSM}
 In this section, we  make  use of 
Lemma \ref{coarea} and Proposition \ref{convmcc} when $n=2$ to decompose 
gradients of functions in $\dot{BV}(\R^2)$ as a continuous sum of measures of the form \eqref{dermeas}, with $\partial _M E$ a rectifiable Jordan curve.
The results in this section, up to and including Proposition \ref{convmcur}, could be developed in an analogous way for $n\geq3$, replacing Jordan curves with Jordan boundaries (see \cite{Ambetal}). 
However, we stick with $n=2$ since our main application, stated in Theorem \ref{rep}, is to describe divergence-free vector fields whereas the connection 
with gradients, stated in Lemma \ref{2Dsole}, only works in the plane.

\begin{lemma}\label{normalsubset}
	Let $E, F\subset\R^2$ be sets  of finite perimeter such that $\mathcal{L}_2(E\setminus F)=0$. 
	Then for $\mathcal{H}^1$-a.e. $x\in\partial_M E\cap\partial_M F$, it holds that $\bs{\nu}_F(x)=\bs{\nu}_E(x)$.
\end{lemma}
\begin{proof}
		Given $\epsilon>0$, $x,v\in\R^2$ with $v\neq0$ and $G\subset\R^2$, define the half-disk 
		\begin{equation}\label{defhp}
		H_\epsilon(x,v):=\{y\in \B(\epsilon,x)\ :\ (y-x)\cdot v>0\},
		\end{equation} 
		and let
		$$
		L_G(x,v):=\lim_{\epsilon\rightarrow 0}\frac{\mathcal{L}_2(H_\epsilon(x,v)\cap G)}{\mathcal{L}_2(H_\epsilon(x,v))}=
\lim_{\epsilon\rightarrow 0}\frac{2\mathcal{L}_2(H_\epsilon(x,v)\cap G)}{\pi\epsilon^2}
		$$
		whenever the limit exists.
		Assume $G$ has finite perimeter. Then, for $\mathcal{H}^1$-a.e. $x\in\partial_M G$, $\bs{\nu}_G(x)$ is the unique unit vector that satisfies 
		$$
		L_G(x,\bs{\nu}_G(x))=1\qquad \text{ and }\qquad L_G(x,-\bs{\nu}_G(x))=0,
		$$
		(see \cite[Proposition 10.3.4 and Theorem 10.3.2]{Attetal2006}
or \cite[Thm. 5.6.5]{Ziemer}).
		Since $E$ is included in $F$ except for a set of
$\mathcal{L}_2$-measure zero, clearly $L_E(x,-\bs{\nu}_F(x))=0$ 
for $\mathcal{H}^1$-a.e. $x\in\partial_M F$. Let $Z\subset\partial_M F$
be the set consisting of such $x$. Moreover, $L_E(x,\bs{\nu}_E(x))=1$
for $\mathcal{H}^1$-a.e. $x\in\partial_M E$, and we let
$Y\subset\partial_M F$ be the set consisting of such $x$. Now,
if for $x\in X\cap Y$ we had $\bs{\nu}_E\neq\bs{\nu}_F$, the truncated positive cone 
$C_\epsilon:=H_\epsilon(x,-\bs{\nu}_F)\cap H_\epsilon(x,\bs{\nu}_E)$ would have strictly 
positive  angle, say  $\theta$, and since 
\[
\limsup_{\epsilon\to0} \frac{2\mathcal{L}_2(H_\epsilon(x,\bs{\nu}_E)\cap E\cap C_\epsilon)}{\pi\epsilon^2} = 
\limsup_{\epsilon\to0} \frac{2\mathcal{L}_2(E\cap C_\epsilon)}{\pi\epsilon^2}\leq L_E(x,-\bs{\nu}_F)=0,
\]
we would have that
\begin{multline*}
L_E(x,\bs{\nu}_E)= \lim_{\epsilon\to0} \frac{\mathcal{L}_2(H_\epsilon(x,\bs{\nu}_E)\cap (E\setminus C_\epsilon))}{\mathcal{L}_2(H_\epsilon(x,\bs{\nu}_E))}
\leq\limsup_{\epsilon\to0} \frac{\mathcal{L}_2(H_\epsilon(x,\bs{\nu}_E)\setminus C_\epsilon)}{\mathcal{L}_2(H_\epsilon(x,\bs{\nu}_E))}
\leq 1-\frac{\theta}{\pi},
\end{multline*}
a contradiction.
\end{proof}

Let us make one more piece of notation:
for $\Gamma\subset\R^2$ a Jordan curve,  
we denote by $\text{int}(\Gamma)$ (resp. $\text{ext}(\Gamma)$)
the bounded (resp. unbounded) connected component of 
$\R^2\setminus \Gamma$.

\begin{lemma}\label{Gamma}
	If $\Gamma\subset\R^2$ is a rectifiable Jordan curve, then $\partial_M(\text{\rm int}(\Gamma))=\Gamma$ mod-$\mathcal{H}^1$.
\end{lemma}
\begin{proof}
	Clearly $\partial_M(\text{int}(\Gamma))$ is a subset of the topological boundary of $\text{int}(\Gamma)$ which is $\Gamma$. 
Now, by \cite[Proposition 2 \& Theorem 7]{Ambetal}, $\partial_M(\text{int}(\Gamma))$ is equal to a rectifiable Jordan curve $\tilde{\Gamma}$ mod-$\mathcal{H}^1$.
	
	Thus, $\mathcal{H}^1(\tilde{\Gamma}\setminus\Gamma)=0$ whence 
	$\tilde{\Gamma}\cap\Gamma$ is dense in $\tilde{\Gamma}$, and so $\tilde{\Gamma}\subset\Gamma$ by compactness of $\Gamma$.
	Therefore, by the Jordan curve theorem, $\tilde{\Gamma}=\Gamma$ which implies our lemma.	
\end{proof}


In \cite[Corollary 1]{Ambetal}, it is shown that the measure-theoretical boundary of a planar set $E$ with finite perimeter can be decomposed
into countably many  Jordan curves.
The next lemma elaborates on this result and relates the latter
decomposition with the $M$-connected components of the complements 
of the $M$-connected components of $E$.
\begin{lemma}\label{Jordan_decomposition}
The measure-theoretical boundary of a set $E\subset\R^2$ of finite perimeter decomposes mod-$\mathcal{H}^1$ as 
the union of two countable families of rectifiable Jordan curves $\{ \Gamma^+_k\}_{k\in K}$ and $\{\Gamma^-_j \}_{j\in J}$, with $K,J\,\subset\{1,2,3\cdots\}$, such that
\begin{align}
\grad\chi_E  &= \sum_{k\in K}\grad\chi_{\mathrm{int}(\Gamma^+_k)} -\sum_{j\in J}\grad\chi_{\mathrm{int}(\Gamma^-_j)}	\label{eq:Jordan_grad} \\
\mathcal{H}^1\mathcal{b}(\partial_M E) &= \sum_{k\in K}\mathcal{H}^1\mathcal{b}\Gamma^+_k +\sum_{j\in J}\mathcal{H}^1\mathcal{b}\Gamma^-_j. \label{eq:Jordan_H1}
\end{align} 
Moreover, if we let
\begin{equation}
\label{defIke}
I_k:=\{j\in J : \text{\rm int}(\Gamma^-_j)\subset\text{\rm int}(\Gamma^+_k) \}\qquad
\mbox{\rm and}\qquad Y_k=\text{\rm int}(\Gamma^+_k)\setminus\cup_{j\in I_k}\text{\rm int}(\Gamma^-_j),
\end{equation} 
 as well as
\begin{equation}
\label{defTf}
Y_0:=\bigcap_{j\in J}\text{\rm ext}(\Gamma^-_j)
\ \text{if}\ \mathcal{L}_2(E)=\infty \ \text{\rm and}\    Y_0:=\emptyset\ 
\text{ otherwise},
\end{equation}
then the $Y_k$ for $k\in K$, together with $Y_0$ if nonempty, are the $M$-connected components of $E$. In particular, it holds that 
\begin{equation}
\label{decE}
E=\left(\bigcup_{k\in K}Y_k\right)\cup Y_0\quad
\mbox{\rm mod-}\mathcal{L}_2.
\end{equation}
In addition, if we put 
\begin{equation}
\label{defIkg}
\tilde I_{k}:=\{j\in I_{k}: \ \text{there is no } k'\in K
\text{ such that} \ \text{\rm int}(\Gamma_{k}^+)\supsetneq
\text{\rm int}(\Gamma_{k'}^+)\supset \text{\rm int}(\Gamma_j^-)\}
\end{equation}
along with 
\begin{equation}
\label{defJt0}
I_\infty:=\{j\in J: \ \text{there is no } k\in K
\text{ such that} \ \text{\rm int}(\Gamma_{k}^+)\supset
\text{\rm int}(\Gamma_{j}^-)\},
\end{equation}
then $I_\infty\neq\emptyset$ if and only if $\mathcal{L}_2(E)=\infty$
and each $j\in J$ belongs to $\tilde I_{k}$ for some unique $k$ or else to
 $I_\infty$. Furthermore, for each $k\in K$,
the sets $\{\mathrm{int }(\Gamma^-_i)\}_{i\in {\tilde I}_k}$ together
with $\mathrm{ext }(\Gamma^+_k)$ are the $M$-connected components of
$\R^2\setminus Y_k$, and if $\mathcal{L}_2(E)=\infty$ then
the $\{\mathrm{int }(\Gamma^-_j)\}_{j\in {I\infty}}$ are the 
$M$-connected components of $\R^2\setminus Y_0$.
\end{lemma}
\begin{proof}
	By \cite[Corollary 1]{Ambetal}, there exists two families $\{ \Gamma^+_k\}_{k\in K}$ and $\{\Gamma^-_j \}_{j\in J}$ of countably many rectifiable Jordan curves (we can always take $K,J\subset \{1,2,3\cdots\})$, satisfying:
	\begin{enumerate}
		\item $\partial_M E=\bigcup_k\Gamma^+_k\cup\bigcup_j\Gamma^-_j$ mod-$\mathcal{H}^1$,
		\item For any two $\text{int}(\Gamma^+_k)$ and $\text{int}(\Gamma^+_l)$ either one is contained in the other or they are disjoint. 
		Similarly, for any two $\text{int}(\Gamma^-_j)$ and $\text{int}(\Gamma^-_i)$ either one is contained in the other or they are disjoint. 
		\item $\mathcal{H}^1(\partial_M E)=\sum_k\mathcal{H}^1(\Gamma_k^+)+\sum_j\mathcal{H}^1(\Gamma^-_j)$, in particular the curves are disjoint 
mod-$\mathcal{H}^1$.
		\item If $l\neq k$ and $\text{int}(\Gamma^+_k)\subset\text{int}(\Gamma^+_l)$ then there exists a $\text{int}(\Gamma^-_j)$ with the property
that $\text{int}(\Gamma^+_k)\subset\text{int}(\Gamma^-_j)\subset\text{int}(\Gamma^+_l)$. 
		Analogously, if $j\neq i$ and $\text{int}(\Gamma^-_j)\subset\text{int}(\Gamma^-_i)$ then there exists a $\text{int}(\Gamma^+_k)$ such that $\text{int}(\Gamma^-_j)\subset\text{int}(\Gamma^+_k)\subset\text{int}(\Gamma^-_i)$.
		\item The $Y_k$ defined in \eqref{defIke}, along with 
$Y_0$ defined in \eqref{defTf} if nonempty\footnote{In \cite[Cor. 1]{Ambetal},
the set $Y_0$ is not introduced, but an abstract  ``Jordan curve" $\Gamma_\infty^+$,
reducing to  the point at $\infty$ ({\it i.e.} having zero length
and interior $\R^2$), is allowed  in case $\mathcal{L}_2(E)=\infty$, 
so that $Y_0$ corresponds to
$\text{int}(\Gamma^+_\infty)\setminus\cup_j\text{int}(\Gamma_j^-)$.}, are the $M$-connected components of $E$, in particular \eqref{decE} holds. Note
that  if $\mathcal{L}_2(E)=\infty$, then $Y_0$ is the $M$-connected component 
of infinite $\mathcal{L}_2$-measure.
Note also that $\mathcal{L}_2(E)=\infty$ (equivalently: $Y_0\neq\emptyset$)
if and only if there exists a $\text{int}(\Gamma^-_j)$ not contained in any $\text{int}(\Gamma^+_k$), that is: if and only if $I_\infty\neq\emptyset$.
	\end{enumerate}
	
	It remains for us to show that this decomposition satisfies 
\eqref{eq:Jordan_grad} and that the last two assertions after \eqref{defJt0}
do hold.
In view of \eqref{dermeasd} and \eqref{eq:Jordan_H1},
it is enough for \eqref{eq:Jordan_grad} to hold  that 
	\begin{enumerate}[label=(\roman*)]
		\item 
for any $k\in K$, $\grad\chi_E\mathcal{b}\Gamma_k^+=\grad\chi_{\text{int}(\Gamma_k^+)}$,
		\item 
for any $j\in J$, $\grad \chi_E\mathcal{b}\Gamma_j^-=-\grad\chi_{\text{int}(\Gamma_j^-)}$.
	\end{enumerate}
To obtain (i) and (ii), we will prove that for each $k_0\in K$
(resp. $j_0\in J$) and $\mathcal{H}^1$-a.e. $x\in\Gamma^+_{k_0}$ (resp. $\Gamma^-_{j_0}$), we have $\bs{\nu}_E(x)=\bs{\nu}_{\text{int}\Gamma^+_{k_0}}(x)$
(resp. $\bs{\nu}_E(x)=-\bs{\nu}_{\text{int}\Gamma^-_{j_0}}(x)$).

	
	Fix $k_0\in K$ and	
	let $F_{k_0}:=\text{int}(\Gamma^+_{k_0})\cap E$. Define  $\tilde K:=\{k\in K : \text{int}(\Gamma^+_k)\subset\text{int}(\Gamma^+_{k_0}) \}$ and $\tilde J:=\bigcup_{k\in\tilde K}I_k$.
	The pair of families of rectifiable Jordan curves 
$\{ \Gamma^+_k\}_{k\in \tilde  K}$, $\{\Gamma^-_j \}_{j\in \tilde J}$  
{\it a fortiori} meets properties (b) and (d) above when
 the indices $k$, $l$ and
$j$, $i$ range over $\tilde  K$ and 
$\tilde J$, respectively.
Also, by (c), these families are such that 
	\begin{enumerate}
		\item[(f)] each two different Jordan curves are disjoint mod-$\mathcal{H}^1$, 
		\item[(g)] $\sum_k\mathcal{H}^1(\Gamma_k)+\sum_j\mathcal{H}^1(\Gamma^-_j)<\infty$, $k\in\tilde K$, $j\in\tilde J$.
	\end{enumerate}
	Moreover, we get from (b)  and \eqref{decE} that
        \begin{enumerate}
\item[(h)] $F_{k_0}=\bigcup_{k\in\tilde{K}} Y_k$ mod-$\mathcal{L}_2$. 
       \end{enumerate}
Properties (b), (d), (f), (g) and (h) show that $F_{k_0}$, $\{ \Gamma^+_k\}_{k\in \tilde  K}$ and $\{\Gamma^-_j \}_{j\in \tilde J}$  
	satisfy  the assumptions of
\cite[Theorem 5]{Ambetal}. The latter implies that $F_{k_0}$ has
finite perimeter and that $\partial_M F_{k_0}=\bigcup_{k\in\tilde{K}}\Gamma^+_k\cup\bigcup_{j\in\tilde{J}}\Gamma^-_j$ mod-$\mathcal{H}^1$.
	Applying Lemma \ref{normalsubset} twice, we now get
that $\bs{\nu}_E(x)=\bs{\nu}_{F_{k_0}}(x)=\bs{\nu}_{\text{int}(\Gamma^+_{k_0})}(x)$ for $\mathcal{H}^1$-a.e. $x\in(\partial_M F_{k_0}\cap\partial_M E\cap\partial_M\text{int}(\Gamma^+_{k_0}))$, and by Lemma \ref{Gamma}
this intersection reduces to $\Gamma^+_{k_0}$ mod-$\mathcal{H}^1$. This proves (i).
	

To prove (ii), pick $j_0\in J$ and assume first that $j_0\notin I_\infty$, 
so there is  ${k_0}\in K$ such that $\text{int}(\Gamma^+_{k_0})\supset \text{int}(\Gamma^-_{j_0})$. As there is no infinite sequence 
$\text{\rm int}(\Gamma_{{\ell}_1}^+)\supsetneq\text{\rm int }(\Gamma_{{\ell}_2}^+)\supsetneq\cdots$ each element of which contains $\text{int}(\Gamma^-_{j_0})$ (otherwise the isoperimetric inequality would imply that
$\pi^{1/2}\mathcal{H}^1(\Gamma^+_{{\ell}_i})\geq\mathcal{L}^{1/2}_2(\text{int}(\Gamma_{j_0}^-))>0$ for all $i$ and this would contradict (g)),
we may choose  ${k_0}$ so that $\text{int}(\Gamma^+_{k_0})$ is smallest 
with the property that $\text{int}(\Gamma^+_{k_0})\supset \text{int}(\Gamma^-_{j_0})$ or, equivalently, such that $j_0\in \tilde{I}_{k_0}$ defined in \eqref{defIkg}. Note that such a $k_0$ is unique, by (b), thereby proving
in passing the next-to-last  assertion after \eqref{defJt0}.

Now, the sets $\{\text{int}(\Gamma^-_j)\}_{j\in\tilde I_{k_0}}$ are disjoint, by (b) and (d). Moreover, for each
$i\in I_{k_0}$,  there is $j\in\tilde{I}_{k_0}$ such that
 $\text{int}(\Gamma^-_i)\subset\text{int}(\Gamma^-_j)$, because of (d) and  
the fact that there is no infinite sequence
$\text{int}(\Gamma_{j_1}^-)\subsetneq\text{int}(\Gamma_{k_1}^+)\subsetneq\text{int}(\Gamma^-_{j_2})\subsetneq
\text{int}(\Gamma_{k_2}^+)\cdots$, by (c) and the isoperimetric inequality again. In particular, we have that 
\begin{equation}
\label{redYell}
Y_{k_0}=\text{int}(\Gamma^+_{k_0})\setminus\bigcup_{j\in \tilde I_{k_0}}\text{int}(\Gamma^-_j).
\end{equation} 
Thus, the set  $Y_{k_0}$ and the pair of families of curves
$\{\Gamma_{k_0}^+\}$, $\{\Gamma^-_j,\,j\in\tilde I_{k_0}\}$ 
(the first family has only one element) satisfy the assumptions of \cite[Theorem 5]{Ambetal}, to the effect that 
\begin{equation}
\label{transitTell}
\partial_M Y_{k_0}=\Gamma_{k_0}^+\cup \bigcup_{j\in\tilde I_{k_0}}\Gamma_j^-
\quad\text{ mod-}\mathcal{H}^1. 
\end{equation}
In another connection, if we define $F_{k_0}$ as before, we get from the first part of the proof and Lemma \ref{normalsubset}  that 
\begin{equation}
\label{initm}
\Gamma_{j_0}^-\subset\partial_M F_{k_0}\cap\partial_M E\qquad \text{and}\qquad
\bs{\nu}_E(x)=\bs{\nu}_{F_{k_0}}(x),\quad  \mathcal{H}^1\text{-a.e.} \  x\in\Gamma^-_{j_0}.
\end{equation}
Moreover, (h) implies that
$F_{k_0}\supset Y_{k_0}$
mod-$\mathcal{L}_2$, and \eqref{initm},
\eqref{transitTell} that  
$\Gamma^-_{j_0}\subset \partial_M F_{k_0}\cap\partial_M Y_{k_0}$ mod-$\mathcal{H}^1$,  therefore we conclude from
 Lemma \ref{normalsubset} that 
\begin{equation}
\label{transinormal}
\bs{\nu}_{F_{k_0}}(x)=\bs{\nu}_{Y_{k_0}}(x),\qquad
\mathcal{H}^1\text{-a.e. } x\in \Gamma^-_{j_0}.
\end{equation}
Besides, since $Y_{k_0}\subset \text{ext}(\Gamma^-_{j_0})$ by \eqref{redYell},
while $\Gamma^-_{j_0}\subset \partial_M Y_{k_0}\cap\partial_M  \text{ext}(\Gamma^-_{j_0})$ mod-$\mathcal{H}^1$ by \eqref{transitTell} and Lemma \ref{Gamma},
we get from  Lemma \ref{normalsubset}  again that
\begin{equation}
\label{fineqnm}
\bs{\nu}_{Y_{k_0}}(x)=
\bs{\nu}_{\text{ext}(\Gamma^-_{j_0})}(x)=
-\bs{\nu}_{\text{int}(\Gamma^-_{j_0})}(x),\qquad
\mathcal{H}^1\text{-a.e. }x\in \Gamma^-_{j_0}. 
\end{equation}
The conjunction of \eqref{initm}, \eqref{transinormal} and \eqref{fineqnm}
proves (ii) when  $j_0\notin I_\infty$.
Next, assume that $j_0\in I_\infty$; in particular $I_\infty\neq\emptyset$
so that $Y_0\neq\emptyset$, where $Y_0$ was defined 
in \eqref{defTf}.
 If we define
\begin{equation}
\label{defIt}\tilde I:=\{i\in J : \ \text{there is no } j\in J\ 
\text{such that} \ \text{int}(\Gamma_j^-)\supsetneq
\text{int}(\Gamma_i^-)\},
\end{equation}
we obviouly have that $Y_0=\bigcap_{i\in \tilde I}\text{ext}(\Gamma^-_j)$.
Note that the sets $\{\text{int}(\Gamma^-_i)\}_{i\in\tilde I}$ are disjoint,
by (b). Thus, if we let $\Upsilon^+_i:=\Gamma^-_i$, we get in view of (c)
that the set $\R^2\setminus Y_0=\bigcup_{i\in\tilde I}\text{int}(\Upsilon^+_i)$
together with the pair of families
of rectifiable Jordan curves
$\{\Upsilon^+_i,\,i\in\tilde I\}$, $\emptyset$ ({\it i.e.} 
the second family is empty),
satisfy the assumptions of    \cite[Theorem 5]{Ambetal}. The latter
implies that 
\begin{equation}
\label{bY}
\partial_M(\R^2\setminus Y_0) =\bigcup_{i\in\tilde I}\Gamma^-_i,
\end{equation}
and since  $j_0\in\tilde I$, by (d), we get from  Lemma \ref{normalsubset}   that
$\bs{\nu}_{\text{int}(\Gamma^-_{j_0})}(x)=\bs{\nu}_{\R^2\setminus Y_0}(x)=-
\bs{\nu}_{Y_0}(x)$ for $\mathcal{H}^1$-a.e. $x\in \Gamma_{j_0}^-$.
As  $Y_0\subset E$ and $\Gamma_{j_0}^-\subset\partial_M E\cap\partial_M Y_0$,
by \eqref{bY},
another application of Lemma \ref{normalsubset}  yields that
$\bs{\nu}_{Y_0}(x)=\bs{\nu}_{E}(x)$ for $\mathcal{H}^1$-a.e. 
$x\in \Gamma_{j_0}^-$, thereby establishing (ii) in this case as well.

To prove the last assertion after \eqref{defJt0}, pick $k\in K$ and observe 
from (b) and (d) that the sets $\mathrm{ext }(\Gamma^+_k)$ and
$\{\mathrm{int }(\Gamma^-_i)\}_{i\in {\tilde I}_k}$ are pairwise disjoint,
while $\R^2\setminus Y_k$ is their union. These 
sets are indecomposable, by Lemma \ref{Gamma} and \cite[Theorem~2]{Ambetal},
and since their measure-theoretical boundaries are pairwise
disjoint mod-$\mathcal{H}^1$,
because of (c), we deduce from \cite[Propostion~3]{Ambetal} that
their perimeters add up to $\mathcal{P}(\R^2\setminus Y_k)$.
Hence, they are indeed the $M$-connected components of
$\R^2\setminus Y_k$. If $\mathcal{L}_2(E)=\infty$, so that
$Y_0\neq\emptyset$, a  similar reasoning on \eqref{bY} shows that
the $\{\mathrm{int }(\Gamma_i^-)\}_{i\in\tilde I}$ are the $M$-connected components of $\R^2\setminus Y_0$, and it remains for us to prove that $\tilde I=I_\infty$.
From (d), we know that $I_\infty\subset \tilde I$. Conversely,
if $j\in J$ and $j\notin I_\infty$, we showed earlier there is a unique 
$k_0\in K$ 
such that $j\in \tilde{I}_{k_0}$. We also know that   $Y_{k_0}$ 
is a $M$-connected component of $E$, therefore it is indecomposable and 
disjoint 
mod-$\mathcal{L}_2$ from $Y_0$ which is another such component.
Consequently,  by 
\eqref{bY}, we have that
$Y_{k_0}\subset \bigcup_{i\in \tilde I}\text{int}(\Gamma^-_i)$
mod-$\mathcal{L}_2$. As the $\{\mathrm{int }(\Gamma_i^-)\}_{i\in\tilde I}$ are the $M$-connected components of $\R^2\setminus Y_0$ and  $Y_{k_0}$ is 
indecomposable, we get that $Y_{k_0}\subset \mathrm{int }(\Gamma_{i_0}^-)$
mod-$\mathcal{L}_2$ for some $i_0\in\tilde I$. 
It implies easily that 
$\mathcal{H}^1$-a.e. point of $\partial_MY_{k_0}$ is not a density point of
$\mathrm{ext }(\Gamma_{i_0}^-)$. 
{\it A fortiori} then, by
\eqref{transitTell}, $\Gamma^-_j\subset
\overline{\mathrm{int }(\Gamma_{i_0}^-)}$ mod-$\mathcal{H}^1$
where the bar indicates Euclidean closure. Since
$\Gamma^-_j$ is a closed curve we get in fact that
$\Gamma^-_j\subset
\overline{\mathrm{int }(\Gamma_{i_0}^-)}$, and by the Jordan curve theorem
it follows that $\mathrm{int }(\Gamma^-_j)\subset
\mathrm{int }(\Gamma_{i_0}^-)$, whence $j\notin \tilde I$.
The proof is now complete.
\end{proof}

Lemma \ref{Jordan_decomposition} tells us  that the measure-theoretical boundary of a set $E$ of finite perimeter consists of two countable families of 
Jordan curves, namely  
$\{ \Gamma^+_k\}_{k\in K}$ and $\{\Gamma^-_j \}_{j\in J}$, such that 
the $\mathrm{int }\,\Gamma^-_j$ and the $\mathrm{ext }\Gamma^+_k$ are the 
$M$-connected components of the complements of the $M$-connected components of $E$. This  will allow us to put a structure on these Jordan curves.
More precisely, recall from Section \ref{BVdsec} (we put $n=2$)
the set $\mathcal{S}$
 of sequences of subsets of $\R^2$ mod-$\mathcal{L}_2$
whose $\mathcal{L}_2$-measures are non increasing and tend to zero,
as well as the set $\dot{\mathcal{S}}$ of equivalence classes modulo permutations. As stressed in that section, the $M$-connected components of a set
of finite perimeter may  be regarded as a member of $\dot{\mathcal{S}}$,
a representative of which is obtained in  $\mathcal{S}$ by arranging the 
$M$-connected components in nonincreasing  measure, and appending to them infinitely many copies of the emptyset if these components are finite in number.
For $S\in\mathcal{S}$, say $S=(F_0,F_1,F_2,\cdots)$, we let for
simplicity $\mathfrak{U}_ S=\cup_j F_j$, and  we let $\mathcal{T}$ be the 
subset of  $\mathcal{S}^{\NatZer}$ consisting
of sequences $(S_0,S_1,S_2,\cdots)$ such that
$(\R^2\setminus\mathfrak{U}_ {S_0}\,,\,\R^2\setminus\mathfrak{U}_{S_1}\,,\,
\R^2\setminus\mathfrak{U}_{S_2}\,,\cdots)$
also lies in $\mathcal{S}$.
We say that two elements $(S_i)_{i\in\NatZer}$ and $(S'_i)_{i\in\NatZer}$
of $\mathcal{T}$ are equivalent if there
 is a permutation $\sigma:\NatZer\to\NatZer$  such that $S_i$ and
$S'_{\sigma(i)}$ represent the same element in $\dot{\mathcal{S}}$.
We call 
 $\dot{\mathcal{T}}$ the set of equivalence classes.

With the notation of Lemma \ref{Jordan_decomposition},
let $K$ be ordered so that the $\mathcal{L}_2(Y_k)$, $k\in K$, are nonincreasing,
and append to the sequence $Y_k$ infinitely many copies of the empty 
set if $K$ is finite.  We define a particular element 
$S=(S_0,S_1,S_2,\cdots)$ of $\mathcal{T}$ as follows. Let 
$S_0=(\emptyset,\emptyset,\cdots)$ if $\mathcal{L}_2(E)<\infty$, otherwise let
$S_0$ be a representative in $\mathcal{S}$ of the $M$-connected components of $\R^2\setminus Y_0$.
Let further $S_k$, for $k\geq1$,
be a representative in $\mathcal{S}$ of the $M$-connected components of $\R^2\setminus Y_k$. Note that $(\R^2\setminus\mathfrak{U}_{S_0},\R^2\setminus\mathfrak{U}_{S_1},
\R^2\setminus\mathfrak{U}_{S_2},\cdots)$ is equal to
$(Y_0,Y_1,\cdots)$ if $\mathcal{L}_2(E)=\infty$ and to
$(\R^2,Y_1,\cdots)$ if $\mathcal{L}_2(E)<\infty$, so it is an element of
$\mathcal{S}$.
Hence, $S:=(S_0,S_1,S_2,\cdots)$ belongs to 
$\mathcal{T}$, and  if for $k\geq0$ 
we write $S_k=(S_{k,0},S_{k,1},\cdots)$, where the $S_{k,j}$ are sets of 
finite perimeter mod-$\mathcal{L}_2$ constitutive of $S_k\in\mathcal{S}$,
then: (i) for $k\geq1$ we have $S_{k,0}=\mathrm{ext }(\Gamma_k^+)$ while
$(S_{k,j})_{j\geq1}$ enumerates the $(\mathrm{int}(\Gamma_j^-))_{j\in \tilde I_k}$ in nonincreasing $\mathcal{L}_2$-measure,
with infinitely many copies of the empty set appended when $\tilde I_k$
is finite;
(ii) if $\mathcal{L}_2(E)=\infty$ then  $(S_{0,j})_{j\in\NatZer}$ enumerates the $(\mathrm{int}(\Gamma_j^-))_{j\in  I_\infty}$ in nonincreasing $\mathcal{L}_2$-measure, 
with infinitely many copies of the empty set appended when $ I_\infty$
is finite, and if $\mathcal{L}_2(E)<\infty$
then $S_{0,j}=\emptyset$ for all $j$.
Altogether, the families 
$\{(\mathrm{ext }(\Gamma_k^+))_{k\in K}\}$,
$\{(\mathrm{int }(\Gamma_j^-))_{j\in J}\}$, padded with copies of the empty set
if needed and  arranged  in the previously described structure as entries
of the infinite array $(S_{k,j})$, $0\leq k,j\leq\infty$,
define some $S\in\mathcal{T}$. Of course, $S$ depends on the ordering
we chose to enumerate the $Y_k$ and the $M$-connected components
of the $\R^2\setminus Y_k$, if there are several orderings
making their $\mathcal{L}_2$-measures nonincreasing. However, the 
equivalence class $\dot{S}\in\dot{\mathcal{T}}$ is independent of 
such choices.

We orient the $\Gamma_k^+$ counterclockwise 
and the $\Gamma_j^-$ clockwise. This allows us to regard
$\Gamma_k^+$ (resp. $\Gamma_j^-$) as the image of a
unique  parametrized Jordan curve
$\bs{\gamma}_k^+$ (resp. $\bs{\gamma}_j^-$).
We shall
identify $\mathrm{ext }(\Gamma_k^+)$ (resp. 
$\mathrm{int } (\Gamma_j^-)$) with $\bs{\gamma}_k^+$ (resp. $\bs{\gamma}_j^-$),
and we regard the emptyset  as a degenerate curve
reducing to a point.
This way, the sets $S_{k,j}$ defined above can be viewed as
parametrized rectifiable Jordan curves, and the latter can in turn be considered
as measures  if we regard a parametrized Jordan curve $\bs{\gamma}$ as the 
member $\bb{R}_{\bs{\gamma}}$ of $\mathcal{M}(\R^2)^2$ defined 
in \eqref{Rgamma}. Here, a degenerate curve has
constant parametrization and therefore corresponds  to the zero measure.   Recall also from Section \ref{sec:decss} that if $\bs{\gamma}$ is
 a parametrized rectifiable Jordan curve of length $L>0$ and
 $\tilde{\bs{\gamma}}:\R\to\R^2$ is the periodic extension of $\bs{\gamma}$,
then $\tilde{\bs{\gamma}}$ defines {\it via} \eqref{SmirnovA} the  elementary solenoid  ${\mathbf T}_{\tilde{\bs{\gamma}}}=\bb{R}_{\bs{\gamma}}/L$, and in the degenerate case where $\bs{\gamma}$ reduces to a point, we define ${\mathbf T}_{\tilde{ \bs{\gamma}}}=0$.

\begin{prop}
\label{convmcur}
Let $\phi\in\dot{BV}(\R^2)$ and $E_t$ be as in \eqref{defsl}. For $t$ such that
$E_t$ has finite perimeter, let $S^t:=(S_0^t,S_1^t,S_2^t,\cdots)\in\mathcal{T}$
be constructed as indicated above from the curves 
$\{(\Gamma_k^+)_{k\in K}\}$,
$\{(\Gamma_j^-)_{j\in J}\}$ obtained by applying Lemma
\ref{Jordan_decomposition} to $E_t$. Write  $S_k^t=(S_{k,0}^t,S_{k,1}^t,\cdots)$ for the components of $S_k^t\in\mathcal{S}$. As we just explained, 
each $S^t_{k,j}$ identifies with a parametrized Jordan curve $\bs{\gamma}^t_{k,j}$ with image $\Gamma_{k,j}^t$. 
To each $\eta>0$, there is a $\sigma$-compact set $\Sigma_\eta\subset\R$, with
$\mathcal{L}_1(\R\setminus\Sigma_\eta)<\eta$, such that:
\begin{enumerate}
\item[(i)] For each $t\in\Sigma_\eta$, it holds that $E_t$ has finite perimeter.
\item[(ii)] For each sequence $(t_m)_{m\geq1}$ in $\Sigma_\eta$ converging to 
$t_0\in\Sigma_\eta$, there is a subsequence $t_{m_\ell}$ such that
$\bb{R}_{\bs{\gamma}_{k,j}^{t_{m_\ell}}}$ converges weak-$*$, as $\ell\to\infty$ for fixed  $k,j$, to $\bb{R}_{\bs{\gamma}_{k,j}}$ for some
parametrized 
Jordan curve $\bs{\gamma}_{k,j}$ with image $\Gamma_{k,j}$.
Moreover, $(\bs{\gamma}_{k,j})_{k,j\in\NatZer}$ is equivalent to 
$S^{t_0}$ 
in $\dot{\mathcal{T}}$.
\item[(iii)] We have the limiting relation 
 $\lim_\ell\mathcal{H}^1(\Gamma^{t_{m_\ell}}_{k,j})=\mathcal{H}^1(\Gamma_{k,j})$ 
for each $(k,j)$.
\item[(iv)] It holds that ${\mathbf T}_{\tilde{\bs{\gamma}}_{k,j}^{t_{m_\ell}}}$ converges weak-$*$, as $\ell\to\infty$ for fixed  $k,j$, to
${\mathbf T}_{\tilde{\bs{\gamma}}_{k,j}}$.
\end{enumerate}
\end{prop}

\begin{proof} We adopt the notation of Lemma \ref{Jordan_decomposition} 
for the decomposition of $E^t$, only with an extra-superscript $t$ to keep track of the level; {\it e.g.}, as in $Y_k^t$. By Lemma \ref{poin} 
we may assume that $\phi\in L^2(\R^2)$, so that
$\mathcal{L}_2(E_t)=\infty$ when $t<0$ and 
$\mathcal{L}_2(E_t)<\infty$ when $t>0$. To avoid bookkeeping 
with indices, we give the proof when $t_0<0$  only,
as the case where $t_0>0$ is similar but simpler. 
Thus, we may assume that 
$t_m<0$ for all $m$.
With $\Sigma_\eta$ as in Proposition~\ref{convmcc},
 we know from the latter that $(i)$ holds and that, for some subsequence 
$t_{m_i}$, 
the $Y_k^{t_{m_i}}$ converge,  locally in measure for fixed $k$ 
as $i\to\infty$, to some $F_k$ such that $(F_k)_{k\geq0}$ is equivalent to
$(Y_k^{t_0})_{k\geq0}$ in $\dot{\mathcal{S}}$. Moreover,
we know from $(iii)$ of this proposition that 
$\lim_i\mathcal{L}_n((Y_k^{t_{m_i}}\setminus  F_k)\cup(F_k\setminus Y^{t_{m_i}}_k))=0$
and that $\lim_i\mathcal{P}(Y^{t_{m_i}}_k)=\mathcal{P}(F_k)$ for each $k$.
Equivalently, the $\R^2\setminus Y_k^{t_{m_i}}$ 
converge  locally in measure to $\R^2\setminus F_k$
as $i\to\infty$ and
$\lim_i\mathcal{L}_n((\R^2\setminus Y_k^{t_{m_i}})\setminus  (\R^2\setminus F_k)\cup((\R^2\setminus F_k)\setminus (\R^2\setminus Y^{t_{m_i}}_k)))=0$,
while  $\lim_i\mathcal{P}(\R^2\setminus Y^{t_{m_i}}_k)=\mathcal{P}(\R^2\setminus F_k)$ for each $k$. This is all we need to apply the proof 
of Proposition \ref{convmcc} to $\R^2\setminus Y_k^{t_{m_i}}$ instead of
$E^{t_{m_i}}$,
to the effect that for each $k\geq0$ there is a subsequence $t^{(k)}_{m_{i_\ell}}$
of $t_{m_i}$ such that $S_{k,j}^{t^{(k)}_{m_{i_\ell}}}$ converges 
locally in measure to some $C_{k,j}$, where $(C_{k,j})_{j\in\NatZer}$
is equivalent to $S_k^{t_0}$ in $\dot{\mathcal{S}}$. Using a diagonal argument, we can  make $t^{(k)}_{m_{i_\ell}}$
independent of $k$ and we rename it as
$t_{m_\ell}$ for simplicity. 
By construction, we may write for $k=0$ or $j\geq1$ that
$C_{k,j}=\mathrm{int }(\Gamma_{k,j})$ 
mod-$\mathcal{L}_2$ with $\Gamma_{k,j}=\Gamma_l^{-,t_0}$ for some $l=l(k,i)$,
while for $k\geq1$ we have $C_{k,0}=\mathrm{ext }(\Gamma_{k,0})$ 
mod-$\mathcal{L}_2$ with $\Gamma_{k,0}=\Gamma_k^{+,t_0}$.
Moreover, we know from the proof of 
Proposition \ref{convmcc} point $(iii)$
that $\lim_\ell\mathcal{P}(S_{k,j}^{t_{m_\ell}})=\mathcal{P}(C_{k,j})$
or, equivalently, that
$\lim_\ell\mathcal{H}^1(\Gamma^{t_{m_\ell}}_{k,j})=\mathcal{H}^1(\Gamma_{k,j})$, which proves $(iii)$. Now, if we let $\bs{\gamma}_{k,j}$ be a parametrization of $\Gamma_{k,j}$ and $\bs{\gamma}^{t_{m_\ell}}_{k,j}$ be a parametrization
of $\Gamma_{k,j}^{t_{m_\ell}}$, 
oriented  clockwise 
for $j\geq1$ or $k=0$ and counterclockwise 
when $j=0$ and $k\geq1$,
it follows from \eqref{dermeasd} and a
mollification argument,
since $\mathcal{H}^1(\Gamma_{k,j}^{t_{m_\ell}})$ is bounded
for fixed $k,j$ as $\ell\to\infty$,
that $\bs{\gamma}^{t_{m_\ell}}_{k,j}$ converges weak-$*$ to
$\bs{\gamma}_{k,j}$. 
 Applying pointwise a rotation 
by $\pi/2$, this is tantamount
to say that $\bb{R}_{\bs{\gamma}^{t_{m_\ell}}_{k,j}}$ converges
weak-$*$ to $\bb{R}_{\bs{\gamma}_{k,j}}$, thereby proving $(ii)$.
 Note that when $\mathcal{H}^1(\Gamma_{k,j})>0$, then   the assertion of item $(iv)$ follows immediately from items $(ii)$ and $(iii)$.   
 Now suppose $\mathcal{H}^1(\Gamma_{k,j})=0$.   Let $\mathbf{f}\in C_c(\R^2)^2$ and   $\epsilon>0$. By uniform continuity, there is some $\delta>0$ such that $|\mathbf{f}(x)-\mathbf{f}(y)|<\epsilon$ whenever $|x-y|<\delta$.  Let $L_\epsilon$ be such that ${\rm diam } (\Gamma^{t_{m_\ell}}_{k,j})<\delta$ for $\ell\ge L_\epsilon$.    Since $\bb{R}_{\bs{\gamma}^{t_{m_\ell}}_{k,j}}$ 
 is divergence free for all $j,k,\ell$, it annihilates constant functions.    Thus, for $x_\ell\in \Gamma^{t_{m_\ell}}_{k,j}$, we have
 $$
| \langle \mathbf{f},\bb{R}_{\bs{\gamma}^{t_{m_\ell}}_{k,j}}\rangle | =  | \langle \mathbf{f}-\mathbf{f}(x_\ell),\bb{R}_{\bs{\gamma}^{t_{m_\ell}}_{k,j}}\rangle | \le \epsilon \mathcal{H}^1(\Gamma^{t_{m_\ell}}_{k,j}),
 $$
 which verifies $(iv)$ in this case.
 \end{proof}


In the discussion before Proposition \ref{convmcur},
we identified the curves $\{ \Gamma^+_k\}_{k\in K}$ and $\{\Gamma^-_j \}_{j\in J}$
forming the measure-theoretical boundary of a set of finite perimeter
with (the equivalence classes of) an element of $\mathcal{\mathcal{T}}$ 
of the form $S=(S_{k,j})_{k,j\in\NatZer}$ where  $S_{k,j}$ is (the interior
of) a (possibly degenerate) Jordan curve oriented  clockwise 
for $j\geq1$ or $k=0$, while
$S_{k,0}$ is (the exterior of) a Jordan curve oriented  counterclockwise 
when $k\geq1$.
We let $\mathcal{C}\subset\mathcal{T}$ denote 
the set of such elements, and $\dot{\mathcal{C}}$ the set of equivalence classes.
Recalling that $\mathcal{M}(\R^2)^2$ equipped with the weak-$*$ topology 
is a metric space, say with distance $d_w$, 
we endow $\mathcal{C}$ with the distance $d_{\mathcal{C}}((S_{k,j}),(S'_{k,j})):=\sup_{k,j}
d_w(S_{k,j},S'_{k,j})$ and $\dot{\mathcal{C}}$ with the quotient topology.
We also find it more convenient to enumerate with a single index 
the curves $S_{k,j}$ 
constitutive of $S\in\mathcal{C}$:
for this, we choose a bijection $\sigma:\NatZer^2\to\NatZer$
and we write $\Gamma_{\sigma(k,j)}:=S_{i,j}$. The orientation of
the corresponding parametrized curve $\bs{\gamma}_{\sigma(i,j)}$ will
depend on the choice of $\sigma$, and so do  the permutations defining 
equivalence classes in $\dot{\mathcal{C}}$,
but our results  will not.
We can now state the representation theorem for divergence-free measures in the plane:

\begin{theorem} 
\label{rep}
Let $\bs{\nu}\in \mathcal{M}(S)^2$  be divergence-free in $\R^2$.
Then, there exists $G\subset\R$ with $\mathcal{L}_1(\R\setminus G)=0$
such that, for $t\in G$,
there is a countable collection  of (possibly degenerate) parametrized rectifiable
Jordan curves $\{\bs\gamma_n^t\}_{n\in\NatZer}$ with
images $\Gamma_n^t$ such that:
\begin{enumerate}
	\item[(i)] 
the $(\Gamma_n^t)_{n\in\mathbb{N}}$ are
disjoint up to a set of $\mathcal{H}^1$-measure zero and
$\Gamma_n^t\subset \text{\rm supp}\,\bs{\nu}$ for each $n$;
\item[(ii)] the union
 $\bigcup_n \Gamma_n^t$ is, up to a set of $\mathcal{H}^1$-measure zero,
the measure-theoretical boundary $\partial_M\Omega(t)$ of
a set $\Omega(t)\subset\R^2$ of finite perimeter; 
\item[(iii)] $\Omega(t_1)\supset\Omega(t_2)$ if $t_1<t_2$,
and the mapping $t\mapsto (\bs{\gamma}_n^t)_{n\in\NatZer}$ from $\R$ to $\dot{\mathcal{C}}$ is approximately continuous for a.e. $t$;
\item[(iv)]
For any Borel set $B\subset\R^2$, $\bb{g}\in L^1[d|\bs{\nu}|]^2$ and $h\in L^1[d|\bs{\nu}|]$, it holds that
\begin{equation}\label{eqrep}
\bs{\nu}(B) = \int_\R\sum_{n\in \NatZer}\left(\int_{B} \boldsymbol{\tau}_n^t \, d\left(\mathcal{H}^1\mathcal{b} \Gamma_n^t\right)\right) \, dt,
\end{equation}
where $\bs{\tau}_n^t=(\boldsymbol{\gamma}_n^t)'/|(\bs{\gamma}_n^t)'|$ 
is the unit tangent vector field to $\Gamma_n^t$ oriented by
$\bs{\gamma}_n^t$,
\begin{equation}\label{eqrepabs}
|\bs{\nu}|(B)=\int_\R \mathcal{H}^1(\partial_M\Omega(t)\cap B)dt=\int_\R \left(\sum_{n\in \NatZer}\mathcal{H}^1(\Gamma_n^t\cap B)\right)dt,
\end{equation}
\begin{equation}\label{eqrepint}
\int\bb{g}\cdot d\bs{\nu} = \int_\R\sum_{n\in \NatZer}\left(\int \bb{g}\cdot \boldsymbol{\tau}_n^t \, d\left(\mathcal{H}^1\mathcal{b} \Gamma_n^t\right)\right) \, dt,
\end{equation}
and 
\begin{equation}\label{eqrepint2}
\int h d|\bs{\nu}| = \int_\R\sum_{n\in \NatZer}\left(\int h \, d\left(\mathcal{H}^1\mathcal{b} \Gamma_n^t\right)\right) \, dt,
\end{equation}
where  the inner integrals on   the right handsides of \eqref{eqrepint} and \eqref{eqrepint2} are well defined for a.e. $t\in \R$.
\item[(v)] 
The set $J:=\bigcup_{\stackrel{t_1\neq t_2\in G}{n_1,n_2\in\mathbb{N}}}\Gamma_{n_1}^{t_1}\cap\Gamma_{n_2}^{t_2}$
is 1-rectifiable in $\R^2$ and $\bs{\nu}\mathcal{b}_J$ is absolutely continuous with respect to $\mathcal{H}^1$; 
for a.e. $t\in G$, $\bb{u}_{\bs{\nu}}(x)=\bs{\tau}_n^t(x)$ for $\mathcal{H}^1$-a.e.   $x\in J\cap\partial_M\Omega(t)$.
More generally, it holds for a.e. $t\in G$ and every $n\in\NatZer$ that $\bb{u}_{\bs{\nu}}(x)=\bs{\tau}_n^t(x)$ for $\mathcal{H}^1$-a.e. $x\in\Gamma^t_n$.
\end{enumerate}
\end{theorem}
\begin{proof}
By Lemma~\ref{2Dsole}, we have $\bs{\nu}(B)=\mathfrak{R}\grad \phi(B)$ for some $\phi\in \dot{BV}(\R^2)$.
 Defining  $E_t$ as in \eqref{defsl},
we get from Lemma \ref{coarea} that it has finite perimeter for a.e. $t$.
We let $G$ be the set of such $t$,
and for $t\in G$ we let $\{\bs{\gamma}_n^t\}_{n\in \NatZer}$
 be a representative in $\mathcal{C}$ of the element of $\dot{\mathcal{C}}$
corresponding to the family of curves $(\bs{\gamma}_{k,j}^t)\in\mathcal{T}$ 
appearing in Proposition
\ref{convmcur}, see discussion after the proof of that proposition.
If we set  $\Omega(t)=E_t$,
then $(ii)$ and the first assertion in $(i)$ come from Lemma
\ref{Jordan_decomposition}, the first assertion in $(iii)$ is obvious  
and the second on approximate continuity follows from Proposition
\ref{convmcur} much like Theorem  \ref{approxceq} did from 
Proposition \ref{convmcc}. Recalling definition
\eqref{Rgammappv}, we see that Lemma \ref{coarea} and the 
 remark after \eqref{dermeasd}
together  imply 
$(iv)$, where it should be noted that equations \eqref{eqrep} through \eqref{eqrepint2}
only depend on the equivalence class of
$\{\bs{\gamma}_n^t\}_{n\in \NatZer}$ in $\dot{\mathcal{C}}$.
Since  \eqref{eqrepabs}
implies that $\mathcal{H}^1(\Gamma_n(t)\setminus\supp\bs{\nu})=0$ for a.e. 
$t\in\R$ the second half of $(i)$ holds.

	Observing that
$\bigcup_{n\in\mathbb{N}}\Gamma_n^t=\partial_M E_t$ mod-$\mathcal{H}^1$, we see for each $t\in G$ that 
every $x\in J$ lies in $\partial_M (\R^2\setminus E_{t_1})\cap\partial_M E_{t_2}$ for some $t_1<t_2$. 
Remembering the definitions in \eqref{jump}, this implies that, for every $x\in J$, $\phi^{\inf}(x)\leq t_1 < t_2 \leq \phi^{\sup}(x)$.
Hence, by Lemma \ref{jump_lemma}, $J\subset J(\phi)$ and the first two assertions of $(v)$ follow.
Now, evaluating $\|\bs{\nu}\|$ with \eqref{eqrepabs} and integrating \eqref{eqrepint} against $\bb{u}_{\bs{\nu}}$ we get,
$$
\int_\R \left(\sum_{n\in \NatZer}\mathcal{H}^1(\Gamma_n^t)\right)dt
= \|\bs{\nu}\| 
= \int\bb{u}_{\bs{\nu}}\cdot d\bs{\nu} = \int_\R\sum_{n\in \NatZer}\left(\int \bb{u}_{\bs{\nu}}\cdot \boldsymbol{\tau}_n^t \, d\left(\mathcal{H}^1\mathcal{b} \Gamma_n^t\right)\right) \, dt,
$$
and noting that $\bb{u}_{\bs{\nu}}\cdot\bs{\tau}_n^t\leq1$, with equality only when $\bb{u}_{\bs{\nu}}=\bs{\tau}_n^t$, gives us the last assertion of $(v)$.
\end{proof}

Decomposition \eqref{eqrep}-\eqref{eqrepabs} is a 
	special case of \eqref{measonb}, as we now show.
	\begin{prop}
\label{mesi}
	Let $\bs{\nu}\in \mathcal{M}(S)^2$  be divergence-free in $\R^2$, with
$G$, $\{\bs\gamma_n^t\}_{n\in\NatZer}$ and $\Gamma_n^t$ as in Theorem \ref{rep}.
	Take $\tilde{\bs{\gamma}}_n^t$ to be the periodic extension to $\R$ of 
	$\bs{\gamma}_n^t$. 
	If we set
	\begin{equation}\label{eq:intrep}
	\rho(\mathfrak{B}):=\int_{\R}\sum_{n\in\mathbb{N}}\mathcal{H}^1(\Gamma_n^t)\delta_{\bb{T}_{\tilde{\bs{\gamma}}_n^t}}(\mathfrak{B})dt \quad\text{for every Borel }\mathfrak{B}\subset \mathfrak{S}(\R^2),
	\end{equation}
	then the integral exists and $\rho$ defines a Borel measure on $\mathfrak{S}(\R^2)$ such that \eqref{measonb} holds with $\bs{\mu}=\bs{\nu}$.
\end{prop}
\begin{proof}

As in Section~\ref{subsec:curves}, let
 $\mathcal{B}_1$ denote the unit ball in  $\mathcal{M}(\R^2)^2$
with the weak-$*$ topology.
Let $\mathfrak{B}\subset \mathcal{B}_1$ be Borel,
and $F:\R\to\R$ denote the integrand in \eqref{eq:intrep}.
Recall from Proposition \ref{convmcur} the $\sigma$-compact sets $\Sigma_\eta$ 
such that $\mathcal{L}_1(\R\setminus\Sigma_\eta)<\eta$ for $\eta>0$.
By the Borel-Cantelli lemma, $\Sigma_0:=\bigcup_{j\in\NatOne}\Sigma_{1/j^2}$ is   $\sigma$-compact   such that $\mathcal{L}_1(\R\setminus\Sigma_0)=0$. Hence, if  $F|_{\Sigma_\eta}$ is a Borel function, then $F$ is also
Borel.
We will show that $F|_{\Sigma_\eta}$ is Borel by writing it as a composition of Borel functions.

Let $\mathcal{Q}:=\ell_1(\NatZer)\times\mathcal{B}_1^\NatZer$ where $\mathcal{B}_1^\NatZer$ is given the product topology, and   $\dot{\mathcal{Q}}$ denote the quotient space under the relation $(a_n,\bs{\mu}_n)\sim(b_n,\bs{\nu}_n)$ if and only if there is a bijection $\sigma:\NatZer\to\NatZer$ such that $b_{\sigma(n)}=a_n$ and $\bs{\nu}_{\sigma(n)}=\bs{\mu}_n$.
We endow $\dot{\mathcal{Q}}$ with the quotient topology.
Define $f_1:\Sigma_\eta\to \dot{\mathcal{Q}}$ by $f_1(t):=[(\mathcal{H}^1(\Gamma^t_n),\bb{T}_{\tilde{\bs{\gamma}}^t_n})]$, 
where the bracket represents the equivalence class; note that indeed
$\sum_n\mathcal{H}^1(\Gamma_n^t)<\infty$, because this sum is $\mathcal{P}(E_t)$
which is uniformly bounded on $\Sigma_\eta$ by construction, see
proof of Proposition \ref{convmcc}.
By points $(iii)$ and $(iv)$ of Proposition \ref{convmcur}, $f_1$ is continuous (observe that $\sim$ takes quotient by all permutations,
not just those used to define $\dot{\mathcal{T}}$, 
which does not affect continuity).
Now let $\tilde{f}_2:\mathcal{Q}\to\R$ be defined by $\tilde{f}_2(a_n,\bs{\mu}_n):=\sum_na_n\chi_\mathfrak{B}(\bs{\mu}_n)$.
Clearly, $\tilde{f}_2$ is Borel since it is the limit of Borel functions, and since it is invariant under permutations on $n$
the quotient map $f_2:\dot{\mathcal{Q}}\to\R$ is well-defined and Borel.
%

Altogether, $F|_{\Sigma_\eta}=f_2\circ f_1$, is Borel and so is $F$.
Hence, since $F$ is nonnegative and its integral is bounded by $\int_{\R}\sum_{n\in\mathbb{N}}\mathcal{H}^1(\Gamma_n^t)=\|\bs{\nu}\|$, the set function $\rho$ given by \eqref{eq:intrep} defines a Borel measure on  $\mathcal{B}_1$.
By restriction ${\rho}$ defines a Borel measure on $\mathfrak{S}(\R^2)$. 
Finally we will show that the left equation of \eqref{measonb} holds, the proof for the right one is similar. 
Let $B\subset\R^2$ be Borel, $\{a_i\}_{i=0}^n$ be a partition of $[-1,1]$, $(T_1,T_2)$ be the components of $\bb{T}$, for $i<n$ and $j=1,2$, $\mathfrak{A}_i^j:=\{\bb{T}\in\mathfrak{S}(\R^2)|a_{i-1}\leq{T}_j(B)<a_i\}$, $\mathfrak{A}_n^j:=\{\bb{T}\in\mathfrak{S}(\R^2)|a_{n-1}\leq{T}_j(B)\leq1\}$, $M_j=\sum_ia_i\rho(\mathfrak{A}_i^j)$ and $m_j=\sum_ia_{i-1}\rho(\mathfrak{A}_i^j)$.
Then
$$
m_j=\sum_ia_{i-1}\int_{\R}\sum_{n\in\mathbb{N}}\mathcal{H}^1(\Gamma_n^t)\delta_{\bb{T}_{\tilde{\bs{\gamma}}_n^t}}(\mathfrak{A}_i^j)dt
\leq \int_{\R}\sum_i\sum_{n:\bb{T}_{\tilde{\bs{\gamma}}_n^t}\in\mathfrak{A}_i}\mathcal{H}^1(\Gamma_n^t){({\bb{T}}_{\tilde{\bs{\gamma}}_n^t}})_j(B)dt,
$$
where the right hand-side of this equation is equal to $(\bs{\nu}(B))_j$ in view of \eqref{eqrep}, Fubini's theorem and the fact that the $\mathfrak{A}_i^j$'s form a partition of $\mathfrak{S}(\R^2)$.
Analogously $(\bs{\nu}(B))_j\leq M_j$, hence, taking the limit as $\max\{a_i-a_{i-1}\}\to0$ and using $\rho(\mathfrak{S}(\R^2))=\|\bs{\nu}\|<\infty$, we get \eqref{measonb}.
\end{proof}

Theorem \ref{rep} $(iii)$ asserts approximate 
continuity of 
$\partial_M\Omega(t)$ with respect to $t$ in the weak-$*$ sense. 
Still, the $\Omega(t)$ could all have different 
topologies as can be seen from the following example. 

\begin{exa}\label{dif_topology}
We will generate a $BV$ function $\varphi_\infty$, valued in $[0,1]$,
whose suplevel sets $E_t$ all have different topologies.
Then, $\bs\nu:=\mathfrak{R}\grad\varphi_\infty$ is divergence-free and
$\Omega(t)=E_t$ in Theorem \ref{rep}, thereby 
yielding an example with the aforementioned property.

We construct $\varphi_\infty$
as the limit of a bounded increasing sequence $(\phi_m)$ of $BV$ functions.
Let us first define a family of sets of finite perimeter that we will use 
to construct the $\phi_m$.
For any two integers $m$ and $n$ such that $m\geq0$ and $1\leq n\leq 2^m$, define the set $b(n,m)\subset\R^2$ to be the closed ball around the point $(n,m)$ with perimeter $2^{-2m-1}$ (thus,  radius $2^{-2m-2}/\pi$) minus
$2^m$ pairwise disjoint nonempty open balls contained in this closed ball.
We pick the sum of the perimeters of this $2^m$ open balls 
to be strictly less than $2^{-2m-1}$.
Note that the $b(n,m)$ are pairwise disjoint.
Define  $\varphi_0:=\frac{1}{2}\chi_{b(1,0)}$ and, for $m>0$, $\varphi_m:=\varphi_{m-1} + \sum_{k=1}^{2^m}\frac{2k-1}{2^{m+1}}\chi_{b(k,m)}$.
Then $\|\grad\varphi_0\|_{TV}<1/2$, moreover for $m>0$:
\begin{align*}
\|\grad\varphi_m\|_{TV} &= \|\grad\varphi_{m-1}\|_{TV} + 
			   \sum_{k=1}^{2^m}\frac{2k-1}{2^{m+1}}\|\grad\chi_{b(k,m)}\|_{TV}	\\
			&< \|\grad\varphi_{m-1}\|_{TV} + 
			\sum_{k=1}^{2^m}\frac{2k-1}{2^{m+1}}(2^{-2m-1}+2^{-2m-1})	\\
			&= \|\grad\varphi_{m-1}\|_{TV} + \frac{2^{2m}}{2^{3m+1}},
\end{align*}
and hence, $\|\grad\varphi_m\|_{TV}<1$ for every $m$. 
Thus, $\varphi_\infty$, the pointwise limit of the nondecreasing sequence of 
functions $\{\varphi_m\}_m$, is a BV function (see \cite[Theorem 5.2.1]{Ziemer}).

Now, for $m$, $n$, $p$ and $q$ some integers such that $1\leq n\leq 2^m$ and $1\leq p\leq 2^q$, it is clear that $b(n,m)$ is topologically equivalent to $b(p,q)$ if and only if $q=m$.
Hence,  with the notation of Theorem \ref{rep}, we see that
given $s,t\in(0,1)$, the sets $\Omega(t)$ and $\Omega(s)$ 
can be topologically equivalent only if they contain, for each fixed $m$, the same number of sets from the family $\{b(n,m)\}_{n=1}^{2^m}$. 
However if $s<t$ then there exist two positive integers $m$ and $n$ such that $s<\frac{2n-1}{2^{m+1}}<t$,  thus $b(n,m)\subset\Omega(s)\setminus\Omega(t)$ and therefore $\Omega(t)$ is not topologically equivalent to $\Omega(s)$.
\end{exa}

\section{Applications to Inverse Magnetization Problems}
\label{AIP}

\subsection{Solutions to Extremal Problem 1}

For $\bs{\mu}, \bs{\nu}\in \mathcal{M}(\R^3)$ with
${\mathbf f}_{\bs\mu}$ to denote the Radon-Nikodym derivative of $\bs \mu$ 
with respect to $|\bs \nu|$, we define for    $|\bs \nu|$-a.e. $x$:
\begin{equation}
\label{wmu}
{\mathbf w}^{\bs \nu}_{\bs \mu}(x):=\begin{cases} \frac{{\mathbf f}_{\bs\mu} (x)}{|{\mathbf f}_{\bs\mu} (x)|}, & {\mathbf f}_{\bs\mu} (x)\neq 0,\\
{\mathbf u}_{\bs \nu}(x),& {\mathbf f}_{\bs\mu} (x)= 0.\end{cases}
\end{equation}

We put  $E= {\mathbf f}^{-1}_{\bs\mu}(0)$  and observe that
\begin{equation} \label{eq:wmunuInt}
	\int  {\mathbf w}^{\bs \nu}_{\bs{\mu}}\cdot  d\bs{\nu} =\int_{E^c} {\mathbf w}^{\bs \nu}_{\bs{\mu}}\cdot {\mathbf u}_{\bs{\nu}}\, d| \bs{\nu}| +|\bs{\nu}|(E). 
	\end{equation}
The next lemma provides a variational characterization of solutions to  Extremal Problem 1.

\begin{lemma}\label{var_equiv} Let $S\subset \R^3$ be closed and suppose 
	$\bs{\mu}, \bs{\nu}\in \mathcal{M}(S)^3$, with ${\mathbf w}^{\bs \nu}_{\bs{\mu}}$ and $E$ as above.   Then 
	\begin{equation}\label{var_equiv1}\|\bs{\mu}  \|_{TV}\le \|\bs{\mu}+ t \bs{\nu}\|_{TV}, \text{ for every } t > 0,
	\end{equation}
	if and only if
	\begin{equation} \label{eq:var_equiv}
	\int  {\mathbf w}^{\bs \nu}_{\bs{\mu}}\cdot  d\bs{\nu}  \ge 0. 
	\end{equation}
	Hence, $\|\bs{\mu}\|_{TV}=M_S(\bs{\mu})$ if and only if \eqref{eq:var_equiv} holds for every $S$-silent $\bs{\nu}\in \mathcal{M}(S)^3$.
The inequality \eqref{var_equiv1} is strict for every $t>0$  if the inequality \eqref{eq:var_equiv} is strict.
\end{lemma}
\begin{proof}
	Let   $\bs{\mu}_s$ denote the singular part
 of $\bs{\mu}$ with respect to $|\bs{\nu}|$.  Then, for $\epsilon>0$, 
	\begin{equation}\begin{split}\label{pf:eqvar0}
	\|\bs{\mu}+\epsilon \bs{\nu}\|_{TV}&=\int |{\mathbf f}_{\bs{\mu}} +\epsilon {\mathbf u}_{\bs{\nu}}|\, d|\bs{\nu}| +\|\bs{\mu}_s\|_{TV} \\
	&=\int_{E^c} |{\mathbf f}_{\bs{\mu}} +\epsilon {\mathbf u}_{\bs{\nu}}|\, d|\bs{\nu}|+\epsilon|\bs{\nu}|(E)+\|\bs{\mu}_s\|_{TV} \\
	&=\|\bs{\mu} \|_{TV}+\epsilon \left(\int_{E^c} {\mathbf w}^{\bs \nu}_{\bs{\mu}}\cdot {\mathbf u}_{\bs{\nu}}\, d|\bs{\nu}|+|\bs{\nu}|(E)\right) +o(\epsilon) \\
	&=\|\bs{\mu} \|_{TV}+\epsilon \int  {\mathbf w}^{\bs \nu}_{\bs{\mu}}\cdot  d\bs{\nu}  +o(\epsilon),
	\end{split}
	\end{equation}
	where the above used that  for $\mathbf{a},\mathbf{b}\in\R^3$, $\mathbf{a}\neq 0$ and $|\mathbf{b}|=1$ (with $\mathbf{a}={\mathbf f}_{\bs{\mu}}$ and $\mathbf{b}= {\mathbf u}_{\bs{\nu}}$),
	$$
	|{\mathbf a}+\epsilon {\mathbf b}|
	=|{\mathbf a}|\left(1+2\epsilon \frac{\mathbf{ a}\cdot \mathbf {b}}{|{\mathbf a}|^2}+\epsilon^2 \frac{|{\mathbf b}|^2}{|{\mathbf a}|^2} \right)^{1/2}
	=|\mathbf a |+\epsilon \frac{\mathbf a}{|\mathbf a|}\cdot \mathbf b +\frac{1}{|\mathbf{a}|}\mathcal{O}(\epsilon^2),
	$$
 together with 
	$|\bs{\nu}|(\{ \mathbf{x}: 0<|{\mathbf f}_{\bs{\mu}}\left(\mathbf{x})|<\epsilon\}\right) =o(1)$ as   $\epsilon\to 0$.
	Using the convexity of the TV-norm we have for $0<\epsilon\le 1$ and $t>0$: 
	\begin{equation*}
	\|\bs{\mu}+t\epsilon\bs{\nu}\|_{TV} =\|(1-\epsilon)\bs{\mu}+\epsilon(\bs{\mu}+t\bs{\nu})\|_{TV} \le (1-\epsilon)\|\bs{\mu}\|_{TV}+\epsilon\|\bs{\mu}+t\bs{\nu}\|_{TV},
	\end{equation*}
	which implies
	\begin{equation}\label{pf:eqvar}
	t\frac{\|\bs{\mu}+\epsilon t\bs{\nu}\|_{TV}-\|\bs{\mu}\|_{TV}}{t\epsilon}\le \| \bs{\mu}+ t\bs{\nu}\|_{TV}-\|\bs{\mu}\|_{TV}.
	\end{equation}
	If \eqref{eq:var_equiv} holds, then it follows in view of \eqref{pf:eqvar0} (with $t\epsilon$ instead of $\epsilon$) that the limit of the left-hand side of \eqref{pf:eqvar} is nonnegative when $\epsilon\rightarrow 0^+$, which implies \eqref{var_equiv1}.
	Conversely, if \eqref{var_equiv1} holds then the left hand side of \eqref{pf:eqvar} is nonnegative and using \eqref{pf:eqvar0} we can take the limit as $\epsilon\rightarrow 0^+$ to obtain \eqref{eq:var_equiv}.    That the inequality \eqref{var_equiv1} is strict for every $t>0$ when the inequality \eqref{eq:var_equiv} is strict follows immediately from the above computations.
\end{proof}

We say that ${\bs{\mu}} \in \mathcal{M}(S)^3$ is {\em carried by} a set  if that set has full $|{\bs{\mu}}|$-measure; i.e., the complement has $|{\bs{\mu}}|$-measure zero.  Recall that a set $B\subset \R^n$ is {\em purely 1-unrectifiable} if $\mathcal{H}^1(E\cap B)=0$ for every 1-rectifiable set $E$.  Clearly a set of $\mathcal{H}^1$-measure zero is purely 1-unrectifiable.  

\begin{theorem}\label{Thm6.2}
	Let $S\subset \R^3$ be slender and closed and suppose  $\widetilde{\bs{\mu}} \in \mathcal{M}(S)^3$ is carried by a purely 1-unrectifiable set.  Then $\widetilde{\bs{\mu}}$ is strictly $TV$-minimal.  Moreover, if $\bs{\mu} \in \mathcal{M}(S)^3$  is $TV$-minimal on 
	$S$, then so is $ \bs{\mu}+\widetilde{\bs{\mu}}$. 
\end{theorem}

\begin{proof}
Since $S$ is slender,   any $S$-silent magnetization  $\bs{\nu}$ is divergence-free.  From the decomposition  \eqref{Smi1}, we then have that $\bs{\nu}$ and $\widetilde{\bs{\mu}}$ are mutually singular since the latter is carried by a purely 1-unrectifiable set, showing that 
$\widetilde{\bs{\mu}}$ is strictly $TV$-minimal.
	
Next suppose $\bs{\mu} \in \mathcal{M}(S)^3$ satisfies $\|\bs{\mu}\|_{TV}=M_S(\bs{\mu})$ and $\bs{\nu}\in\mathcal{M}(S)^3$ be $S$-silent.  Since $\bs{\nu}$ and $\widetilde{\bs{\mu}}$ are mutually singular, $ d\widetilde{\bs{\mu}}/d|\bs{\nu}|=0$  and thus,  recalling definition \eqref{wmu}, we see that ${\mathbf w}^{\bs \nu}_{\bs \mu}={\mathbf w}^{\bs \nu}_{\bs\mu+\widetilde{\bs{\mu}}}$, $|\bs{\nu}|$-a.e.  Lemma \ref{var_equiv}  then implies
 $\|\bs{\mu}+\widetilde{\bs{\mu}}\|_{TV}=M_S(\bs{\mu}+\widetilde{\bs{\mu}})$.
\end{proof}

The first assertion of Theorem~\ref{Thm6.2} sharpens   Theorem~2.6 of \cite{BVHNS}   stating that a magnetization  supported on a purely 1-unrectifiable set is strictly $TV$-minimal.   
In the case that $S$ is planar, this result can be strengthened by the following theorem.

\begin{theorem}\label{Thm6.3}
	Let $S\subset\R^2\times\{0\}$ be   closed  and suppose $\bs{\mu}$  is a magnetization carried by a Borel set $Z\subset S$ that satisfies
	\begin{equation}\label{EP1Cond}
	\mathcal{H}^1(\Gamma \cap Z) \le \mathcal{H}^1(\Gamma\setminus Z), 
	\end{equation}
	for any rectifiable Jordan curve $\Gamma\subset S$.
	Then $\bs{\mu}$ is $TV$-minimal on $S$.  If $\bs{\nu}\in \mathcal{M}(S)^3$ is $S$-silent 
	and $\|\bs{\mu}+\bs{\nu}\|_{TV}=\|\bs{\mu}\|_{TV}$, 
	then equality holds in \eqref{EP1Cond} when $\Gamma=\Gamma_n^t$ for almost every $t$  and every $n\in \NatZer$ in the loop decomposition of $\bs{\nu}$. 
	In particular, $\bs{\mu}$ is
	strictly $TV$-minimal on $S$ if the inequality \eqref{EP1Cond} is strict for every nondegenerate   $\Gamma\subset S$,  and then  $\bs{\mu}+\tilde{\bs{\mu}}$ is also strictly $TV$-minimal when $\tilde{\bs{\mu}}$ is carried by a purely 1-unrectifiable set. 
	\end{theorem}
\begin{proof}
	Let $\bs{\nu}$ be an  $S$-silent magnetization with ${\mathbf f}_{\bs\mu}$,  ${\mathbf w}^{\bs \nu}_{\bs{\mu}}$ as  in \eqref{wmu}, and loop decompositions $\{\Gamma_n^t\}$ and recall 
	$E= \bb{f}^{-1}_{\bs\mu}(0)$.
	Also let   $\bs{\mu}_s$ denote the singular part
	of $\bs{\mu}$ with respect to $|\bs{\nu}|$. 
	By Lemma \ref{2Dsole} $\bs{\nu}=(\bs{\nu}_T,0)$ where $\bs{\nu}_T\in\mathcal{M}(S)^2$ is divergence-free.
	For    $t\in\R$ and $n\in \NatZer$, let $\Gamma_n^t$ and $\bs{\tau}_n^t$ be as in Theorem \ref{rep} from the decomposition of $\bs{\nu}_T$. 

By assertion $(v)$ of Theorem~\ref{rep}, we know  for  a.e. $t\in\R$ and for every $n\in \NatZer$ that $\bb{u}_{\bs{\nu}}(x)=(\bs{\tau}_n^t(x),0)$ for $\mathcal{H}^1$-a.e.   $x\in \Gamma_n^t$.
	Note also that by $(iv)$ of Theorem \ref{rep}, $\bb{w}^{\bs \nu}_{\bs{\mu}}\cdot (\bs{\tau}_n^t,0)$ is $\mathcal{H}^1$-integrable on $\Gamma_n^t$ for every $n\in \NatZer$ and a.e. $t\in\R$.
	Now, for every such $t$,
	\begin{equation}\label{CarryIneq}
	\begin{split}	
	\int_{\Gamma_n^t}  {\mathbf w}^{\bs \nu}_{\bs{\mu}}\cdot (\bs{\tau}_n^t,0) \, d\mathcal{H}^1
	&=\int_{\Gamma_n^t\cap E^c} {\mathbf w}^{\bs \nu}_{\bs{\mu}}\cdot (\bs{\tau}_n^t,0) \, d\mathcal{H}^1 +\int_{\Gamma_n^t\cap E} {\mathbf u}_{\bs{\nu}}\cdot (\bs{\tau}_n^t,0) \, d\mathcal{H}^1 \\
	&=\int_{\Gamma_n^t\cap E^c} {\mathbf w}^{\bs \nu}_{\bs{\mu}}\cdot (\bs{\tau}_n^t,0) \, d\mathcal{H}^1 +\mathcal{H}^1(\Gamma_n^t\cap E) \\
	&\ge -\mathcal{H}^1(\Gamma_n^t\cap E^c) +\mathcal{H}^1(\Gamma_n^t\cap E) .	\end{split}\end{equation}
	
	From \eqref{eqrepint2} we have
	$$
	0=\int_{Z^c}|\bb{f}_{\bs{\mu}}| d\bs{|\nu|}	=  \int_{T_0} \sum_{n\in \NatZer}\left(\int_{Z^c}|\bb{f}_{\bs{\mu}}|\, d\left(\mathcal{H}^1\mathcal{b} \Gamma_n^t\right)\right) \, dt.
	$$
	Observing that   $|{\bf f}_{\bs \mu}(x)|>0$ for $x\in E^c$, the above equation implies that the $\mathcal{L}_1$-measure of
		 $$T_0:=\{t\in\R\ |\ \exists n\in\NatZer : \mathcal{H}^1(\Gamma_n^t\cap E^c\cap Z^c)\ne0\}$$ is zero; that is, $\mathcal{H}^1(\Gamma_n^t\cap E^c\cap Z^c)=0$ for a.e. $t$.
	Thus, by \eqref{CarryIneq} we get
	\begin{equation}\label{GammaZ1}
	\int_{\Gamma_n^t}  {\mathbf w}^{\bs \nu}_{\bs{\mu}}\cdot (\bs{\tau}_n^t,0) \, d\mathcal{H}^1
	\ge -\mathcal{H}^1(\Gamma_n^t\cap Z) +\mathcal{H}^1(\Gamma_n^t\setminus Z)\geq 0, 
	\end{equation}
	where  the last inequality follows from the condition \eqref{EP1Cond}.
	Therefore, by \eqref{eqrepint},
	\begin{equation}\label{GammaZ2}
	\int_{\R^2}  {\mathbf w}^{\bs \nu}_{\bs{\mu}}\cdot  d\bs{\nu} = \int_\R\sum_{n\in N^t}\left(\int_{\R^2} {\mathbf w}^{\bs \nu}_{\bs{\mu}}\cdot (\bs{\tau}_n^t,0) \, d\left(\mathcal{H}^1\mathcal{b} \Gamma_n^t\right)\right) \, dt 
	\ge 0,
	\end{equation} 
	and, hence,   Lemma \ref{var_equiv} gives us $\|\bs{\mu}  \|_{TV}\le \|\bs{\mu}+  \bs{\nu}\|_{TV}$.    
	Moreover, if there is a set of positive measure $E\subset\R$ such that for every $t\in E$ there exists an $n$ for which the rightmost inequality in \eqref{GammaZ1} is strict, then the inequality in \eqref{GammaZ2} is also strict.
	Finally, \eqref{EP1Cond} is invariant upon adding a purely 1-unrectifiable set to $Z$.
\end{proof}

\begin{corollary}\label{Cor5.4}
	Let $S\subset\R^2\times\{0\}$ be   closed  and suppose $\bs{\mu}$  is a magnetization carried by a Borel set $Z\subset S$ that is contained in a purely 1-unrectifiable set plus a countable union $ \bigcup_{k\in K}L_k$ where  the $L_k$ are disjoint line segments such that the distance from any $L_k$ to any $L_j$, $j\ne k$, is greater than or equal to the length of $L_k$.   
	Then \eqref{EP1Cond} holds for any rectifiable Jordan curve $\Gamma$, and thus  $ \bs{\mu} $ is $TV$-minimal on $S$.  
	Moreover, if the distance from any $L_k$ to any $L_j$, $j\ne k$, is strictly greater than the length of $L_k$, then \eqref{EP1Cond} is strict and $ \bs{\mu} $ is strictly $TV$-minimal on $S$.
\end{corollary}
\begin{proof}
	By the last assertion of Theorem \ref{Thm6.3}, it is enough to assume $Z$ is contained in a countable union of line segments with the aforementioned properties.
	 Let $\Gamma$ be a  rectifiable Jordan curve   oriented by a parametrization ${\bs \gamma}$.  Without loss of generality 
	we may assume that $Z\cap L_k\neq \emptyset$ for all $k\in K$.  If $K=\{1\}$ is a singleton, then (since $L_1$ is a line segment)
	$$\mathcal{H}^1(\Gamma \cap Z)\le  \mathcal{H}^1(\Gamma \cap L_1)<
	\mathcal{H}^1(\Gamma \setminus L_1)\le\mathcal{H}^1(\Gamma \setminus Z).
	$$
	Otherwise, for each $k\in K$ there is some directed sub-arc $\Gamma_k\subset \Gamma$ with initial point in $L_k$,  end  point in  some $L_j$
	for $j\neq k$, and interior in the complement of $\bigcup_{\ell\neq k}L_\ell$. 
	Note that for $j\neq k\in K$, the interiors of $\Gamma_k$ and $\Gamma_j$ are disjoint, 
	and  that $\mathcal{H}^1(\Gamma \cap L_k)\le \mathcal{H}^1(\Gamma_k)$ by assumption. Also note that this inequality is strict under the final assumption. Thus,
	$$\mathcal{H}^1(\Gamma \cap Z)\le  \sum_{k\in K}\mathcal{H}^1(\Gamma \cap L_k)
	 \le \sum_{k\in K}\mathcal{H}^1( \Gamma_k)\le\mathcal{H}^1(\Gamma \setminus Z),
	$$
	where the second inequality is strict under the last assumption.
\end{proof}

We next characterize the space of $S$-silent 
magnetizations when $S$ contains only a finite number of  Jordan curves.
First we consider the class of closed $S\subset \R^2$ that contain no {\em rectifiable} Jordan curve at 
all, and hence, cannot hold nontrivial silent magnetizations.  We call such  $S$   {\em tree-like}.  Note that any closed purely 1-unrectifiable set is tree-like, but  the converse is not true.   We also note that a tree-like set may contain a Jordan curve, such as the Koch curve, which is not rectifiable.   
As a consequence of Theorem \ref{rep} we obtain the following result.

\begin{lemma} \label{lem:treelike}
Let $S$ be a closed subset of $\R^2\times\{0\}$.  If
$\bs{\mu}\in \mathcal{M}(S)^3$ is nonzero and $S$-silent, then  the support of 
$\bs{\mu}$ contains a rectifiable Jordan curve.  Hence, if $S$ is tree-like the only $S$-silent magnetization is the zero magnetization. 
\end{lemma}
\begin{proof}
	Since $S\subset\R^2\times\{0\}$, it is slender and hence $S$-silent magnetizations are divergence free.
	The lemma now follows from Theorem \ref{rep}.
\end{proof}

For a closed set $S\subset \R^2\times\{0\}$, let $\Sigma(S)$ denote the linear subspace of 
  $\mathcal{M}(S)^3$ consisting of $S$-silent sources.   The previous lemma shows that 
  $\Sigma(S)$ is the trivial subspace when $S$ is tree-like.   
  The next theorem provides sufficient conditions that  $\Sigma(S)$ is finite dimensional and generalizes the second assertion of Lemma~\ref{lem:treelike} when $\mathcal{H}^1(S)$ is finite. 

\begin{theorem}\label{thm:SigS}
Let $S\subset \R^2\times\{0\}$ be closed with empty interior.  If the number $n$ of bounded connected components of $\R^2\times\{0\}\setminus S$ is finite, then
the dimension of $\Sigma(S)$ is less than or equal to $n$.
Furthermore, the dimension is equal to $n$ if $\mathcal{H}^1(S)$ is finite.
\end{theorem}
\begin{proof}
	Let $S'\subset S$ be the union of all rectifiable Jordan curves contained in $S$ and let $m$ be the number of bounded connected components of $\R^2\setminus S'$.
	Since $(\R^2\setminus S) \cup (S\setminus S') = (\R^2\setminus S')$ and the set $S\setminus S'$ is a subset of the topological boundary of $\R^2\setminus S$, then  $n\geq m$.
	From Theorem  \ref{rep} it follows that $\Sigma(S)=\Sigma(S')$,	thus showing that $\dim\Sigma(S')=m$ will prove our theorem.
	
	Let $\{E_i\}_{i=1}^m$ be the family of bounded connected components of $\R^2\setminus S'$. 
	Note that each $E_i$ is of finite perimeter since $\mathcal{H}^1(S')$ is finite.
	Let $\bs{\ell}_i:=\mathfrak R\grad\chi_{E_i}$ for  $i=1,...m$. 
	By Lemma~\ref{2Dsole} each  $\bs{\ell}_i$ is $S'$-silent.   
	To show that $\{\bs{\ell}_i\}_{i=1}^m$ generates $\Sigma(S')$, it is sufficient by 
  Theorem \ref{rep} to prove that for any   rectifiable Jordan curve $\Gamma\subset S'$ with arclength parametrization $\bs \gamma$,  the magnetization $\mathbf{R}_{\bs \gamma}$ defined by \eqref{Rgamma} is in the span of the  $ \bs{\ell}_i$'s.

	Using the Jordan curve theorem we can see that for any $E_i$ such that int$(\Gamma)\cap E_i\neq\emptyset$ we  have that $E_i\subset$ int$(\Gamma)$.
	Hence there exists a $J\subset\{1,...,m\}$ such that $\bigcup_{i\in J}E_i\subset$ int$(\Gamma)\subset S'\cup\bigcup_{i\in J}E_i$ and since $\mathcal L_2(S')=0$, then
	\begin{align*}
	\mathbf{R}_{\bs \gamma}&= \mathfrak R\grad\chi_{\text{int}(\Gamma)} 
	= \mathfrak R\grad\chi_{\bigcup_{i\in J}E_i}\\
	&	=\sum_{i\in J}\mathfrak R\grad\chi_{E_i}
	=\sum_{i\in J}\bs{\ell}_i, 
	\end{align*}	
	where the first equality comes from  the remark after \eqref{dermeasd}, equation \eqref{dermeas} and Lemma \ref{Gamma}.
	
	To show linearly independence, assume that $\sum_{i=1}^mc_i\bs{\ell}_i=0$ where $c_i\in\R$, $i=1,..,m$.
	Since $0=\sum_{i=1}^mc_i\mathfrak R\grad\chi_{E_i}=\mathfrak R\grad\left(\sum_{i=1}^mc_i\chi_{E_i}\right)$, thus $\sum_{i=1}^mc_i\chi_{E_i}$ is a constant but since the $E_i$'s are bounded and disjoint then each $c_i=0$ and hence the $\bs\ell_i$'s are indeed linearly independent.
	%
\end{proof}

\subsection{Regularization by penalizing the total variation}
\label{MtoFop}

Let $S\subset \R^2\times\{0\}$  and $Q\subset\R^3$ be closed and positively separated.
For $\bs{\mu}\in \mathcal{M}(S)^3$ and
$v$ a unit vector in $\R^3$,  the component  of the magnetic field
$\bb{b}(\bs{\mu})$ in the direction $v$ at $x\not \in S$ 
is given,  in view of  \eqref{bDef1}, by
\begin{equation}\label{b3K}
b_v(\bs{\mu})(x):=v\cdot\bb{b}(\bs{\mu})(x)  =-\frac{\mu_0}{4\pi}\int \bb K_v(x-y)\cdot \, d\bs{\mu}(y),
\end{equation}
where 
\begin{equation}
\label{noyauK}
\bb K_v(x)=\frac{v}{|x|^3}  -3  x\frac{v\cdot x}{|x|^5}=  \grad \left( \frac{v\cdot x}{|x|^3}\right).
\end{equation}
Consider  a finite, positive Borel measure  $\rho$
with support contained in $Q$  and let $A:
\mathcal{M}(S)^3\to L^2(Q,\rho)$ be the so-called {\it forward operator}
defined by
\begin{equation}\label{Adef}
A(\bs{\mu})(x):= b_v(\bs{\mu})(x), \qquad x\in Q.
\end{equation}  
 The adjoint operator   $A^*$ is then given by (see \cite[Section 3]{BVHNS})
\begin{equation}
\label{b3*Wbis}
A^*(\Psi)(x):=-\mu_0 \grad(\grad U^{\rho,\psi}\cdot v)(x),\qquad
U^{\rho,\psi}(x)=-\frac{1}{4\pi}\int \frac{\Psi(y)}{|x-y|}d\rho(y).
\end{equation}

Since $Q$ and $S$ are positively separated it follows from the harmonicity of $K_v$ that $A^*(\Psi)\in C_0(S)^3$ and thus $A^*:(L^2(Q,\rho))^*\sim L^2(Q,\rho)\to C_0(S)^3\subset(\mathcal{M}(S)^3)^*$. 
Note the kernel of the forward operator $A$ contains all $S$-silent magnetizations.
In the case this kernel consists exactly of $S$-silent magnetizations, we say that $A$ is $S${\it-sufficient}.
It follows from \cite[Lemmma~2.3]{BVHNS} and the discussion thereafter that $A$ is $S$-sufficient when $S\subset \R^2\times\{0\}$ and $Q\subset\R^3$ are positively separated closed sets and for some complete real analytic surface $\mathcal{A} \subset \R^3\setminus S$ we have:
	\begin{enumerate}
		\item $S$ and $\mathcal{A}$ are positively separated;
		\item $S$ lies entirely within one connected component of $\R^3\setminus \mathcal{A}$;
		\item $Q\cap \mathcal{A}$  has Hausdorff dimension strictly greater than 1 in  each connected component of $\R^3\setminus S$;
		\item $\supp{\rho}=Q$. 
	\end{enumerate}
For $\bs{\mu}\in\mathcal{M}(S)^3$, $f\in L^2(Q,\rho)$, and $\lambda>0$, 
recall from 
\eqref{defcrit0} the definition of $\mathcal{F}_{f,\lambda}$,
and from \eqref{crit0} the notation  $\bs{\mu_\lambda}\in \mathcal{M}(S)^3$ to designate
a minimizer of  $\mathcal{F}_{f,\lambda}$.
As a second application of our results in Section \ref{loppdecSM}, we prove:

\begin{theorem} 
\label{uniqueness}
{Let $S$ be a closed subset of $\R^2\times\{0\}$,
	$Q\subset\R^3$ be a closed set and $\rho\in\mathcal{M}(Q)$ be such that the forward operator $A$ defined in \eqref{Adef} is $S$-sufficient. 
	For $f\in L^2(Q,\rho)$ and $\lambda>0$, the solution to \eqref{crit0} is unique.}
\end{theorem}

\begin{proof}
It is well known (see {\it e.g.} \cite[Propostion~3.6]{BrePikk}) that 
$\bs{\mu}_\lambda\in \mathcal{M}(S)^3$ is 
	a minimizer of $\mathcal{F}_{f,\lambda}$ if and only if:
	\begin{equation}
	\label{CP}
	\begin{array}{ll}
	A^*(f-A\bs{\mu}_\lambda)&=\frac{\lambda}{2} \bb{u}_{\bs{\mu_\lambda}} 
	\qquad |\bs{\mu}_\lambda|\text{\rm -a.e. and}\\
	\left |A^*(f-A\bs{\mu}_\lambda)\right|&\leq \frac{\lambda}{2}\quad\text{\rm everywhere on } S.
	\end{array}
	\end{equation}
	Moreover, it follows from the strict convexity of the $L^2$-norm 
that $\bs{\mu}'_\lambda\in\mathcal{M}(S)^3$ is another solution 
	if and only if $A(\bs{\mu}'_\lambda-\bs{\mu}_\lambda)=0$.

Assume for a contradiction that $\bs{\mu}_\lambda$ and $\bs{\mu}'_\lambda$ 	are two distinct minimizers in \eqref{crit0} and let $\bs{\mu}:=\bs{\mu}'_\lambda-\bs{\mu}_\lambda$. 
	As $\bs\mu'_\lambda-\bs{\mu}_\lambda=\bs{\mu}$ is absolutely continuous with respect to $|\bs{\mu}|$, the Lebesgue  decompositions of $\bs\mu_\lambda$ and $\bs\mu_\lambda'$ with respect to $|\bs{\mu}|$ must have the same singular term.
That is, these decompositions are necessarily of the form
	\begin{equation*}
	\label{RDderc}
	d\bs{\mu}_\lambda=\boldsymbol{\gamma} d|\bs{\mu}|+d\bs{\nu},\qquad
	d\bs{\mu}'_\lambda=\boldsymbol{\gamma}' d|\bs{\mu}|+d\bs{\nu},
	\end{equation*}
	where $|\bs{\nu}|$ is singular with respect to $|\bs{\mu}|$ and $\boldsymbol{\gamma}$, $\boldsymbol{\gamma}'$ are $|\bs{\mu}|$-integrable $\R^3$-valued functions.
	
	Put for simplicity $\psi=(2/\lambda)(f-A(\bs{\mu}_\lambda))=
	(2/\lambda)(f-A(\bs{\mu}'_\lambda))$.
	Thanks to \eqref{CP} we know that $\bb u_{\bs\mu_\lambda}= A^*\psi$ and $\bb u_{\bs\mu_\lambda'}=A^*\psi$, $\bs{\mu}_\lambda$ and $\bs{\mu}_\lambda'$-a.e. respectively.
	Now, since $d|\bs{\mu}_\lambda|=|\boldsymbol{\gamma}| d|\bs{\mu}| + d|\bs{\nu}|$ and $d|\bs{\mu}'_\lambda|=|\boldsymbol{\gamma}'| d|\bs{\mu}| + d|\bs{\nu}|$, we have that
	\begin{align*}
	\bb u_{\bs{\mu}} d|\bs{\mu}|= d\bs{\mu} = \bb u_{\bs\mu_\lambda'}d|\bs{\mu}'_\lambda|-\bb u_{\bs\mu_\lambda}d|\bs{\mu}_\lambda| = A^*\psi d|\bs{\mu}'_\lambda| - A^*\psi d|\bs{\mu}_\lambda| = A^*\psi (|\boldsymbol{\gamma}'|-|\boldsymbol{\gamma}|) d|\bs{\mu}|.
	\end{align*}
	Therefore $\bb u_{\bs{\mu}} = A^*\psi (|\boldsymbol{\gamma}'|-|\boldsymbol{\gamma}|)$ at $|\bs\mu|$-a.e point, and since $|A^*\psi|=1$ on the supports of $\bs\mu_\lambda$ and $\bs\mu_\lambda'$ it holds that  $\bb u_{\bs{\mu}}(x) = \pm_x A^*\psi(x)$ for $|\bs\mu|$-a.e. $x$, where the choice of sign $\pm_x$ 
has a subscript $x$ to indicate that it may vary with
$x$.
	
	From the $S$-sufficiency of $A$ we know that $\bs{\mu}$ is $S$-silent. 
	Also, by 
	\cite[Corollary~4.2]{BVHNS} (take $\mathcal{B}=\R^2\times \{0\}$ there), the supports of	$\bs{\mu}_\lambda$ and $\bs{\mu}'_\lambda$ are contained in a finite collection of points and analytic arcs. In particular, there are only finitely many rectifiable Jordan curves contained in the support of $\bs{\mu}$ and they are all piecewise analytic.
	Thus, applying Theorem \ref{rep} to $\bs{\mu}$, we find there are finitely many piecewise analytic oriented Jordan curves $\Gamma_1,\cdots,\Gamma_N$ 
	with respective unit tangent vector fields 	$\boldsymbol{\tau}_1,\cdots,\boldsymbol{\tau}_n$, and strictly positive real numbers $a_1,\cdots,a_N$ such that $\bs{\tau}_m=\bs\tau_n$ on $\Gamma_m\cap \Gamma_n$, $\mathcal{H}^1$-a.e. and
	$$
	d\bs{\mu}=\sum_{n=1}^N a_n\bs{\tau}_n d\left(\mathcal{H}^1\mathcal{b}	\Gamma_n\right).
	$$
	In particular, $d|\bs{\mu}|=\sum_{n=1}^N a_n d\left(\mathcal{H}^1\mathcal{b}	\Gamma_n\right)$ and $\bs{\tau}_n (x)= \bb u_{\bs{\mu}}(x) = \pm_x A^*\psi(x)$, for $|\bs{\mu}|$-a.e. $x$, hence $\mathcal{H}^1$-a.e., on $\Gamma_n$.
	
	Fix $n$ and let $E$ be an analytic sub-arc of $\Gamma_n$.
		Being the unit tangent to an oriented analytic arc, 
$\bs{\tau}_n(x)$ must be an analytic function of $x\in E$, and so is
$A^*\psi(x)$ by the real analyticity of $A^*\psi$, {\it cf.} \eqref{b3*Wbis}.
Hence, either $\bs{\tau}_n=A^*\psi$ or $\bs{\tau}_n=-A^*\psi$ everywhere on $E$.
	Therefore, $E$ is a subset of a trajectory of the autonomous differential equation $\dot{x}=A^*\psi(x)$. 
	Moreover, since $E$ is bounded and percursed at unit speed, the corresponding trajectory extends beyond the endpoints of $E$, and since two distinct trajectories cannot intersect we conclude that $\Gamma_n$ is smooth and constitutes a single, periodic trajectory. 
	This, however, is impossible because $A^*\psi$ is a gradient vector field, by \eqref{b3*Wbis}.
\end{proof}





When $S$ is planar and EP-1 has a unique solution, Theorem~4.3 from \cite{BVHNS} and Theorem~\ref{uniqueness} together imply the following corollary.

\begin{corollary} 
	\label{constree}
	Let $S\subset \R^2\times \{0\}$ be closed, the forward operator $A$ be $S$-sufficient, and $\bs{\mu}_0\in\mathcal{M}(S)^3$.  Set $f=A\bs \mu_0$ and, for $e\in L^2(Q,\rho)$, set $f_e:=f+e$.
	For $\lambda>0$, there is a unique  minimizer  $\bs{\mu}_{\lambda,e}$ of
	\eqref{defcrit0} where $f$ gets replaced by $f_e$.
	
	If $\|\bs{\mu}\|_{TV}>\|\bs{\mu}_0\|_{TV}$ for any magnetization $\bs{\mu}$ that is $S$-equivalent to $\bs{\mu}_0$, then  $\bs{\mu}_{\lambda,e}$ (resp. $|\bs{\mu}_{\lambda,e}|$)
	converges to $\bs{\mu}_0$ (resp. $|\bs{\mu}_0|$) in the narrow sense as $\lambda\to0$ and $\|e\|_{L^2(Q)}/\sqrt{\lambda}\to0$.
\end{corollary}

Theorems~\ref{Thm6.2} and~\ref{Thm6.3}, Corollary~\ref{Cor5.4}, and  Lemma~\ref{lem:treelike} give sufficient conditions for the uniqueness of solutions to EP-1.   Hence, if $\bs{\mu}_0\in\mathcal{M}(S)^3$ is carried by a set $Z\subset S\subset \R^2\times \{0\}$, then we may apply the above corollary under the following conditions:

\begin{itemize}
 \item[(a)] $\mathcal{H}^1(\Gamma \cap Z) < \mathcal{H}^1(\Gamma\setminus Z)$  for any rectifiable Jordan curve $\Gamma\subset S$, or
\item[(b)]   $ Z\subset W\cup \bigcup_{k\in K}L_k$ where $W\subset S$ is  purely 1-unrectifiable   and  the $L_k$ are disjoint line segments such that the distance from any $L_k$ to any $L_j$, $j\ne k$, is greater than  the length of $L_k$, or 
\item[(c)] $S$ is tree-like.
\end{itemize}
In particular, it follows from condition (b) that Corollary~\ref{constree} applies when $\bs{\mu}_0$ is carried by a countable collection of points and sufficiently separated line segments. 

We conclude with an example.
\begin{exa}
Let $v_0=v_4=(0,0)$, $v_1=(1,0)$, $v_2=(1,1)$, and $v_3=(0,1)$ denote the vertices of the unit square $[0,1]^2$ and let $\bs{\gamma}_i$ denote the arclength parametrization of the directed line segment from $v_i$ to $v_{i+1}$ for $i=0,1,2,3$.  Let $\bs{\mu}_0={\bf R}_{\bs{\gamma}_0}+{\bf R}_{\bs{\gamma}_2}$ and $\bs{\mu}_1=-{\bf R}_{\bs{\gamma}_1}-{\bf R}_{\bs{\gamma}_3}$ and let  $S$ be any closed set that contains the unit square (e.g. $S=\R^2$).  By Corollary~\ref{Cor5.4}  both $\bs{\mu}_0$ and $\bs{\mu}_1$ are $TV$-minimal on $S$. However,  $\bs{\mu}_0$ and $\bs{\mu}_1$ are not strictly $TV$-minimal since $\bs{\mu}_0-\bs{\mu}_1$ is the loop around $[0,1]^2$, showing that  $\bs{\mu}_0$ and $\bs{\mu}_1$ are $S$-equivalent.  
Clearly, any convex combination  $(1-\alpha)\bs{\mu}_0+\alpha \bs{\mu}_1$, $\alpha\in [0,1]$, is also $S$-equivalent to $\bs{\mu}_0$ and $TV$-minimal on $S$.   
In fact, any $TV$-minimal magnetization is of this form.
	Indeed, taking $\bs{\mu}=\bs{\mu}_0$ and $Z=\supp\bs{\mu}_0$, in \eqref{EP1Cond}, the only $\Gamma$ that makes this inequality an equality is the boundary of $[0,1]$. 
	Hence, by Theorem \ref{Thm6.3}, any $TV$-minimal magnetization is of the form $\bs{\mu}_0+s(\bs{\mu}_1-\bs{\mu}_0)$ for some $s\in\R$.
	Then minimality of the total variation forces $0\leq s\leq 1$.
	
	If we take $Q=[0,1]^2\times\{1\}$ and $\rho=\mathcal{L}_2\mathcal{b}Q$ then the forward operator $A$ is $S$-sufficient. 
	With the notation of Corollary \ref{constree}, we get
since $\mathfrak{R}\bs{\mu}_0=\bs{\mu}_1$ that if $e=\mathfrak{R}e$ then $\mathfrak{R}f_e=f_e$.
	In this case, we get from Theorem \ref{uniqueness}
that $\mathfrak{R}\bs{\mu}_{\lambda,e}=\bs{\mu}_{\lambda,e}$
for every $\lambda>0$.
	Now, we know that any weak-$*$ limit of minimizers of EP-2 is $TV$-minimal, provided that both $\lambda$ and $\|e\lambda^{-1/2}\|_{L^2(Q,\rho)}$ 
	tend to $0$ (see \cite[Theorems~2\&5]{BurOsh}).
	Because the limit should also be invariant under $\mathfrak{R}$, it must be equal to $(\bs\mu_0+\bs\mu_1)/2$.
	In particular, we get global weak-$*$ convergence 
of $\bs{\mu}_{\lambda,e}$ and $|\bs{\mu}_{\lambda,e}|$ for this example, as long as the noise $e$ has the same symmetry as the data.
\end{exa}

\appendix
\section{}
In this appendix we gather several technical results (particularly Lemma~\ref{foncset}) concerning the Smirnov decomposition that are needed in Section~\ref{sparse3DSec}.
\label{sec:appendix}
\begin{lemma}
\label{RNcurve}
Let $\bs{\gamma}:[a,b]\to\R^n$ be a parametrized rectifiable curve, 
$\bs{\Gamma}=\bs{\gamma}([a,b])$
its image and 
$\bb{R}_{\bs{\gamma}}$ the $\R^n$-valued measure defined by \eqref{Rgamma}. 
Then, $\bb{R}_{\bs\gamma}$ is absolutely continuous with respect to 
$\mathcal{H}^1\mathcal{b}\bs{\Gamma}$, and its Radon-Nikodym derivative 
is given by
\[d\mathbf{R}_{\bs{\gamma}}/d(\mathcal{H}^1\mathcal{b}\Gamma)(x)=\sum_{t\in\bs{\gamma}^{-1}(x)} \bs{\gamma}^\prime(t),\qquad \mathcal{H}^1\text{-a.e.  }x\in\bs{\Gamma}.\]
\end{lemma}
\begin{proof}
As  $|\mathbf{R}_{\bs{\gamma}}|$ is regular 
(being a finite Borel measure on $\R^n$), for any open set $V\subset \R^n$ we 
have that 
\begin{equation}
\label{initac}
|\mathbf{R}_{\bs{\gamma}}|(V)=\sup\{|\langle \mathbf{R}_{\bs{\gamma}},\bs{\varphi}\rangle|,\, \bs{\varphi}\in C_c(V,\R^n),\,|\bs{\varphi}|\leq1\}
\leq\int_{\Gamma\cap V}N(\bs{\gamma},x)\,d\mathcal{H}^1(x).
\end{equation}
Now, $\mathcal{H}^1\mathcal{b}\Gamma$ is also regular, 
since  it is finite and
every open set in $\Gamma$ is $\sigma$-compact, see
\cite[Theorem~2.18]{Rudinrca}. In particular, if
$B\subset\R^n$ is a Borel set such that $\mathcal{H}^1(B\cap\Gamma)=0$, then
there is a decreasing 
sequence $V_k$ of open sets in $\R^n$ with
$V_k\supset B\cap\Gamma$ and $\mathcal{H}^1(\cap_k V_k\cap\Gamma)=0$.
Hence, we obtain from  \eqref{initac}, \eqref{areas} and the dominated convergence theorem  that
\[
|\mathbf{R}_{\bs{\gamma}}|(B)=|\mathbf{R}_{\bs{\gamma}}|(B\cap\Gamma)
\leq \liminf_k |\mathbf{R}_{\bs{\gamma}}|(V_k)\leq
\lim_k\int_{\Gamma\cap V_k}N(\bs{\gamma},x)\,d\mathcal{H}^1(x)=0.
\]
Thus, $|\mathbf{R}_{\bs{\gamma}}|$ and {\it a fortiori} $\mathbf{R}_{\bs{\gamma}}$ are absolutely continuous with 
respect to $\mathcal{H}^1\mathcal{b}\Gamma$.
Next, it holds 
for any Borel set $B\subset\R^n$ that the characteristic 
function  ${\chi_B}_{|\Gamma}$ is  the bounded pointwise limit 
$\mathcal{H}^1\mathcal{b}\Gamma$-a.e. (and thus 
$|\mathbf{R}_{\bs{\gamma}}|$-a.e. by what precedes) of a 
sequence of continuous functions $g_k:\Gamma\to\R$, by Lusin's theorem. 
Since $g_k$ 
is the restriction to $\Gamma$ of
some $f_k\in C_c(\R^n)$ with $\sup|f_k|=\sup|g_k|$
by the Tietze extension theorem (for $\Gamma$ is compact), 
we get from \eqref{Rgamma} that for any $v\in\R^n$
\[\langle \mathbf{R}_{\bs{\gamma}},f_k v\rangle= 
v\cdot\int_{\Gamma}f_k\left( \sum_{t\in\bs{\gamma}^{-1}(x)}\bs{\gamma}^\prime(t) \right)d\mathcal{H}^1(x)
\]
and, applying the dominated convergence theorem to both sides when $k\to\infty$,
we conclude since $v$ was arbitrary that  
\[
\mathbf{R}_{\bs{\gamma}}(B) 
=
\int_{\Gamma\cap B}\left( \sum_{t\in\bs{\gamma}^{-1}(x)}
\bs{\gamma}^\prime(t) \right)d \mathcal{H}^1(x).
\]
\end{proof}

\begin{lemma}
\label{sCell}
Let $\bs{\gamma}:[a,b]\to\R^n$ be a unit speed parametrization, 
$\bs{\Gamma}=\bs{\gamma}([a,b])$
its image and 
$\bb{R}_{\bs{\gamma}}$ the $\R^n$-valued measure defined by \eqref{Rgamma}. 
Then, $\|\bb{R}_{\bs{\gamma}}\|_{TV}=\ell(\bs{\gamma})$
if and only if, for $\mathcal{H}^1$-a.e. $x\in\bs{\Gamma}$, we have that
$\bs{\gamma}^\prime(t)$ is independent of $t\in\bs{\gamma}^{-1}(x)$.
\end{lemma}
\begin{proof}
If $\|\bb{R}_{\bs{\gamma}}\|_{TV}=\ell(\bs{\gamma})$,
there is a sequence of  continuous functions 
$\bb{g}_k\in C_c(\R^n,\R^n)$,
with $|\bb{g}_k|\leq 1$, such that
\begin{equation}
\label{calCell}
\ell(\bs{\gamma})=\lim_{k\to\infty}\langle \mathbf{R}_{\bs{\gamma}},\bb{g}_k\rangle=
\lim_{k\to\infty}\int_{\Gamma}\left( \sum_{t\in\bs{\gamma}^{-1}(x)}\bb{g}_k(x)\cdot\bs{\gamma}^\prime(t)\right) \,d\mathcal{H}^1(x).
\end{equation}
As $|\bb{g}_k(\bs{\gamma}(t))|\leq1=|\bs{\gamma}^\prime(t)|$,
we see from \eqref{areas}, \eqref{calCell}
 and the definition of $N(\bs{\gamma},x)$ that
for some subsequence $j(k)$ and $\mathcal{H}^1$-a.e. $x\in\Gamma$, 
we have $\lim_{k}\bb{g}_{j(k)}(x)\cdot\bs{\gamma}^\prime(t)=1$ for
all $t$ such that $\bs{\gamma}(t)=x$. In particular, 
$\bs{\gamma}^\prime(t)$ is independent of
$t\in\bs{\gamma}^{-1}(x)$ for $\mathcal{H}^1$-a.e. $x$. Conversely, 
if the latter property hold, 
we get from \eqref{Rgammappv} and \eqref{areas} that 
$|\bb{R}_\gamma|(\R^n)=\ell(\bs{\gamma})$.
\end{proof}
\begin{lemma}
\label{foncset}
Let $\bs{\mu}\in\mathcal{M}(\R^n)^n$ and $\rho$ be a finite positive 
Borel measure on $\mathcal{C}_\ell$ for some
$\ell>0$. Then, \eqref{Smiw} holds if and only if
\eqref{Smi1} does.
\end{lemma}
\begin{proof}
Assume that \eqref{Smiw} holds, and let $V\subset\R^n$ be open. Let 
$\varphi_k\in C_c(V)$ be a sequence 
of  nonnegative functions increasing to $\chi_V$;   such a sequence is easily 
constructed using Urysohn's lemma and the
$\sigma$-compactness of $V$. Applying the second identity in \eqref{Smiw}
to $\varphi_k$, we get by monotone convergence that
\begin{equation}\label{Smimo}
|\bs{\mu}|(V)=\lim_{k\to+\infty}\langle|\bs{\mu}|,\varphi_k\rangle=
\lim_{k\to+\infty}\int\langle|\bb{R}_{\bs{\gamma}}|,\varphi_k\rangle
d\rho(\mathbf{R}_{\bs\gamma})=\int|\bb{R}_{\bs{\gamma}}|(V)d\rho(\mathbf{R}_{\bs\gamma}).
\end{equation}
Hence, $|\bs{\mu}|$ and $\int|\bb{R}_{\bs{\gamma}}|d\rho$ coincide on 
open sets.
In particular, we get for $V=\R^n$ that
\begin{equation}\label{Smim}
\|\bs{\mu}\|_{TV}=\int_{\mathcal{C}_\ell}\|\mathbf{R}_{\bs\gamma}\|_{TV}
d\rho(\mathbf{R}_{\bs\gamma}).
\end{equation}
Moreover, as $|\bs{\mu}|$ is
regular, we see from \eqref{Smimo} that for any Borel 
set $B\subset\R^n$:
\begin{equation}\label{Smio}
|\bs{\mu}|(B)=\inf\{|\bs{\mu}|(V), \ B\subset V \,\mbox{\rm open}\}
=\inf_V \,\int_{\mathcal{C}_\ell}|\mathbf{R}_{\bs\gamma}|(V)
d\rho(\mathbf{R}_{\bs\gamma})\geq
\int_{\mathcal{C}_\ell}|\mathbf{R}_{\bs\gamma}|(B)
d\rho(\mathbf{R}_{\bs\gamma}).
\end{equation}
The conjunction of \eqref{Smim} and \eqref{Smio} implies the 
second equality  in
\eqref{Smi1}. 

To obtain the first equality in \eqref{Smi1}, apply Lusin's theorem to the effect
that  $\chi_B$ is the bounded pointwise limit of
a sequence $f_k\in C_c(\R^n)$, except on a Borel set $E$ of 
$|\bs{\mu}|$-measure  zero. From the second equality in \eqref{Smi1}, 
it follows that
$|\mathbf{R}_{\bs{\gamma}}|(E)=0$ for $\rho$-a.e. $\mathbf{R}_{\bs{\gamma}}\in\mathcal{C}_\ell$. Thus,
if we set $\mathbf{R}_{\bs{\gamma}}=(m_1,\cdots,m_n)^T$ to indicate the 
components of $\mathbf{R}_{\bs{\gamma}}$ in $\mathcal{M}(\R^n)^n$, 
we get {\it a fortiori} 
that $|m_j|(E)=0$ for $\rho$-a.e. $\mathbf{R}_{\bs{\gamma}}$.
So, picking $v=(v_1,\cdots,v_n)^T\in\R^n$,
we deduce for such $\mathbf{R}_{\bs{\gamma}}$ 
on applying the dominated convergence theorem component-wise
that
\begin{equation}
\label{limsi}
\lim_k \langle \mathbf{R}_{\bs{\gamma}},f_kv\rangle
=\sum_{j=1}^n v_j\lim_k 
 \int f_kd m_j=\\\sum_{j=1}^n v_j
 \int \chi_Bd m_j=v\cdot \mathbf{R}_{\bs{\gamma}}(B).
\end{equation}
Since $v$ was arbitrary, we can now show 
the first equality in \eqref{Smi1} from the first equation in \eqref{Smiw},
applied with $\bb{g}=f_k v$, by invoking 
the dominated convergence theorem when $k\to\infty$,
in $L^1[d|\bs{\mu}|]$ on the left hand side and 
in $L^1[d|{\rho}|]$ on the right hand side.

Conversely, if \eqref{Smi1} holds, sets of
$|\bs{\mu}|$-measure zero have $|\bb{R}_{\bs{\gamma}}|$-measure zero for
 $\rho$-a.e. $\bb{R}_{\bs{\gamma}}$, moreover
$|\bs{\mu}|$ and $\int|\bb{R}_{\bs{\gamma}}|d\rho$
(resp. $\bs{\mu}$ and $\int\bb{R}_{\bs{\gamma}}d\rho$) have the same integral on simple functions, hence also on $L^1[d|\bs{\mu}|]$
(resp. $(L^1[d|\bs{\mu}|])^n$). This is logically stronger than \eqref{Smiw}.
\end{proof}

\begin{lemma}
\label{unitels}
Let $\bb{T}_{\bb{f}}$ be an elementary solenoid as in
 \eqref{SmirnovA}.
Then, there is a Lipschitz map $\bb{g}:\R\to\R^n$ with
$|\bb{g}^\prime(t)|=1$ a.e. such that $\bb{T}_{\bb{g}}$ is an elementary solenoid with
$\bb{T}_{\bb{f}}=\bb{T}_{\bb{g}}$.
\end{lemma}
\begin{proof}
Recall from Section \ref{sec:decss} that 
$\bb{T}=*\lim\,\mathbf{R}_{\bb{f}_s}/s$ as $s\to+\infty$, 
where we have set $\bb{f}_s=\bb{f}_{|[-s,s]}$.
As  the $TV$-norm of the weak-$*$
limit cannot exceed the limit of the $TV$-norms, we get
since $|\bb{f}'(t)|\leq1$ that
\begin{equation}
\label{ninf}
1=\|\bb{T}_{\bb{f}}\|_{TV}\leq \liminf_{s\to+\infty}
\frac{1}{2s}\|\mathbf{R}_{\bb{f}_s}\|_{TV}\leq
\liminf_{s\to+\infty}
\frac{1}{2s}\int_{-s}^s|\bb{f}'(t)|dt\leq1.
\end{equation}
Thus,  $\frac{1}{2s}\int_{-s}^s|\bb{f}'(t)|dt\to1$
as $s\to+\infty$ and therefore, reparametrizing $\bb{f}$ by
unit speed like we did for $\bs{\gamma}$ 
after \eqref{Rgamma}, we obtain the desired function $\bb{g}$.
\end{proof}


\begin{lemma}
\label{desces}
Let $\bb{T}_{\bb{f}}$ be an elementary solenoid as in
 \eqref{SmirnovA} and $\Gamma_s=\bb{f}([-s,s])$. 
Then, the family $\{\nu_s\}_{s>0}$ of normalized
arclengths on $\Gamma_s$,
defined in \eqref{prob}, converges weak-$*$, when
$s\to+\infty$, to the probability measure $|\bb{T}_{\bb{f}}|$.
Moreover, if $\bs{\varphi}_j\in C_c(\R^n,\R^n)$  is a sequence of continuous 
functions, with $|\bs{\varphi}_j|\leq1$,
such that $\langle \bb{T}_{\bb{f}},\bs{\varphi}_j\rangle\to1$ as $j\to\infty$,
then
\begin{equation}
\label{limtan}
\lim_{j\to\infty}\,\limsup_{s\to+\infty}\int\left|\bs{\varphi}_j(x)-
\frac{\sum_{t\in\bb{f}^{-1}(x),\,|t|\leq s}
\bb{f}^\prime(t) }{ N(\bb{f},x,s)}\right|^2d {\nu}_s=0.
\end{equation}
\end{lemma}
\begin{proof}
The family  $\{ {\nu}_s\}_{s>0}$ has at least
one weak-$*$ accumulation point as $s\to+\infty$, say $ {\nu}$.
Let $s_k$ be a sequence of positive real numbers tending to 
$+\infty$ and such that
$\nu_{s_k}$ converges weak-$*$ to $\nu$. For $V\subset\R^n$ an open set, we get 
by \eqref{rouv} that
\begin{multline*}
|\bb{T}_{\bb{f}}|(V)=
\sup\{\langle\bb{T}_{\bb{f}},\bs{\varphi}\rangle,\,\bs{\varphi}\in C_c(V,\R^n),\,
|\bs{\varphi}|\leq1\}\\
=\sup_{\bs{\varphi}}\,\lim_{s\to+\infty}
\int_{\Gamma_{s}}\bs{\varphi}(x) \cdot\frac{\left( \sum_{t\in\bb{f}^{-1}(x),\,|t|\leq s}
\bb{f}^\prime(t) \right)}{2s}d\mathcal{H}^1(x)\\
\leq\sup_{\bs{\varphi}}\,\liminf_{k\to\infty}
\int_{\Gamma_{s_k}}|\bs{\varphi}(x)| \frac{\left| \sum_{t\in\bb{f}^{-1}(x),\,|t|\leq s_k}
\bb{f}^\prime(t) \right|}{2s_k}d\mathcal{H}^1(x)\\
\leq\sup_{\bs{\varphi}}\,\lim_{k\to\infty}
\int_{\Gamma_{s_k}}|\bs{\varphi}(x)| \frac{
N(\bb{f},x,s_k)}{2s_k}d\mathcal{H}^1(x)
=\sup_{\bs{\varphi}}\,\langle\nu, |\bs{\varphi}|\rangle\leq\nu(V).
\end{multline*}
Thus, by regularity, $|\bb{T}_{\bb{f}}|(B)\leq\nu(B)$ for any Borel set 
$B\subset\R^n$, and since $|\bb{T}_{\bb{f}}|$ is a probability measure 
(by definition of an elementary solenoid) while $\|\nu\|_{TV}\leq1$ by
the Banach-Alaoglu theorem, we conclude that
$|\bb{T}_{\bb{f}}|=\nu$. This proves the first assertion.

Next, if $\bs{\varphi}\in C_c(\R^n,\R^n)$, $|\bs{\varphi}|\leq1$,
is such that $\langle \bb{T}_{\bb{f}},\bs{\varphi}\rangle>1-\varepsilon$ 
for some $\varepsilon\in(0,1)$, then it follows from \eqref{Rgammappv}
and the definition of $\bb{T}_{\bb{f}}$ that for $s>s_0=s_0(\bs{\varphi})$ 
large enough:
\[
1-\varepsilon<\int_{\Gamma_{s}}\bs{\varphi}(x) \cdot\frac{\left( \sum_{t\in\bb{f}^{-1}(x),\,|t|\leq s}
\bb{f}^\prime(t) \right)}{2s}d\mathcal{H}^1(x)=
\int\bs{\varphi}(x) \cdot\frac{\left( \sum_{t\in\bb{f}^{-1}(x),\,|t|\leq s}
\bb{f}^\prime(t) \right)}{N(\bb{f},x,s)}d\nu_s(x).
\]
Because $| \sum_{t\in\bb{f}^{-1}(x),\,|t|\leq s}
\bb{f}^\prime(t)|\leq N(\bb{f},x,s)$, the above inequality entails
that
\[
\int\left|\bs{\varphi}(x)-
\frac{\sum_{t\in\bb{f}^{-1}(x),\,|t|\leq s}
\bb{f}^\prime(t) }{ N(\bb{f},x,s)}\right|^2d {\nu}_s<
2\varepsilon,
\]
which implies \eqref{limtan}.
\end{proof}

 \bibliographystyle{abbrv}

\bibliography{Planarrefs}

\begin{thebibliography}{10}

\bibitem{Ambetal}
L.~Ambrosio, V.~Caselles, S.~Masnou, and J.-M. Morel.
\newblock Connected components of sets of finite perimeter and applications to
  image processing.
\newblock {\em J. Eur. Math. Soc. (JEMS)}, 3(1):39--92, 2001.

\bibitem{AFP}
L.~Ambrosio, N.~Fusco, and D.~Pallara.
\newblock {\em Functions of Bounded Variation and Free Discontinuity Problems}.
\newblock Oxford Mathematical Monographs. Oxford University Press, 2000.

\bibitem{Apostol1957}
T.~M. Apostol.
\newblock {\em Mathematical analysis ; a modern approach to advanced calculus}.
\newblock Reading, Mass., Addison-Wesley Pub. Co., 1957.

\bibitem{Attetal2006}
H.~Attouch, G.~Buttazzo, and G.~Michaille.
\newblock {\em Variational analysis in {S}obolev and {BV} spaces}, volume~6 of
  {\em MPS/SIAM Series on Optimization}.
\newblock Society for Industrial and Applied Mathematics (SIAM), Philadelphia,
  PA; Mathematical Programming Society (MPS), Philadelphia, PA, 2006.
\newblock Applications to PDEs and optimization.

\bibitem{IMP}
L.~Baratchart, D.~Hardin, E.~Lima, E.~Saff, and B.~Weiss.
\newblock Characterizing kernels of operators related to thin-plate
  magnetizations via generalizations of hodge decompositions.
\newblock {\em Inverse Problems}, 29(1):015004, 2013.

\bibitem{BVHNS}
L.~Baratchart, C.~Villalobos~Guill{\'e}n, D.~P. Hardin, M.~C. Northington, and
  E.~B. Saff.
\newblock Inverse potential problems for divergence of measures with total
  variation regularization.
\newblock {\em Foundations of Computational Mathematics}, Nov 2019.

\bibitem{BG}
P.~Bonicatto and N.~A. Gusev.
\newblock On the structure of divergence-free measures on {$\mathbb R^2$}.
\newblock {\em {\rm arXiv:1912.10936}}, 2019.

\bibitem{BreCar}
K.~Bredies and M.~Carioni.
\newblock Sparsity of solutions for variational inverse problems with
  finite-dimensional data.
\newblock {\em Calc. Var.}, 59, 2020.

\bibitem{BrePikk}
K.~Bredies and H.~K. Pikkarainen.
\newblock Inverse problems in spaces of measures.
\newblock {\em ESAIM COCV}, 19:190--218, 2013.

\bibitem{BurOsh}
M.~Burger and S.~Osher.
\newblock Convergence rates of convex variational regularization.
\newblock {\em Inverse Problems}, 20:1411--1421, 2004.

\bibitem{DeGiorgi1}
E.~De~Giorgi.
\newblock {\em Complementi alla teoria della misura {$(n-1)$}-dimensionale in
  uno spazio {$n$}-dimensionale}.
\newblock Seminario di Matematica della Scuola Normale Superiore di Pisa,
  1960-61. Editrice Tecnico Scientifica, Pisa, 1961.

\bibitem{DeGiorgi2}
E.~De~Giorgi.
\newblock {\em Frontiere orientate di misura minima}.
\newblock Seminario di Matematica della Scuola Normale Superiore di Pisa,
  1960-61. Editrice Tecnico Scientifica, Pisa, 1961.

\bibitem{Dem2012}
F.~Demengel and G.~Demengel.
\newblock {\em Functional spaces for the theory of elliptic partial
  differential equations}.
\newblock Universitext. Springer, London; EDP Sciences, Les Ulis, 2012.
\newblock Translated from the 2007 French original by Reinie Ern\'e.

\bibitem{evans_gariepy_2015}
L.~C. Evans and R.~F. Gariepy.
\newblock {\em Measure Theory and Fine Properties of Functions}.
\newblock CRC Press, 2015.

\bibitem{Federer1}
H.~Federer.
\newblock The {G}auss-{G}reen theorem.
\newblock {\em Trans. Amer. Math. Soc.}, 58:44--76, 1945.

\bibitem{Federer2}
H.~Federer.
\newblock A note on the {G}auss-{G}reen theorem.
\newblock {\em Proc. Amer. Math. Soc.}, 9:447--451, 1958.

\bibitem{FedererBook}
H.~Federer.
\newblock {\em Geometric measure theory}.
\newblock Die Grundlehren der mathematischen Wissenschaften, Band 153.
  Springer-Verlag New York Inc., New York, 1969.

\bibitem{FouRau}
S.~Foucart and H.~Rauhut.
\newblock {\em A Mathematical Introduction to Compressive Sensing}.
\newblock Birkh\"auser, 2013.

\bibitem{Gerhards2015}
C.~Gerhards.
\newblock On the unique reconstruction of induced spherical magnetizations.
\newblock {\em Inverse Problems}, 32(1), 2015.

\bibitem{HoKaPoSch}
B.~Hoffmann, B.~Kaltenbacher, C.~P\"oschl, and O.~Scherzer.
\newblock A convergence rates result for {T}ikhonov regularization in {B}anach
  spaces with non-smooth operators.
\newblock {\em Inverse Problems}, 23:987--1010, 2007.

\bibitem{Jac1975}
J.~D. Jackson.
\newblock {\em Classical electrodynamics}.
\newblock John Wiley \& Sons, Inc., New York-London-Sydney, second edition,
  1975.

\bibitem{MirandaJr}
M.~M. Jr.
\newblock Functions of bounded variation on ”good” metric spaces.
\newblock {\em J.Math. Pures Appl.}, 82:975–1004, 2003.

\bibitem{KW}
J.~R. Kirtley and J.~P. Wikswo.
\newblock Scanning squid microscopy.
\newblock {\em Annu. Rev. Mater. Sci.}, 29:117--148, 1999.

\bibitem{LWBHS}
E.~A. Lima, B.~P. Weiss, L.~Baratchart, D.~P. Hardin, and E.~B. Saff.
\newblock Fast inversion of magnetic field maps of unidirectional planar
  geological magnetization.
\newblock {\em J. Geophys Res.}, 118:2723--2752, 2013.

\bibitem{Mattila}
P.~Mattila.
\newblock {\em Geometry of sets and measures in {E}uclidean spaces}, volume~44
  of {\em Cambridge Studies in Advanced Mathematics}.
\newblock Cambridge University Press, Cambridge, 1995.
\newblock Fractals and rectifiability.

\bibitem{Rudinrca}
W.~Rudin.
\newblock {\em Real and Complex Analysis}.
\newblock Mc Graw-Hill, 1986.

\bibitem{Schwartz}
L.~Schwartz.
\newblock {\em Th\'eorie des distributions. {T}ome {I}}.
\newblock Actualit\'es Sci. Ind., no. 1091 = Publ. Inst. Math. Univ. Strasbourg
  9. Hermann \& Cie., Paris, 1950.

\bibitem{Smi94}
S.~K. Smirnov.
\newblock Decomposition of solenoidal vector charges into elementary solenoids,
  and the structure of normal one-dimensional flows.
\newblock {\em St. Petersburg Math. J.}, 5:841--867, 1994.

\bibitem{VillPhDThesis}
C.~Villalobos~Guill{\'e}n.
\newblock {\em A Measure Theoretic Approach for the Recovery of Remanent
  Magnetizations}.
\newblock PhD thesis, Vanderbilt University, 2019.

\bibitem{WLFB}
B.~P. Weiss, E.~A. Lima, L.~E. Fong, and F.~J. Baudenbacher.
\newblock Paleomagnetic analysis using squid microscopy.
\newblock {\em J. Geophys Res.}, 112(B9), 2007.

\bibitem{Ziemer}
W.~P. Ziemer.
\newblock {\em Weakly differentiable functions}, volume 120 of {\em Graduate
  Texts in Mathematics}.
\newblock Springer-Verlag, New York, 1989.
\newblock Sobolev spaces and functions of bounded variation.

\end{thebibliography}

\end{document}